\documentclass{article}

\usepackage[utf8]{inputenc} 
\usepackage[T1]{fontenc}    
\usepackage{hyperref}       
\usepackage{url}            
\usepackage{booktabs}       
\usepackage{amsfonts}       
\usepackage{nicefrac}       
\usepackage{microtype}      
\usepackage{enumerate}

\usepackage{microtype}
\usepackage{graphicx}
\usepackage{subcaption}
\usepackage{booktabs} 
\usepackage[numbers, compress]{natbib}

\graphicspath{{./figures/}}
\usepackage{hyperref}

\newcommand\numberthis{\addtocounter{equation}{1}\tag{\theequation}}
\usepackage{amsmath,amssymb,amsfonts}
\usepackage{mathtools}
\usepackage{amsthm}
\usepackage{algorithm}
\usepackage{algorithmic}

\usepackage[capitalize,noabbrev]{cleveref}

\usepackage{makecell}

\usepackage{float}
\usepackage{placeins}
\usepackage{multirow}
\theoremstyle{plain}
\newtheorem{theorem}{Theorem}[section]
\newtheorem{proposition}[theorem]{Proposition}
\newtheorem{lemma}[theorem]{Lemma}
\newtheorem{corollary}[theorem]{Corollary}
\theoremstyle{definition}
\newtheorem{definition}[theorem]{Definition}
\newtheorem{example}[theorem]{Example}
\newtheorem{assumption}{Assumption}
\theoremstyle{remark}
\newtheorem{remark}[theorem]{Remark}

\usepackage[textsize=tiny]{todonotes}
\usepackage[toc,page,header]{appendix}

\usepackage{url}            
\usepackage{booktabs}       
\usepackage{amsfonts}       
\usepackage{nicefrac}       
\usepackage{microtype}      

\usepackage{caption} 
\usepackage{wrapfig}

\usepackage{graphicx}

\usepackage{pifont}

\usepackage[most]{tcolorbox}

\usepackage{minitoc}

\newcommand{\R}{\mathbb{R}}

\newcommand{\xmark}{\ding{55}}
\newcommand{\cmark}{\ding{51}}%

\newcommand{\emphblockoption}{drop shadow,
    colframe=black!60,
    colback=black!10,
    coltitle=white!, 
    left=.2pt,
    right=.2pt,
    boxrule=0pt,
    arc=.5pt}
\tcbset{emphblock/.style={code={\pgfkeysalsofrom{\emphblockoption}}}}

\title{Efficient First-Order Optimization on the Pareto Set
for Multi-Objective Learning under Preference Guidance}
\date{}
{\author{Lisha Chen\footnotemark[1] \and 
Quan Xiao\footnotemark[1] \and
Ellen Hidemi Fukuda\footnotemark[2] 
\and Xinyi Chen\footnotemark[3] 
\and Kun Yuan\footnotemark[3] 
\and Tianyi Chen\footnotemark[1] \\
\hspace{-5cm} 
{\small \footnotemark[1] Rensselaer Polytechnic Institute
\footnotemark[2] Kyoto University
\footnotemark[3] Peking University}\\
}
}

\allowdisplaybreaks

\begin{document}

\maketitle

\begin{abstract}
Multi-objective learning under user-specified preference is common in real-world problems such as multi-lingual speech recognition under fairness. 
In this work, we frame such a problem as a semivectorial bilevel optimization problem, whose goal is to optimize a pre-defined preference function, subject to the constraint that the model parameters are weakly Pareto optimal. To solve this problem,  we convert the multi-objective constraints to a single-objective constraint through a merit function with an easy-to-evaluate gradient, and then, we use a penalty-based reformulation of the bilevel optimization problem. We theoretically establish the properties of the merit function, and the relations of solutions for the penalty reformulation and the constrained formulation. Then we propose algorithms to solve the reformulated single-level problem, and establish its convergence guarantees. We test the method on various synthetic and real-world problems. The results demonstrate the effectiveness of the proposed method in finding preference-guided optimal solutions to the multi-objective problem.

\end{abstract}

\section{Introduction}
Many machine learning tasks naturally involve multiple objectives, which may encompass diverse performance metrics such as accuracy, fairness, and privacy, or even the same metrics evaluated across different  datasets~\cite{sener2018multi}. A common approach to tackling such multi-objective problems is to learn a shared model that performs well across all objectives simultaneously. 
Compared to training separate models for each objective, this approach
offers significant benefits -- most notably, it reduces model size and inference time, making it more efficient and scalable.
Multi-objective optimization facilitates this by enabling the learning of models that minimize vector-valued objectives~\cite{miettinen_nonlinear_1998,ehrgott_multicriteria_2005}. In practical scenarios, it is often desirable to obtain solutions that provide controlled trade-offs or reflect specific preferences among competing objectives, 
rather than treating all objectives equally.

To further illustrate, we use one example on multi-lingual speech or language processing problem in Figure~\ref{fig:example_multi_lingual}.
The goal of this problem is to minimize multiple losses from different languages, while satisfying the user-specified preferences.
Preferences can control trade-offs among multiple losses and enhance steerability, enabling the solver to return diverse solutions on the Pareto front.
Analytically, the preferences can be defined as constraints or objectives, see, e.g.,~\cite{lin2019pareto,curtis2023fair,chen2024ferero}.
To prioritize finding the optimal solutions of the multi-lingual losses over satisfying the preferences, we model the preference as a secondary scalar-valued objective.
We first optimize the vector-valued objective formed by concatenating the multi-lingual losses, and then optimize the scalar-valued preference objective.
\begin{figure}
\centering
\includegraphics[width=.7\linewidth]{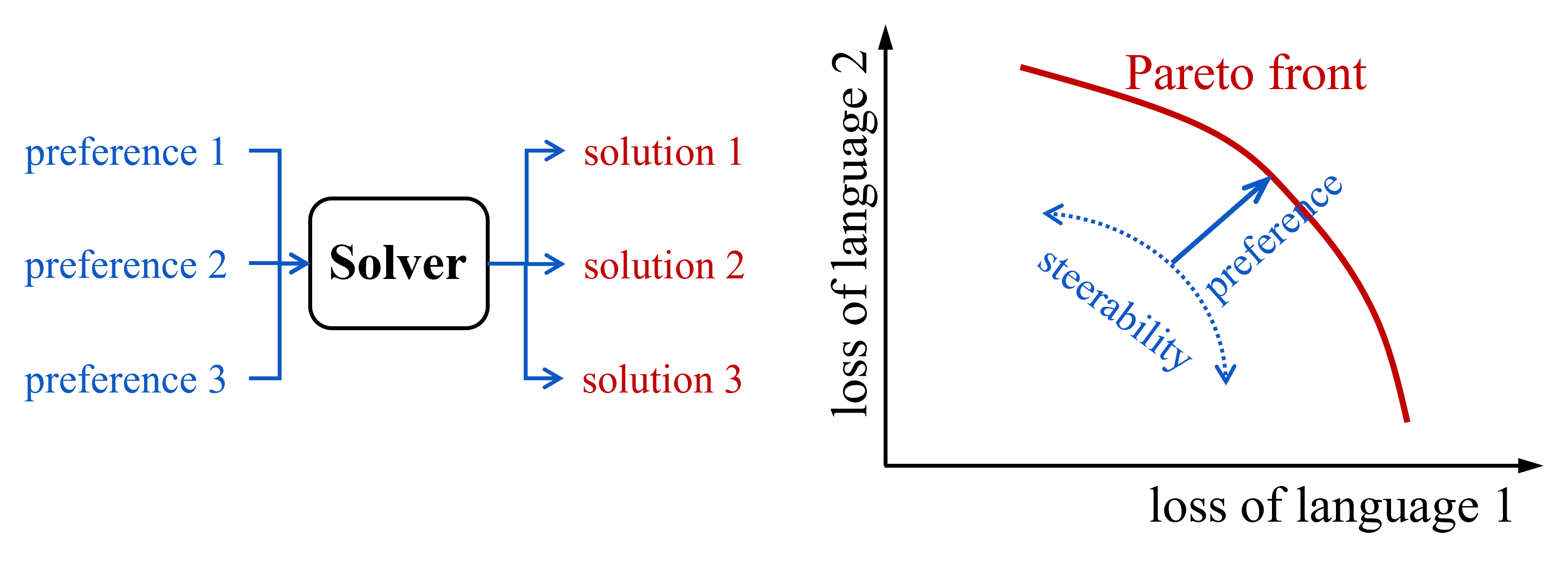} 
\caption{Example on multi-lingual speech or language processing problem under user-specified preference guidance.
The red curve represents the Pareto front, which is, informally, the set of objective values that achieve the best trade-offs among multiple objectives.}
\label{fig:example_multi_lingual}
\end{figure}
More formally, let $f_1,\dots,f_M: {\cal X}\to\mathbb{R}$ be the objective functions, and $f_0: {\cal X}\to \mathbb{R}$ be the preference function, with ${\cal X} \subseteq \R^q$ being nonempty, closed and convex.
Depending on the problem, $\cal X$ can be compact or $\R^q$.
Let $F=(f_1,\dots,f_M):{\cal X}\to\mathbb{R}^M$ be a vector-valued function. 
Then the optimization problem is 
\begin{equation}\label{eq:intro_bi_vector}
\min \nolimits_{x \in {\cal X}} f_0(x)
~~\mathrm{s.t.}~~x \in \mathop{\arg\min} \nolimits_{x \in {\cal X}}~ F(x)
\end{equation}
where for minimizing the vector-valued objective $F(x)$, we consider the widely used Pareto optimality, whose formal definition is deferred to Section~\ref{sec:preliminaries}. 
Then problem~\eqref{eq:intro_bi_vector} is also known as Optimization on the Pareto Set (OPS), or semivectorial simple bilevel optimization problem~\cite{Bolintineanu1993_Necessary_ops,Bolintineanu1993_Optimality_ops}.

The OPS problem is generally difficult to solve.
Existing studies usually make strong assumptions such as the convexity of the objectives $F$~\cite{roy2023opt_pareto_set}, or the algorithms require evaluating the second-order derivatives of the objectives, and is thus inefficient for large-scale problems~\cite{chen2022weighted_crosstask,roy2023opt_pareto_set}.
To address these challenges, 
we introduce reformulations of the OPS problem, 
establish relations of the reformulations and the original problem, 
and propose algorithms to solve the reformulated problem with convergence guarantees.

Our contributions can be summarized as follows:\\
First, we propose a smoothed merit function to convert the vector-valued optimization to a scalar-valued optimization, with easy-to-evaluate gradient, 
and prove that the smoothed merit function preserves the approximate weak Pareto optimality.
Then we use the smoothed merit function as a penalty function,
prove its error bound properties,
and establish the relation between the global/local/stationary solutions of the penalty-based reformulation and original problem.  
Based on the reformulation, we propose an efficient first-order algorithm with convergence guarantees.
Experiments are conducted on various synthetic and real datasets with possibly nonconvex objectives to demonstrate its effectiveness.

Technically, we address the following challenges:
\begin{itemize}
\item[T1] Existing merit functions for weak Pareto optimality have max-min non-smooth structures, and their subdifferential can be hard to evaluate. 
We propose a penalty function using a smoothed merit function, which admits easy-to-evaluate gradients.
\item[T2] Instead of directly assuming the merit function satisfies the error bound as in existing (simple) bilevel optimization literature, we prove the proposed merit function satisfies the desired error bound.
\item[T3] We establish the KKT condition for
a simple bilevel problem with provable constraint qualification condition under the Kurdyka-Łojasiewicz inequality. Based on this, we establish 
the relation of the stationarity solution of the penalty problem 
and the KKT solution to the simple bilevel problem.
\end{itemize}

\section{Problem setup and preliminaries} 
\label{sec:preliminaries}

For the optimization problem $\min_{x \in {\cal X}} F(x)$ in~\eqref{eq:intro_bi_vector}, 
we use the standard definitions for Pareto optimality~\cite{miettinen_nonlinear_1998}.
Given two vectors $v$ and $w$, we use $v<w$ and $v\leq w$ to denote $v_i < w_i$  and $v_i\leq w_i$ for all $i$, respectively. We use  $v\lneq w$ to denote $v\leq w$ and $v\neq w$, and define $>, \geq$, $\gneq$ analogously.
We use $\mathbf{1}_q$ to denote an all-one vector with dimension $q$, 
where $q$ is sometimes ommitted if it is clear from the context.
Then Pareto dominance and weak Pareto optimality are formally defined below.
\begin{definition}[Pareto dominance and optimality]
\label{def:generalized_dominance}
Given $v,w \in \R^M$, we say
$v$ strictly dominates $w$ if and only if  $v - w < 0$.
Correspondingly, a point $x\in {\cal X}$ is \emph{weakly Pareto optimal} if
 there is no $x' \in {\cal X}$ such that, $F(x') < F(x)$.
In addition, a point $x\in {\cal X}$ is \emph{$\epsilon$-weakly Pareto optimal} 
if there exists no $x'\in {\cal X}$ and $x' \neq x$ such that, $F(x') < F(x) - \epsilon \mathbf{1}$.
\end{definition}

Throughout the paper, we assume $f_m(x), m = 0, \ldots, M$ are proper and bounded below. And we assume they are twice continuously and directionally differentiable.
Denote the directional derivative of $f_m$ at point $x$ along direction $d$ 
as $f_m'(x; d)$, defined as
\begin{equation}
  f_m'(x; d) \coloneqq \lim_{\alpha\downarrow 0} 
  \frac{f_m(x + \alpha d) - f_m(x)}{\alpha}.
\end{equation} 
Then the Pareto stationarity is defined as follows.
\begin{definition}[Pareto stationarity~e.g.~\cite{ehrgott_multicriteria_2005}]
A point $x \in {\cal X}$ is Pareto stationary if 
$\max_{m\in [M]} f_m'(x; z - x) \geq 0$ for all $z \in {\cal X}$,
where $[M] = \{1, \ldots, M\}$.
\end{definition}

Denote $WP(F) \subseteq {\cal X}$ as the weak Pareto set of $F$, which contains all the weakly Pareto optimal solutions for $\min_{x \in {\cal X}} F(x)$. 
We consider the following problem
\begin{equation}\label{eq:pareto_constrained_opt}
  \min \nolimits_{x \in WP(F)}  ~~f_0(x) .
\end{equation}
We say that the solutions to \eqref{eq:pareto_constrained_opt} are \emph{preferred Pareto optimal}.

\section{Problem reformulation} 
\label{sec:problem_reformulation}

In this section, we first convert the original problem~\eqref{eq:pareto_constrained_opt} with multiple lower-level objectives to an equivalent problem with a single  lower-level objective using a merit function. Then we discuss its penalty-based reformulation.

\subsection{A smoothed merit function and its properties}
A merit function associated with the multi-objective optimization problem $\min_{x \in {\cal X}} F(x)$
is non-negative, and returns zero only at the weakly Pareto optimal solutions~\cite{Auslender1976,Hearn1982_gap_convex}. 
Under this requirement, assuming the lower semicontinuity of $F$,
then 
\begin{equation}
  \bar{u}(x) \coloneqq \sup_{y\in {\cal X}}\min_{m\in [M]}~~\{f_m(x)-f_m(y) \}  
\end{equation}
is a merit function in the sense of weak Pareto optimality~\citep[Theorem~3.1]{tanabe2018proximal}. 
In other words, $\bar{u}(x) = 0$ if and only if $x$ is weakly Pareto optimal.
Given this equivalence, it is desirable to convert the original lower-level multi-objective optimization problem to minimizing the scalar-valued function $\bar{u}(x)$.
However, $\bar{u}(x)$ in general can be non-differentiable due to its max-min structure, posing challenges to directly applying gradient-based approaches to minimize $\bar{u}(x)$.  
To address this challenge,
we propose the following smoothed and regularized merit function
$v_{l,\tau}(x)$ given $l \geq 0, \tau > 0$.
{\small\begin{subequations}
\begin{align}
\label{eq:hltau}
& h_{l,\tau}(x,y) \coloneqq \tau\ln \Big(\sum_{m=1}^M e^{\frac{f_m(y) - f_m(x)}{\tau} } \Big) + \frac{l}{2}\|x - y\|^2 \\
\label{eq:vl}
& v_{l,\tau}(x) \coloneqq 
- \min_{y\in {\cal X}} h_{l,\tau}(x,y) .
\end{align}
\end{subequations}}
Note that, $v_{l,\tau}$ can be seen as a smoothed and regularized function of $\bar{u}$. Specifically, when $l=0$, 
$v_{0,\tau}$ smoothes the maximization operation over $m \in [M]$ in $\bar{u}$ with the log-sum-exponential (LSE) function~\cite{Nesterov2005_smooth_min}.
And it uniformly converges to $\bar{u}$ as $\tau$ converges to zero.
Besides, adding the regularization with $l > 0$
can further lift a weakly convex objective to a strongly convex one, 
so that not only the minimization $\min_{y\in {\cal X}} h_{l,\tau}(x,y)$ in~\eqref{eq:vl} enjoys a unique solution, but also $v_{l,\tau}(x)$ is smooth and has easy-to-evaluate gradient.
To further illustrate, 
we provide a visualization of $\bar{u}$, $v_{l,\tau}$ on a simple example in Figure~\ref{fig:u_v_plot}. 
It shows that $v_{l,\tau}$ with smaller $\tau$ or $l$ approximates $\bar{u}$ better, while larger $\tau$ or $l$ makes $v_{l,\tau}$ smoother.
Too large $\tau$ or $l$ could possibly change the shape of 
$v_{l,\tau}$ significantly compared to $\bar{u}$.
\begin{figure}[ht]
\centering
\begin{subfigure}[b]{0.3\textwidth}
\centering
\includegraphics[width=.98\linewidth]{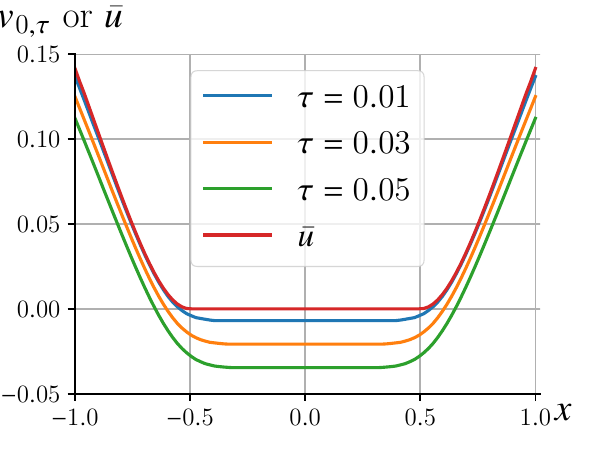}
\caption{$v_{0,\tau}$ and $\bar{u}$}
\label{sfig:u_v0}
\end{subfigure}  
\begin{subfigure}[b]{0.3\textwidth}
\centering
\includegraphics[width=.98\linewidth]{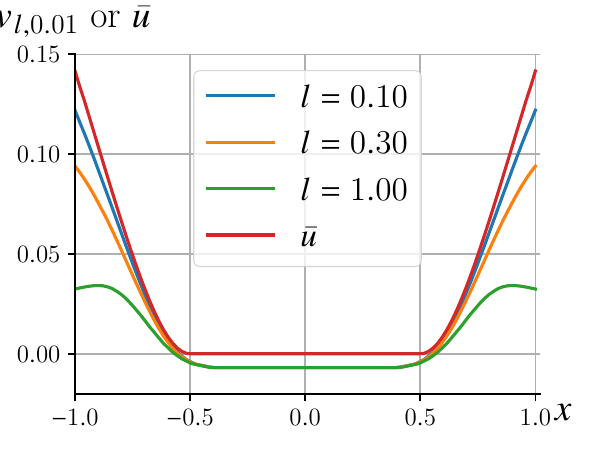}
\caption{$v_{l,0.01}$ and $\bar{u}$}
\label{sfig:u_vl}
\end{subfigure}  
\caption{Illustration of $v_{l,\tau}$ and $\bar{u}$ 
with different values of $\tau$ and $l$ with 
$F(x) = \Big( \sqrt[6]{(x + \frac{1}{2})^2 + \frac{1}{8}}, 
\sqrt[6]{(x - \frac{1}{2})^2 + \frac{1}{8} } \Big)^\top$.}
\label{fig:u_v_plot}
\end{figure}

Then we discuss the properties of the smoothed merit function and outline the procedure for computing its gradient. \\
\textbf{Properties of the smoothed merit function.} 
We establish how the value of the smoothed merit function changes 
when $x$ reaches the weak Pareto optimality condition.
In the prior work~\cite{tanabe2022new}, such properties have been
established for the merit functions $\bar{u}$ and $u_l$ (c.f. Appendix~\ref{sub_app:aux_property_v_po}).
We will show that the smoothed version $v_{l,\tau}$ 
can still preserve these properties approximately depending on the hyperparameters $l$ and $\tau$.
For our analysis, we make the following assumptions that are common in multi-objective optimization~\cite{fliege2019complexity,tanabe2018proximal,liu2021stochastic}.
Note that we do not require all the following assumptions to hold for all our results, which will be specified correspondingly.

\begin{assumption}
\label{assmp:local_lip_f}
For all $m\in\{0,\ldots, M\}$, $f_m(x)$ is locally Lipschitz on any bounded set
in ${\cal X}$.
\end{assumption}

\begin{definition}[Weak convexity]
A locally Lipschitz function $f: {\cal X}\to \R$ is 
$\mu$-weakly convex if 
$f(x) - \frac{\mu}{2}\|x\|^2$ is convex
for $x\in {\cal X}$.  
\end{definition}

\begin{assumption}
\label{assmp:weak_c_f}
For all $m \in [M]$, $ f_m(x)$ is 
locally Lipschitz and $\mu$-weakly convex on $\cal X$.
\end{assumption}

Before proceeding to the theoretical results, 
we introduce the following definition, which relaxes the commonly used 
convexity assumption of the objective functions. 
\begin{definition}[Point strong quasar-convex functions~{\citep[Definition~2.1]{hardt2018gd_linear_dy_quasar}}]
\label{def:quasar_convex}
A function $f: {\cal X} \to \R$ is $(c_q, \mu)$-point 
strong quasar-convex with $c_q \in (0, 1], \mu \geq 0$ 
at $x^* \in {\cal X}$ if for all  $x \in {\cal X}$,
{\small\begin{equation}
f(x^*) \geq f(x) + \frac{1}{c_q} \nabla f(x)^{\top}(x^*-x)
+\frac{\mu}{2}\|x^*-x\|^2 .
\end{equation}}
\vspace{-6mm}
\end{definition}
The point (strong) quasar-convexity in Definition~\ref{def:quasar_convex}
is a relaxation of the (strong) convexity.
When $c_q = 1$, the quasar-convexity  implies 
point star-convexity~\cite{lee2016optimizing_star}. And if the point star-convexity holds at all $x \in {\cal X}$, then it implies convexity.
There are many examples of nonconvex but point quasar-convex functions, 
see the discussions in, e.g.,~\cite{hinder2020near_star_quasar}.
In machine learning, typical examples that satisfy the quasar convexity 
include the linear dynamical systems identification~\cite{hardt2018gd_linear_dy_quasar} and generalized linear models with leaky ReLU or logistic activation functions~\cite{wang2023continuized_acc_quasar}.

Based on the above assumptions and definition,
we introduce the properties of the smoothed merit function 
$v_{l,\tau}$ in Proposition~\ref{prop:vl_property}. 
\begin{tcolorbox}[emphblock]
\vspace{-2mm}
\begin{proposition}[Properties of $v_{l,\tau}$]
\label{prop:vl_property}
Suppose Assumption~\ref{assmp:weak_c_f} holds.
The merit function $v_{l,\tau}(x)$ defined in~\eqref{eq:vl} satisfies the following properties: \\ 
1. $\bar{u}(x) - \tau \ln M \leq v_{0,\tau} (x) \leq \bar{u}(x)$.
Furthermore, $\min_{x\in {\cal X}} v_{l,\tau}(x) = - \tau \ln M$. \\
2. If $x$ is weakly Pareto optimal, then $v_{l,\tau}(x)\leq 0$. 
Conversely, if 
~a) $l = 0$, $v_{l,\tau}(x)\leq 0$,
then $x$ is $\epsilon$-weakly Pareto optimal with $\epsilon = \tau \ln M $;
~b) $l > 0$, $v_{l,\tau}(x)\leq - \tau \ln M$, and for all $m\in [M]$, 
$f_m$ are $(1, 0)$-point quasar-convex at $x$,
then $x$ is weakly Pareto optimal. 
\end{proposition}
\vspace{-3mm}
\end{tcolorbox}
The proof of Proposition~\ref{prop:vl_property} is deferred to Appendix~\ref{sub_app:proof_v_weak_po}.
In Appendix~\ref{sub_app:example_quasar},
we provide some examples of nonconvex $F$ that satisfies 
condition 2-b) in Proposition~\ref{prop:vl_property}.

Next we discuss how to compute the gradient of 
the smoothed merit function $v_{l,\tau}$, so that we can use 
gradient-based methods to directly minimize $v_{l,\tau}$.\\
\textbf{Gradient of the smoothed merit function.}
Under Assumption~\ref{assmp:weak_c_f}, all the objectives
$f_m$ are $\mu$-weakly convex, i.e., $f_m(x) - \frac{\mu}{2}\|x\|^2$ is convex on ${\cal X}$. The LSE function preserves weak convexity (c.f. Lemma~\ref{lemma:log_sum_exp_preserve_weak_convexity}),
thus $h_{l,\tau}(x,y)$ is strongly convex w.r.t. $y$
if $l + \mu > 0$. Then $y^*_{l,\tau}(x) \coloneqq 
\arg\min_{y\in {\cal X}} h_{l,\tau}(x,y)$ is a singleton,
and is continuous w.r.t. $x$,
so the Danskin-type theorem can be applied here
to compute the gradient of $v_{l,\tau}$, given by
\begin{subequations}\label{eq:grad_v_main}
\begin{align}
& \nabla v_{l,\tau}(x) = 
\sum_{m=1}^M \pi_m(x) \nabla f_m(x) - l(x - y^*_{l,\tau}(x)), \\
& \text{with}~~
\pi_m(x) \coloneqq \frac{e^{\frac{1}{\tau}(f_m(y^*_{l,\tau}(x)) - f_m(x))}} {\sum_{m=1}^M e^{\frac{1}{\tau}(f_m(y^*_{l,\tau}(x)) - f_m(x))} }.
\end{align}  
\end{subequations}
\textbf{Reformulation of problem~\eqref{eq:pareto_constrained_opt}.}
We then consider the following optimization problem by approximating the Pareto set constraint $x \in WP(F)$ through a merit function constraint using $v_{l,\tau}$.
{\begin{equation}
\min_{x\in {\cal X}} ~ f_0(x), 
~\mathrm{s.t.}~ x \in {\cal X}_{v_{l,\tau}}^* 
\!\!\coloneqq \!\!
\{x\in {\cal X}\mid v_{l,\tau}(x)
+\tau\ln M \leq 0\} .
\label{eq:bilevel_opt_vltau}
\tag{CP}
\end{equation}}
We name the above program a constrained program (CP) reformulation.
By Proposition~\ref{prop:vl_property}, 
$ S(F) \subseteq WP(F)$, and these two sets become equal 
as $l,\tau\downarrow 0$.
To solve~\eqref{eq:bilevel_opt_vltau},
we further consider a penalty-based program~\eqref{eq:approximate_penalty}
with penalty parameter $\gamma, \theta > 0$ below
\begin{align}\label{eq:approximate_penalty}
\tag{PP$_\gamma$}
&\min_{x\in {\cal X}} ~ \varphi_\gamma(x) 
\coloneqq f_0(x) + \gamma p(x) 
~~~\text{with}~~{p(x) \coloneqq (v_{l,\tau}(x) + \tau \ln M)^{\theta}} 
\nonumber .
\end{align}
For~\eqref{eq:approximate_penalty}, when $\gamma \to \infty$, any limit point 
of the sequence of solutions to the approximation problem~\eqref{eq:approximate_penalty} is a solution to the problem~\eqref{eq:bilevel_opt_vltau}, 
as proved in Appendix~\ref{sub_app:proof_relation_asymptotic}.

\subsection{Relation of different formulations}
\label{sub:relation_formulation}

To establish the relations of the solutions 
to~\eqref{eq:approximate_penalty} and~\eqref{eq:bilevel_opt_vltau},
without loss of generality, we assume there exists 
at least one $x^* \in \arg\min_{x\in {\cal X}} v_{l,\tau}(x)$
that $x^*$ is bounded, and that the function value $f_m$ 
and gradient $\nabla f_m$ at $x^*$ for $m=0,\ldots,M$
are also bounded.
We also introduce the following subanalyticity
assumptions on the objectives below.
\begin{definition}[Subanalyticity~\cite{Bierstone1988_Semianalytic}]
\label{def:subanalytic}
1) A subset $S \subset \R^q$ is called \emph{semianalytic} if each point of $\R^q$ admits a neighborhood $V$ for which $S \cap V$ assumes the following form
\begin{align}
  \bigcup_{i=1}^I \bigcap_{j=1}^J\left\{x \in V: f_{i j}(x)=0, g_{i j}(x)>0\right\},  
\end{align}
where $f_{i j}, g_{i j}: V \rightarrow \R$ are real analytic functions for $1 \leq i \leq I, 1 \leq j \leq J$.\\
2) A subset $S \subset \R^q$ is called \emph{subanalytic} if each point of $\R^q$ admits a neighborhood $V$ such that
\begin{align}
  S \cap V = \{x \in \R^q \mid (x, y) \in B \}  
\end{align}
where $B$ is a bounded semianalytic subset of $\R^q \times \R^m$.\\
3) A function $f: \R^q \rightarrow \R \cup\{+\infty\}$ is called \emph{subanalytic} if its graph is a subanalytic subset of $\R^q \times \R$.  
\end{definition}

\begin{definition}[Global subanalyticity~{\citep[p.~506]{VanDenDriesMiller1996_ominimal}}]
\label{def:global_suba}
Let $x = [x_1, \ldots, x_q]^\top \in \R^q$. Define the function 
\begin{align}\label{eq:Phi_q}
\Phi_q(x) \coloneqq \left(\frac{x_1}{1+x_1^2}, \ldots, \frac{x_q}{1+x_q^2}\right).
\end{align}
1) A subset $S$ of $\R^q$ is called \emph{globally subanalytic} if its image under $\Phi_q$ is a subanalytic subset of $\R^q$.\\
2) A function $f: \R^q \rightarrow \R \cup\{+\infty\}$ is called \emph{globally subanalytic} if its graph is a globally subanalytic subset of $\R^q \times \R$.
\end{definition}
\begin{assumption}[Subanalyticity of $f_m(x)$]
\label{assmp:subanalyticity_f}
For all $m \in [M]$, $f_m(x)$ is subanalytic on $\cal X$.
\end{assumption} 
The properties of subanalytic functions
are provided in Appendix~\ref{sec_app:subanalytic_property}.
Subanalyticity can be 
generally satisfied by many widely-used objective functions~\cite{VanDenDriesMiller1996_ominimal,Bolte2007_subanalytic}. 
For example, the $\ell_p$-norm with $p\geq 1$, 
and the LSE and polynomial functions defined on a bounded set, 
all satisfy the subanalyticity.
More discussions and examples are provided in Appendix~\ref{sub_app:examples_global_suba} and Table~\ref{tab:HEB_example}.
Intuitively speaking, global subanalytic functions can be described 
by finite combinations of locally analytic functions.
They exhibit  a ``tame'' geometry, 
thus stability under basic operations,
and desirable properties for optimization.
One of them is the H\"{o}lderian error bound 
defined below. 
\begin{definition}[$(\varrho,\eta)$-H\"{o}lderian error bound]
\label{def:varrho_eta_dist_bound}
For a function $v: {\cal X} \to \R$, 
let ${\cal X}^*_v \coloneqq \arg\min_{x\in {\cal X}}~ v(x)$.
Then $v$ satisfies the $(\varrho,\eta)$-H\"{o}lderian error bound (HEB) if 
\begin{equation}
\varrho \big(v(x) - \min_{x\in {\cal X}} v(x) \big) 
\geq \big(\mathrm{dist}(x, {\cal X}^*_v) \big)^\eta 
\end{equation}
where $\mathrm{dist}(x, S)$ is the Euclidian distance from a point $x$ to a set $S$, and $\varrho, \eta > 0$.
\end{definition}
  
The HEB in Definition~\ref{def:varrho_eta_dist_bound} 
generalizes the widely used Quadratic Growth (QG) condition 
with $\eta = 2$ in optimization~\cite{karimi2016linear_pl,drusvyatskiy2018_eb_qg}, and the weak 
sharp minima condition with $\eta = 1$~\cite{Burke1993_weaksharpmin,Samadi2024_rapm}. This condition ensures the point is close to the solution set if the function value gap at the point is small.
In our problem, it is desirable that the function $v_{l,\tau}$
also satisfies such a condition, so that it satisfies HEB near its solution set.
This can be proved based on the properties of subanalytic functions,
as described in Lemma~\ref{lemma:global_suba_X_v} below.
\begin{lemma}[Subanalyticity of ${\cal X}^*_{v_{l,\tau}}$ and $v_{l,\tau}(x)$]
\label{lemma:global_suba_X_v}
Under Assumption~\ref{assmp:subanalyticity_f},  
given a subanalytic and compact set ${\cal X}_C \subseteq {\cal X}$,
suppose $f_m(x)$ is continuous and bounded on ${\cal X}_C$ for all $m \in [M]$. 
Then both ${\cal X}^*_{v_{l,\tau}}\cap {\cal X}_C$ and $v_{l,\tau}(x)$ on 
${\cal X}_C$ are globally subanalytic.
Consequently, $v_{l,\tau}(x)$, $p(x)$ satisfy the $(\varrho,\eta)$ and 
$(\varrho_p,\eta_p)$-HEB in Definition~\ref{def:varrho_eta_dist_bound}
on ${\cal X}_C \subseteq {\cal X}$, respectively, with some $\varrho, \eta > 0$,
$\eta_p = \theta \eta$, and $\varrho_p = \varrho^{\theta}$.
\end{lemma}
Definition~\ref{def:varrho_eta_dist_bound} and  Lemma~\ref{lemma:global_suba_X_v} are crucial for establishing 
the relations of the global/local/stationary solutions of the penalty reformulation~\eqref{eq:approximate_penalty} and the constrained formulation~\eqref{eq:bilevel_opt_vltau},
as described in
Theorem~\ref{thm:global_solution_relation_eps}. 
Below we first define the global and local solutions
to~\eqref{eq:bilevel_opt_vltau},
then we discuss their relation with 
global and local solutions to~\eqref{eq:approximate_penalty}.
\begin{definition}[Global and local solutions]
We say $x$ is an $(\epsilon, \delta)$-global solution to~\eqref{eq:bilevel_opt_vltau} if  it satisfies
\begin{subequations}
\begin{align}
& f_0(x) - \min_{x\in {\cal X}_{\delta}} f_0(x) \leq \epsilon, \\
& x\in {\cal X}_{\delta} \coloneqq 
\{x \in {\cal X} \mid v_{l,\tau}(x) + \tau\ln M \leq \delta \} .
\end{align}  
\end{subequations}
Let $\mathcal{N}(x,r)$ denote the neighborhood of $x$
with radius $r$. We say $x$ is an $(\epsilon, \delta)$-local solution
to~\eqref{eq:bilevel_opt_vltau} on $\mathcal{N}(x,r)$ if  it satisfies that
\begin{align}
& f_0(x) - \min_{x\in {\cal X}_{\delta}\cap \mathcal{N}(x,r)} f_0(x) \leq \epsilon, 
~~\text{and}~~ x\in {\cal X}_{\delta} .
\end{align}  
\end{definition}

We then discuss the relations of solutions of the penalty formulation
and the constrained formulation.
The works most related to ours on the relations include~\cite{Ye1995_exact_penalty_blo,Luo1996_exact_penalty_mpec} and recent works~\cite{shen2023penalty,chen2024penalty_simple_blo}.
A major difference is that they directly assume the lower-level scalar objective satisfies HEB, QG, convexity or PL inequality, while in our work, we prove the property holds on a bounded set  for $v_{l,\tau}(x)$ and $p(x)$ when the objective $F$ is subanalytic, which is nontrivial.
The two works~\cite{Ye1995_exact_penalty_blo,Luo1996_exact_penalty_mpec}
focus on the cases with exact penalty.
Other differences include that the results in~\cite{shen2023penalty}
only consider LL objective satisfies HEB with $\eta = 2$, while we consider more general $\eta$. Also, we do not require the global convexity or Lipschitz assumption 
as in~\cite{chen2024penalty_simple_blo}. Furthermore, we provide the relation between the stationary solution of the penalty formulation and the KKT solution of the constrained formulation under the general KL inequality, which is not discussed in either of the two works.
By definition, the requirements of HEB and KL are weaker than QG and PL used in~\cite{shen2023penalty}, respectively.
The relations of strong convexity (SC), proximal error bound (EB), proximal PL, QG have been studied in existing literature. 
We summarize these relations using the equation below.
\begin{align*}
  f~\text{is SC}   %
\stackrel{(a)}{\implies} \text{proximal EB} 
\stackrel{(b)}{\iff} \text{proximal PL} 
\stackrel{(c)}{\implies} &\text{QG} \\
\Downarrow \qquad\qquad &~\Downarrow \\
\text{proximal EB} \iff \text{proximal KL} 
\stackrel{}{\implies} &\text{HEB} 
\numberthis \label{eq:kl_heb_relation}
\end{align*}
where $(a)$ has been proved in e.g.,~\citep[Appendix~F-2]{karimi2016linear_pl}; 
$(b)$ has been proved in e.g.,~\citep[Appendix~G]{karimi2016linear_pl};
$(c)$ has been proved in e.g.,~\citep[Theorem~3.1]{liao2024_eb_pl} with additional conditions that $\phi$ is closed and weakly convex,
or in~\citep[Theorem~2]{karimi2016linear_pl} with additional conditions that $g$ is a constant.

\begin{tcolorbox}[emphblock]
\vspace{-2mm}
\begin{theorem}[Relation of $\epsilon$-global solutions]
\label{thm:global_solution_relation_eps}
Suppose Assumptions~\ref{assmp:local_lip_f} and~\ref{assmp:subanalyticity_f} hold. Then a bounded $\epsilon$-global solution of~\eqref{eq:approximate_penalty} is an $(\epsilon, \delta)$-global solution for \eqref{eq:bilevel_opt_vltau} with proper choices of $\gamma, \epsilon$ depending on $\delta$ for all $\delta > 0$.
Conversely, an $(\epsilon_b, \epsilon)$-global solution for~\eqref{eq:bilevel_opt_vltau} is a $\delta$-global solution of~\eqref{eq:approximate_penalty}
with proper choices of $\epsilon_b, \epsilon, \gamma$ depending on $\delta$ for all $\delta > 0$.
\end{theorem}  
\vspace{-3mm}
\end{tcolorbox}
The proof of Theorem~\ref{thm:global_solution_relation_eps} 
is deferred to Appendix~\ref{sub_app:global_relation}.
It extends the result from~\cite{shen2023penalty} using the HEB
proved in Lemma~\ref{lemma:global_suba_X_v}
instead of directly assuming the QG condition. 
Furthermore, it does not rely on 
convexity assumptions of $v_{l,\tau}(x)$. 
A more detailed comparison is given in Appendix~\ref{sub_app:blo_work}.

\begin{table*}
\centering
\fontsize{6.5}{7}\selectfont 
\caption{A recipe to choose hyperparameters $\theta, \gamma$ 
to obtain solutions to~\eqref{eq:approximate_penalty} 
and thus the solutions to~\eqref{eq:bilevel_opt_vltau}.
Denote $\eta_p,\alpha_p$ as the
HEB, KL exponents of $p(x)$, and $\eta,\alpha_v$ as the
HEB, KL exponents of $v_{l,\tau}(x)$ on a subanalytic and compact set, respectively.}
\label{tab:recipe_simple_blo}
\begin{tabular}{c|ccccccc}
\toprule
\eqref{eq:approximate_penalty} 
&$\cal X$ & $v_{l,\tau}(x)$ property &$\theta$ & $p(x)$ property  & $\gamma$ 
&\eqref{eq:bilevel_opt_vltau}\\
\midrule
$\epsilon$-global (Theorem~\ref{thm:global_solution_relation_eps})
&\multirow{2}{*}{compact or $\R^q$} &\multirow{2}{*}{$\eta > 0$} 
&\multirow{2}{*}{$\theta = \frac{\eta_p}{\eta}$} &$\eta_p = 1$  
&\multirow{2}{*}{$\epsilon^{\frac{1}{\eta_p} - 1}$} 
&$(\epsilon, \epsilon)$-global \\
& & & &$\eta_p > 1$   
& &$(\epsilon, \epsilon)$-global \\
\hline
local (Theorem~\ref{thm:relation_local_solution})
&\multirow{2}{*}{compact or $\R^q$} &$\eta > 0$ &\multirow{2}{*}{$\theta = \frac{\eta_p}{\eta}$} 
&$\eta_p = 1$ 
&\multirow{2}{*}{$\epsilon^{\frac{1}{\eta_p}-1}$} 
&$(\epsilon, \epsilon)$-local \\
& & & &$\alpha_p > 1 $  
& &$(\epsilon, \epsilon)$-local\\
\hline
$\epsilon$-stat. (Theorem~\ref{thm:relation_stationary_solution})
&{$\R^q$} &$\alpha_v \geq 2$  & $\theta = 1 $ 
&$\alpha_p \geq 2$   
&$\epsilon^{-1}$ 
&$(\epsilon, \epsilon)$-KKT to~\eqref{eq:blo_nablap}\\
\bottomrule
\end{tabular}
\end{table*}

For generally nonconvex $f_0$ and $v_{l,\tau}$, it is difficult
to guarantee convergence to the global optimal solutions of~\eqref{eq:approximate_penalty}.
In these scenarios, we also establish the 
relations of the local solutions of the penalty reformulation~\eqref{eq:approximate_penalty} and the constrained formulation~\eqref{eq:bilevel_opt_vltau} in Theorem~\ref{thm:relation_local_solution}. 
\begin{tcolorbox}[emphblock]
\vspace{-2mm}
\begin{theorem}[Relation of local solutions]
\label{thm:relation_local_solution}
Suppose Assumptions~\ref{assmp:local_lip_f} and~\ref{assmp:subanalyticity_f} hold, and $f_0$ is $\ell_f$-Lipschitz on any bounded set. 
Let $x_\gamma $  be a local solution of~\eqref{eq:approximate_penalty} on $\mathcal{N}( x_\gamma,  r)$. 
If there exists $\bar{x} \in \mathcal{N} (x_\gamma, r ) \cap {\cal X}$ such that $p ( \bar{x} ) \leq \epsilon$ for some $\epsilon \geq 0$.
Then with proper choices of $\gamma, \epsilon$ depending on $\delta$ for all $\delta > 0$, $ x_\gamma $ is an $(\epsilon, \delta)$-local solution of~\eqref{eq:bilevel_opt_vltau}.
\end{theorem}  
\vspace{-3mm}
\end{tcolorbox}
Theorem~\ref{thm:relation_local_solution} states that
if $v_{l,\tau}$ satisfies HEB,
under additional conditions, the local solution of~\eqref{eq:approximate_penalty}
is a local solution of~\eqref{eq:bilevel_opt_vltau}.
In e.g.,~\cite{shen2023penalty,chen2024penalty_simple_blo}, 
these conditions can be satisfied under certain assumptions on
the lower-level objective, such as the PL inequality or convexity.
However, in our problems~\eqref{eq:approximate_penalty}
and~\eqref{eq:bilevel_opt_vltau}, we cannot directly assume such conditions hold for $v_{l,\tau}$. Therefore, next we will show that the conditions hold under additional conditions
specified in Proposition~\ref{prop:local_relation_condition}.
We first introduce the Kurdyka-Łojasiewicz inequality below.

\begin{definition}[Kurdyka-Łojasiewicz inequality]
A proper and lower semicontinuous function $f: \R^q \to (-\infty,+\infty]$  satisfies the $(c,\alpha)$-Kurdyka-Łojasiewicz (KL) inequality at $\bar{x} $ if there exist $\nu \in(0,+\infty]$, $c>0, \alpha > 1$, a neighborhood $\mathcal{N}(\bar{x})$, such that for all $x \in \mathcal{N}(\bar{x}) $ and $f(\bar{x}) < f(x) < f(\bar{x}) + \nu $, the following inequality holds
\begin{align}
c \big(\mathrm{dist}(0, \partial f(x)) \big)^{\alpha} \geq \|f(x) - f(\bar{x})\| .
\end{align}
Moreover, if $f$ satisfies the $(c,\alpha)$-KL inequality for every pair of points $(x, \bar{x})$ on a set ${\cal X}_C$ with $f(\bar{x}) = \min_{x\in {\cal X}_C} f(x)$, then we say $f$ is $(c,\alpha)$-KL on ${\cal X}_C$.
\end{definition}

\begin{proposition}
\label{prop:local_relation_condition}
Let $x\in {\cal X}$ be a bounded $\epsilon$-stationary point of $\min_{x\in {\cal X}}v_{l,\tau}(x)$.
If there exists $x^*\in {\cal X}^*_{v_{l,\tau}}$ and $x\in \mathcal{N}(x^*)$ with KL inequality at $x^*$, then $v_{l,\tau}$ satisfies the condition in Theorem~\ref{thm:relation_local_solution}.
The above condition holds if 
~a) for all $m\in [M]$, $f_m$ satisfies the $(1,0)$-point quasar-convexity at $x$, and $(1,\mu)$-point strong quasar-convexity at $y^*_{l,\tau}(x) = \arg\min_{y\in {\cal X}} h_{l,\tau}(x, y)$;
or~b) $v_{l,\tau}(x) + \tau\ln M \leq \nu$ in Lemma~\ref{lemma:kl_subanalytic}.
\vspace{-2mm}
\end{proposition}      
%
The condition in Theorem~\ref{thm:relation_local_solution}
 requires a local solution to 
$\min_{x\in {\cal X}} v_{l,\tau}(x)$ is also $\epsilon$-globally optimal to this problem.
Such a condition holds if $F$ satisfies point convexity or subanalyticity near the stationary points of $v_{l,\tau}(x)$, as shown in Proposition~\ref{prop:local_relation_condition}.
The proof is deferred to Appendix~\ref{sub_app:local_relation}. 
Next, we define the KKT condition, its necessity,
and establish the relation of stationary solutions.

\begin{definition}[{$(\epsilon,\delta)$-KKT condition}, e.g.~{\cite{liu2022bome,xiao2023alternating}}]
\label{def:eps_kkt_constrained_nabla_p}
Let ${\cal X} = \R^q$.
For the function $p$ such that $\nabla p(x) = 0$ and $\nabla^2 p(x) = 0$ exist, 
and $\nabla p(x) = 0$ implies $ p(x) = 0$,
a gradient-based reformulation of the problem~\eqref{eq:bilevel_opt_vltau} is
\begin{equation}\label{eq:blo_nablap}
\min \nolimits_{x\in\R^q} f_0(x), 
~\mathrm{s.t.}~\nabla p(x) = 0,  
\end{equation}
The $(\epsilon,\delta)$-KKT condition 
of the problem~\eqref{eq:blo_nablap} is
\begin{equation}\label{eq:blo_kkt_hessianp}
\|\nabla f_0(x) + \nabla^2 p(x) w \| = O (\epsilon),
~\|  \nabla p(x)  \| = O(\delta)
\end{equation}
where $w\in \R^q$ is bounded.
\end{definition}

\begin{tcolorbox}[emphblock]
\vspace{-2mm}
\begin{theorem}[Relation of $\epsilon$-stationary solutions]
\label{thm:relation_stationary_solution}
Let ${\cal X} = \R^q$, and $\theta=1$. 
Let $x_{\gamma}$ be a bounded $\epsilon$-stationary solution
to~\eqref{eq:approximate_penalty}. 
Then there exists a  compact subanalytic set ${\cal X}_C \subset \R^q$
with ${\cal X}_C \cap {\cal X}_{v_{l,\tau}}^* \neq \emptyset$
and $x_\gamma \in {\cal X}_C$. 
Suppose Assumption~\ref{assmp:local_lip_f} holds, 
and that on ${\cal X}_C$, $\nabla v_{l,\tau}(x) $ 
exists and is $\ell_{v, 2}$-smooth, 
and $v_{l,\tau}(x)$ is $(c_v,\alpha_v)$-KL with $\alpha_v \geq 2$.
Then with proper choice of parameter $\gamma$, $\epsilon$
depending on $\delta$,
$x_{\gamma}$ is an $(\epsilon, \delta)$-KKT point to 
the problem~\eqref{eq:blo_nablap}.
\end{theorem}
\vspace{-3mm}      
\end{tcolorbox}
The $\ell_{v, 2}$-smoothness of $\nabla v_{l,\tau}$ can be justified 
under additional assumptions of the Hessian of $f_m$ for $m\in[M]$.
See a detailed discussion in Lemma~\ref{lemma:smooth_nabla_vltau}.
Theorem~\ref{thm:relation_stationary_solution} shows that
for a simple bilevel problem where the lower-level objective $v_{l,\tau}(x)$ 
satisfies the KL inequality with exponent $\alpha_v\geq 2$, 
a stationary solution to the penalty reformulation
approximates the KKT solution to the constrained 
formulation in~\eqref{eq:blo_nablap}. 
This indicates that though the KKT solution 
to a bilevel problem often requires the second-order information
of the lower-level objective
as shown in~\eqref{eq:blo_kkt_hessianp} and discussed in e.g.,~\cite{roy2023opt_pareto_set,liu2022bome,xiao2023alternating},
one can still use first-order methods to approximate such solutions.
The proof of Theorem~\ref{thm:relation_stationary_solution}
is provided in Appendix~\ref{sub_app:stationary_relation}.
We first prove that a constraint qualification condition holds in the above settings,
ensuring the KKT conditions are necessary for global optimality.
In addition,  we discuss cases where the necessary KKT condition only requires 
first-order information depending on the HEB exponent of $p(x)$.
We also discuss when $p(x)$ satisfies the above conditions.

\subsection{Comparison with existing methods}
To address the preference-guided MOL problem, 
one commonly used formulation is linear scalarization (LS)
where the preference is modeled by different weights of the objectives.
Another commonly used formulation is constrained vector optimization~\cite{lin2019pareto,chen2024ferero,zhang2024pmgda}, given by
\begin{equation}
\min \nolimits_{x\in {\cal X}} F(x),
~~\mathrm{s.t.}~~G(x)\leq 0, H(x) = 0 
\label{eq:constrained_vo}
\end{equation}
where $G, H$ are vector-valued functions defined by user-specified perferences.
Though the above formulation has successful applications in, e.g., multi-task learning, it is only guaranteed to converge to a KKT point of~\eqref{eq:constrained_vo}, which may be far away from the optimal solutions to $\min_{x\in {\cal X}} F(x)$.
See Example~\ref{exmp:kkt_suboptimal}, and the corresponding results in Figure~\ref{fig:exmp_intro}.
An intuitive explanation for this suboptimality is that \eqref{eq:constrained_vo} puts constraints at the lower level, emphasizing more on satisfying the constraints rather than minimizing the objectives.
\begin{figure}
\centering
\begin{subfigure}[b]{0.25\textwidth}
\centering
\includegraphics[width=.98\linewidth]{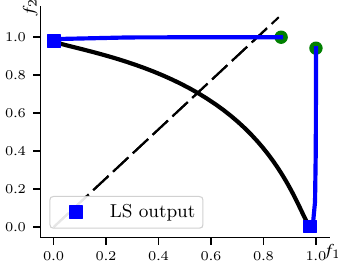}
\caption{LS}
\label{sfig:ls_toy_exmp}
\end{subfigure}  
\begin{subfigure}[b]{0.25\textwidth}
\centering
\includegraphics[width=.98\linewidth]{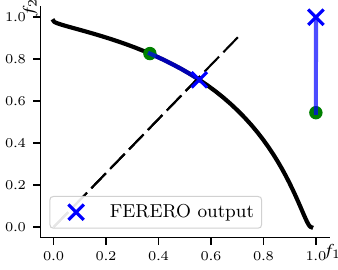}
\caption{FERERO}
\label{sfig:pmtl_toy_exmp}
\end{subfigure}
\begin{subfigure}[b]{0.25\textwidth}
\centering
\includegraphics[width=.98\linewidth]{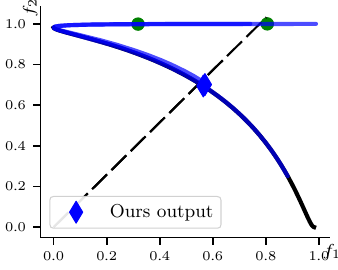}
\caption{FOOPS}
\label{sfig:foops_toy_exmp}
\end{subfigure}
\caption{Results of LS, FERERO, and FOOPS on Example~\ref{exmp:kkt_suboptimal}.
The black dashed lines is the preference defined by $H(x) =0$, or $f_0(x) = \|H(x)\|^2 = 0$.
Green dots represent initial values, blue markers represent converged values for different methods.}
\label{fig:exmp_intro}
\vspace{-2mm}
\end{figure}
\begin{example}\label{exmp:kkt_suboptimal}
Let $M=2, q = 1$, and $M_g=0, M_h=1$. Let ${\cal X} = \R$.
The objective $F$ and constraint $H$ for the 
preference-constained formulation~\eqref{eq:constrained_vo} is defined as 
{\begin{subequations}\label{eq:obj_H_exmp_fail}
\begin{align}\label{eq:obj_synthetic1}
F(x) =& \big( 1-e^{-\|x - \mathbf{1}_q\|_2^2}, 
~~1-e^{- \|x + \mathbf{1}_q \|_2^2} \big), \\
H(x) =& 5f_1(x) - 4f_2(x)  .
\label{eq:obj_H_exmp}
\end{align}  
\end{subequations}}
The corresponding preference function $f_0$
in formulation~\eqref{eq:intro_bi_vector}
to minimize the constraint violation of $H$ 
is defined as
\begin{align}
 f_0(x) = \|H(x)\|^2 .
\end{align}
In Figure~\ref{fig:exmp_intro}, it is easy to see
that there exists a solution $x^*$ in the Pareto set
whose objective $F(x^*)$ is at the intersection of the 
dashed line defined by $H(x) = 0$ and the 
solid curve representing the Pareto front.
Therefore, $x^*$ is an optimal solution to 
both~\eqref{eq:intro_bi_vector} and~\eqref{eq:constrained_vo}. 
\end{example}  
In this example, linear scalarization (LS) fails to converge to the preferred region that $H(x) = 0$, even after enumerating different weights.
Indeed, we could prove that with certain initialization,
LS cannot converge to a point that satisfies~\eqref{eq:obj_H_exmp}
with different weights, see Proposition~\ref{prop:ls_fail_exmp}.
The intuition is that, in Example~\ref{exmp:kkt_suboptimal}, 
a solution to both~\eqref{eq:intro_bi_vector} and~\eqref{eq:constrained_vo} that satisfies 
$\nabla F(x) \lambda = 0$ with $\lambda\in \Delta^M$ is a local maximum point of the objective $ \lambda^\top F(x)$.
It is also worth noting that, there are some other examples 
where LS cannot find all points on the Pareto front 
even after enumerating all possible weights~\cite{Osyczka1984Multicriterion,athan_quasi_1994,hu2024revisiting_ls_limit}.
See a detailed discussion in Appendix~\ref{sub_app:limitation_ls_kkt}.
\begin{proposition}\label{prop:ls_fail_exmp}
Under certain initializations, there exists no $\lambda\in \Delta^M$ such that 
gradient descent algorithm on LS objective with weight $\lambda$  converges to the solution of~\eqref{eq:obj_H_exmp_fail}.
\end{proposition}

Moreover, algorithms developed under the preference-constrained formulation~\eqref{eq:constrained_vo}, such as
PMTL~\cite{lin2019pareto} and
FERERO~\cite{chen2024ferero} (with the partial order cone being 
a nonnegative orthant cone $\R_{+}^M$), converge to a KKT point to~\eqref{eq:constrained_vo}, which is not necessarily Pareto optimal or Pareto stationary.
For a more detailed discussion on this example, 
see Appendix~\ref{sub_app:limitation_ls_kkt}.
Proposition~\ref{prop:kkt_not_ps} is provided 
to further support the claim 
that the KKT solution to~\eqref{eq:constrained_vo}
can be suboptimal for $\min_{x\in {\cal X}} F(x)$,
with its proof 
in Appendix~\ref{sub_app:limitation_ls_kkt}.
\begin{proposition}
\label{prop:kkt_not_ps}
The KKT solution to \eqref{eq:constrained_vo} is not necessarily Pareto stationary to $\min_{x\in {\cal X}} F(x)$.
\end{proposition}

\begin{table}
\centering
\caption{Comparison with existing methods for OPS.
``NC'', ``C'' and ``SC'' represent ``nonconvex'', 
``convex'' and ``strongly convex'', respectively.
``Ncs.'' represents whether the method converges to 
a necessary condition to the OPS problem.
}
\label{tab:comparison_methods}
\begin{tabular}{c|cccc}
\toprule
\multirow{2}{*}{Method} 
& \multirow{2}{*}{\makecell{$f_m(x)$,\\$m\in [M]$}} 
& \multirow{2}{*}{$f_0(x)$} 
& \multirow{2}{*}{\makecell{First\\order}}  
& \multirow{2}{*}{\makecell{Ncs.}} \\
\\
\midrule
PB-PDO~\cite{kamani2021pareto_fairness} 
&NC &NC &\cmark   
& \xmark \\
TAWT~\cite{chen2022weighted_crosstask} 
&NC &NC &\xmark & \xmark \\
PNG~\cite{ye2022pareto_nav} 
&{NC} &{NC} 
& {\cmark} & \xmark \\
PMM~\cite{roy2023opt_pareto_set}
&SC &SC & \xmark  & \cmark \\
BSG~\cite{giovannelli2024bilevel_moo}
&Stricly C &NC & \xmark  & \cmark \\
Ours 
&NC &NC & \cmark  & \cmark \\
\bottomrule
\end{tabular}
\vspace{-5mm}
\end{table}
Besides modeling preference by weights or constraints, 
there are also some works which model the preference by 
objectives, and formulate the problem as optimization on the
Pareto set (OPS), as in~\eqref{eq:pareto_constrained_opt}.
Among these methods, Preference-Based Pareto 
Descent Optimization (PB-PDO)~\cite{kamani2021pareto_fairness} 
and Pareto Navigation Gradient descent (PNG)~\cite{ye2022pareto_nav}
use a descent-type algorithm to ensure the output of the algorithm 
satisfies the preference and decreases the objectives $F$
at each iteration.
However, it has been shown in~\cite{roy2023opt_pareto_set} that the stationary condition derived in~\cite{ye2022pareto_nav} is not a necessary optimality condition for~\eqref{eq:pareto_constrained_opt}.
Furthermore, as discussed in~\citep[Proposition~3]{roy2023opt_pareto_set}, a non-trivial stationarity condition for~\eqref{eq:pareto_constrained_opt} typically requires second-order derivatives of the objectives $f_m, m\in[M]$.
Different from these methods, 
Target-Aware Weighted Training (TAWT)~\cite{chen2022weighted_crosstask}
and Pareto Majorization-Minimization (PMM)~\cite{roy2023opt_pareto_set}
convert the lower-level vector-valued objective to a scalar-valued
objective through linear scalarization (LS), and optimize both the scalarization
weight and the model parameter.
However, as discussed in Example~\ref{exmp:kkt_suboptimal} 
and Appendix~\ref{sub_app:limitation_ls_kkt}, 
optimality for LS is not necessary for~\eqref{eq:pareto_constrained_opt} 
unless $f_m$ are convex for all $m\in [M]$.

We summarize in Table~\ref{tab:comparison_methods} the key differences 
of our work compared to existing methods for OPS.
See also a detailed review in Section~\ref{sec:related_works}
and Appendix~\ref{sec_app:related_work}.

\section{Algorithms and analysis}
\label{sec:analysis}
In this section, we introduce practical first-order
gradient-based algorithms to solve~\eqref{eq:approximate_penalty}. 
We first update $y$ to obtain an estimate for $y^*_{l,\tau}(x)$, 
and thus an estimate for $v_{l,\tau}(x)
= - h_{l,\tau}(x, y^*_{l,\tau}(x))$.
Then we update $x$ based on the estimated penalty function $\varphi_\gamma$. 

At the $t$-th outer iteration and the $k$-th inner iteration,
we iteratively update $x_t$ and $y_{t,k}$ as follows.
\begin{subequations}
\begin{align}
\label{eq:update_ytk}
y_{t,k+1} &= \mathrm{U}_y(y_{t,k},\Delta y_{t,k};\beta_{t,k}, k);\\
\label{eq:update_xt}
\!\!\!\!
x_{t+1} &= \mathrm{U}_x(x_t,\Delta x_{t};\alpha_{t},t) 
\end{align}
\end{subequations}
where $\mathrm{U}$ is some gradient based oracle, and the gradient vectors with respect to $y$ and $x$ are defined as 
\begin{align*}
\Delta y_{t,k}&=\nabla_y h_{l,\tau}(x_t, y_{t,k})\\
\Delta x_{t}&=\nabla f_0(x_t ) \!-\! \gamma_t \theta \operatorname{sign}(v_t)|v_t|^{\theta-1} \nabla_x h_{l,\tau}(x_t , y_{t+1})
\end{align*}
with $y_{t+1}=y_{t,K_t}$ and $v_t=\tau\ln M - h_{l,\tau}(x_t , y_{t+1})$. 

A meta algorithm with the above updates is summarized in Algorithm~\ref{alg:meta_v_double_loop}.
We name it \textsf{F}irst-\textsf{O}rder \textsf{O}ptimization on the \textsf{P}areto \textsf{S}et (\textbf{FOOPS}) algorithm.

\begin{algorithm}[H]
\caption{The FOOPS algorithm with oracles}
\label{alg:meta_v_double_loop} 
\begin{algorithmic}[1]
\STATE Initialize $t=0$, $x_0,y_0$, set step sizes $\{\alpha_t, \beta_t \}$, 
penalty parameter $\{\gamma_t\}$, inner-loop iterations $\{K_t\}$.
\WHILE {$\|\nabla \varphi_{\gamma_t}(x_t)\|^2 > \epsilon $ }
\STATE Set $k = 0$;
\FOR {$ k = 0, \ldots, K_t-1 $ }
\STATE Update $y_{t,k}$ by \eqref{eq:update_ytk};
\ENDFOR
\STATE Set $y_{t+1} = y_{t,K_t}$;
\STATE Update $x_t$ by \eqref{eq:update_xt};
\STATE Set $t = t+1$;
\ENDWHILE
\end{algorithmic} 
\end{algorithm}

We then discuss the choices of update oracles and the non-asymptotic convergence rate of Algorithm~\ref{alg:meta_v_double_loop} with different oracles below. Let $w$ represent the updated parameter, which can be either $x$ and $y$, $\Delta w$ denote the gradient vector, $\alpha$ the stepsize, and $t$ is the iteration number.

We give examples of oracles using projected gradient desent (PGD), 
and momentum updates (Momentum) in~\eqref{eq:update_oracle}.
The Nesterov's acceleration and Adam update are also applicable, which is detailed in Appendix~\ref{sec_app:proof_convergence}, 
one can also see in e.g.,~\citep{wang2024convergence}. 
\begin{subequations}\label{eq:update_oracle}
\begin{align}
&\text{PGD:}~\mathrm{U}(w,\Delta w;\alpha_t, t)
=\mathrm{Proj}_{\mathcal{X}}(w-\alpha_t\Delta w) \\
&\text{Momentum:}~
\mathrm{U}(w,\Delta w;\alpha_t, t)
=\mathrm{Proj}_{\mathcal{X}}(w-\alpha_t v_t), \nonumber\\
&\qquad\qquad\qquad
\text{with}~v_t=\tilde{\alpha} v_{t-1}+\Delta w 
\end{align}      
\end{subequations}
\textbf{Discussion about the convergence.} Since $h_{l,\tau}(x,y) $ is $\mu_{h_y}$-strongly convex w.r.t. $y$ as detailed in the proof of Lemma~\ref{lemma:contraction_ytk}, when we choose $\mathrm{U}_y$ as projected gradient descent, the momentum updates and Nesterov's acceleration, it gives linear convergence rate for the inner-loop of $y$. For the outer-loop w.r.t. $x$, if $\theta\geq 1$, then the objective $\varphi_\gamma(x)$ is differentiable and thus projected gradient descent, momentum, Nesterov's acceleration, Adam updates give $O(1/T)$ convergence rate in the deterministic setting and  $O(1/\sqrt{T})$ convergence rate in the stochastic setting according to \citep{wang2024convergence}. Therefore, combining outer-loop and inner-loop update oracles together, Algorithm \ref{alg:meta_v_double_loop} converges. We provide a proof for the convergence of  Algorithm~\ref{alg:meta_v_double_loop} in Appendix~\ref{sec_app:proof_convergence}, Theorem~\ref{thm:convergence_meta_alg}, when choosing both $\mathrm{U}_y$ and $\mathrm{U}_x$ as the PGD oracle and choosing $\theta \geq 1$,
and assuming the objectives $f_m(x)$ for $m=0,\ldots, M$ are smooth. 
For $\theta<1$, $p(x)$ can be nonsmooth. The convergence for Algorithm~\ref{alg:meta_v_double_loop} can be built upon nonsmooth optimization~\citep{kiwiel2004convergence,davis2018subgradient,davis2018stochastic}, but possibly under additional assumptions. 
Some discussions about algorithms for nonsmooth lower-level objective are provided in~\cite{chen2024penalty_simple_blo}. 
When $f_0(x)$ and $p(x)$ are Lipschitz, the algorithm can be applied to our problem.
We defer a more detailed study of algorithm development in nonsmooth cases to future research.

\textbf{Single-loop and stochastic variants.}
When we take the exact penalty method with $\eta_p=1$, 
$\gamma_t$ can be upper bounded by a constant.
Then $K_t$ can also be upper bounded by a constant.
In fact, under Assumption~\ref{assmp:weak_c_f}, 
when $l > \ell_{f,1}$,
a single-loop variant of Algorithm~\ref{alg:meta_v_double_loop}
that takes $K_t = 1$ also has non-asymptotic convergence guarantees,
see e.g.~\cite{chen2021closing_gap}.
Besides, stochastic variants of Algorithm~\ref{alg:meta_v_double_loop}
can also be derived which replace the deterministic gradients with their unbiased stochastic estimates.
We leave the convergence analysis of such variants for future work.

\section{Related works}
\label{sec:related_works}
We discuss recent works that are most related to ours.
An extended discussion is provided in Appendix~\ref{sub_app:blo_work}.\\
\textbf{Preference-guided multi-objective learning.}
Preferences in multi-objective optimization can be represented using weights, thresholds, or preference vectors. Scalarization methods, such as linear and Tchebycheff scalarization, convert vector objectives into scalar objectives by applying weighted norms~\cite{miettinen_nonlinear_1998}. Alternatively, $\epsilon$-constraint methods impose thresholds on objectives to convert the problem to a constrained optimization problem~\cite{curtis2023fair}. 
There are also other approaches which represent preferences with vectors in the objective space, focusing on finding optimal solutions satisfying constraints defined by these vectors~\cite{lin2019pareto,chen2024ferero,zhang2024pmgda} or minimizing distances to the vectors~\cite{mahapatra2020multi,momma2022multi}.
See also a comprehensive review in~\cite{chen2025gradient_moo_review}.
\\
\textbf{Optimization on the Pareto Set.}
In machine learning, there are {specific instantiations} of the OPS problem.
For example, EPO~\cite{mahapatra2020multi} and Preference-Based Pareto Descent Optimization (PB-PDO)~\cite{kamani2021pareto_fairness} find a Pareto model such that the objective values satisfy a ratio constraint by minimizing the non-uniformity score. 
Specifically, EPO finds update directions to improve the objective values, or to reduce the constraint violations, or both.
PB-PDO finds common descent directions for the lower-level multi-objectives and the upper-level objective.
They are not guaranteed to converge to a necessary condition of the OPS problem.
Target-Aware Weighted Training (TAWT)~\cite{chen2022weighted_crosstask} learns a multi-task model that minimizes the discrepancy of task representations to ensure they are similar.
TAWT converts the lower-level objectives to a linearly scalarized objective, and optimizes the scalarization weights in the upper level.
This approach has the limitation of introducing undesired local solutions. 
\\
\textbf{(Simple) bilevel optimization.}
Problem~\eqref{eq:bilevel_opt_vltau} is a constrained reformulation of the simple bilevel program with a generally nonconvex lower-level (LL) objective.
For nonconvex LL objectives, algorithms are proposed in e.g.,~\cite{liu2021nonconvex_blo,huang2023momentum_nc_blo,xiao2023alternating}.
However, they usually require second-order derivatives in the algorithms, which can be expensive to implement.
More recently, \emph{first-order Hessian-free} approaches were proposed to address this by the value function reformulation of BLO. For example, 
interior-point method~\cite{liu21_interior_blo},
sequential quadratic programming~\cite{liu2022bome},
and smoothed Lagrangian method~\cite{lu2023slm_blo} 
were used to solve the constrained problem.  
Later on, a \emph{penalty-based} algorithm was proposed~\cite{shen2023penalty}. 
Variants such as Moreau Envelope based algorithms~\cite{kwon2023penalty_envelope,liu2024moreau} and adaptive algorithms~\cite{chen2024penalty_simple_blo}
were proposed.
However, none of these works tackle BLO with vector-valued LL objective.
Moreover, even after converting the vector-valued LL objective to a scalar-valued one through the merit function $v_{l,\tau}$, the LL objective is generally nonconvex and non-PL even if the objectives $f_m, m\in[M]$ are all convex or PL. Therefore, it is difficult to directly apply the
existing analysis or algorithms to our problem.

\section{Experiments} 
\label{sec:experiments}

In this section, we conduct experiments to verify our theory and show the applicability of the algorithms to 
preference-guided multi-task learning. 
We use LS, PMTL~\cite{lin2019pareto}, EPO~\cite{mahapatra2020multi}, XWC-MGDA (XM)~\cite{momma2022multi}, FERERO~\cite{chen2024ferero} as baselines for comparison.
For a preference vector-guided MOL problem,
we define $f_0(x) = \|H(x)\|^2$, where $H(x)$
is the equality constraint function derived from the preference vector~\cite{chen2024ferero}.

\textbf{Metrics.}
\emph{Objective loss and accuracy.}
We report the objective losses and accuracies in classification.\\
\emph{Hypervolume.} Let $F' \in \R^M$ denote the Nadir point, i.e., the worst performance on single-task baselines, and $\mathcal{S}$ denote a set of objective function values of the obtained models.
Hypervolume measures the size of the dominated space of $\mathcal{S}$ relative to $F'$, which can be computed by $H(\mathcal{S})=\Lambda ( \{q \in \R^M \mid \exists F \in \mathcal{S}: F \leq q  \leq F' \} )$, where $\Lambda(\cdot)$ denotes the Lebesgue measure.\\
\textbf{Additional details.} The implementation and additional experiments can be found in Appendix~\ref{sec_app:experiment}. 

\begin{figure}[ht]
\centering
\begin{subfigure}[b]{0.25\textwidth}
\centering
\includegraphics[width=.98\linewidth]{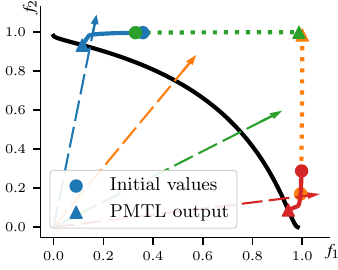}
\caption{PMTL}
\label{sfig:syn_init_close_pmtl}  
\end{subfigure}
\begin{subfigure}[b]{0.25\textwidth}
\centering
\includegraphics[width=.98\linewidth]{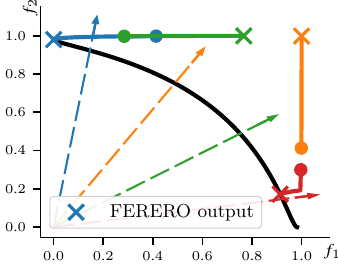}
\caption{FERERO}
\label{sfig:syn_init_close_ferero}  
\end{subfigure}
\begin{subfigure}[b]{0.25\textwidth}
\centering
\includegraphics[width=.98\linewidth]{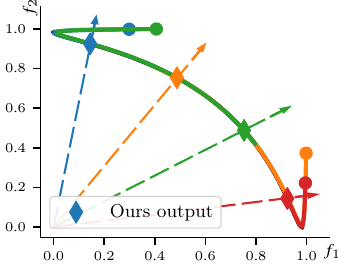}
\caption{FOOPS}
\label{sfig:syn_init_close_ours}  
\end{subfigure}
\caption{Outputs (colored markers) and optimization trajectories (colored curves) of different methods when initial objectives are near the Pareto front. Dashed arrows with different colors represent different preferences.}
\label{fig:syn_init_close}  
\end{figure}

\begin{figure*}[ht]
\centering
\begin{subfigure}[b]{0.32\textwidth}
  \centering
  \includegraphics[width=.98\linewidth]{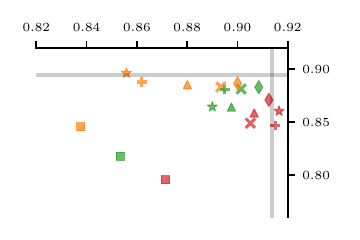}
  \caption{Multi-MNIST accuracy}
  \label{sfig:mnist_acc}
\end{subfigure}
\begin{subfigure}[b]{0.32\textwidth}
  \centering
  \includegraphics[width=.98\linewidth]{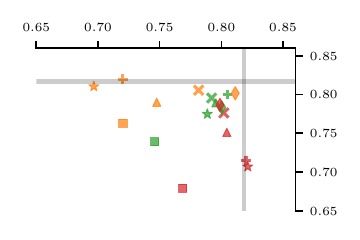}
  \caption{Multi-Fashion accuracy}
  \label{sfig:fashion_acc}  
\end{subfigure}
\begin{subfigure}[b]{0.32\textwidth}
  \centering
  \includegraphics[width=.98\linewidth]{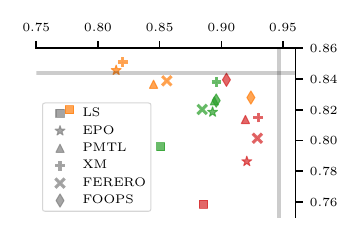}
  \caption{Multi-F+M accuracy}
  \label{sfig:fashion_and_mnist_acc}  
\end{subfigure}
\begin{subfigure}[b]{0.32\textwidth}
\centering
\includegraphics[width=.98\linewidth]{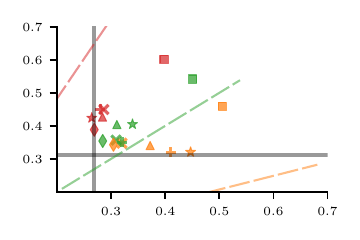}
\caption{Multi-MNIST loss}
\label{sfig:mnist_loss_app}  
\end{subfigure}
\begin{subfigure}[b]{0.32\textwidth}
\centering
\includegraphics[width=.98\linewidth]{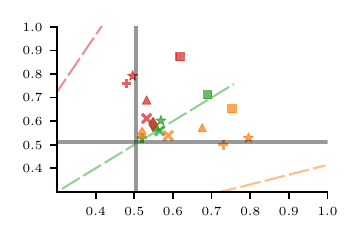}
\caption{Multi-Fashion loss}
\label{sfig:fashion_loss_app}  
\end{subfigure}
\begin{subfigure}[b]{0.32\textwidth}
\centering
\includegraphics[width=.98\linewidth]{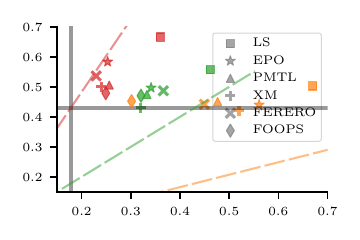}
\caption{Multi-F+M loss}
\label{sfig:fashion_and_mnist_loss_app}  
\end{subfigure}
\vspace{-3mm}
\caption{
Losses and accuracies of various methods with different preferences across three image datasets. The horizontal and vertical axes represent results for objective 1 and objective 2, respectively. Different colored dashed arrows indicate various preference vectors. Different markers denote the solutions obtained by different methods, with marker colors matching the preferences.
}
\vspace{-3mm}
\label{fig:img_class_MF}  
\end{figure*}

\begin{table*}[ht]
\small
\caption{Hypervolumes of different methods $\uparrow$
($\times 10^{-2}$).}
\label{tab:hyperv}
\centering
\begin{tabular}{l|cccccc}
\toprule 
{Datasets} & {LS} & {PMTL} & {EPO} 
& {XM} & {FERERO} & FOOPS\\
\midrule
Mt-M loss   
&1.68 &1.41 &1.35 &1.42 
&{1.95{\tiny$\pm$0.21}} &\textbf{2.62{\tiny$\pm$0.21}}\\
Mt-F loss 
&6.75 &5.90 &6.02 &6.77 
&{7.76{\tiny$\pm$0.18}} &\textbf{8.32{\tiny$\pm$0.37}}\\
Mt-F+M  loss 
&3.63 &3.03 &3.76 &{3.89} 
&3.82{\tiny$\pm$0.21} &\textbf{4.80{\tiny$\pm$0.45}}\\
Mt-M accuracy 
&0.19 &0.15 &0.15 &0.16 
&{0.25{\tiny$\pm$0.04}} &\textbf{0.33{\tiny$\pm$0.02}}\\
Mt-F accuracy 
&0.99 &0.87 &0.87 &0.99 
&{1.13{\tiny$\pm$0.07}} &\textbf{1.22{\tiny$\pm$0.07}}\\
Mt-F+M accuracy 
&0.48 &0.40 &0.50 &0.52 
&{0.53{\tiny$\pm$0.04}} &\textbf{0.72{\tiny$\pm$0.06}}\\
\bottomrule
\end{tabular}
\end{table*}

\textbf{Example~\ref{exmp:kkt_suboptimal}.} Following~\cite{lin2019pareto,mahapatra2020multi}, 
the first objective we consider is \eqref{eq:obj_synthetic1} in Example~\ref{exmp:kkt_suboptimal}, but with $q = 20$.
The results for the experiments with hard initialization are displayed in Figure~\ref{fig:syn_init_close}.
They show that under certain initializations and preferences, 
algorithms developed 
under~\eqref{eq:constrained_vo} such as PMTL and FERERO (with $A=I$
therein) could fail to reach the Pareto front.
It further justifies our Proposition~\ref{prop:kkt_not_ps} 
that the KKT solution to~\eqref{eq:constrained_vo},
with preferences modelled by the constraints at the lower level,
can be suboptimal for $\min_{x\in {\cal X}} F(x)$.
In contrast, FOOPS successfully converges to 
preferred Pareto optimal solutions 
under different preferences.
It demonstrates the benefit of the OPS formulation over~\eqref{eq:constrained_vo}, by prioritizing the attainment of weak Pareto optimality instead of the optimality of the preference function.

\begin{wraptable}{r}{0.6\textwidth}
\caption{WERs $\downarrow$ (\%) for speech recognition. 
}
\label{tab:lib_ais} 
\centering
\begin{tabular}{@{}l|cc c@{}}
\toprule
Method & English & Chinese & Average\\
\midrule
Komatsu et al.~\citep{ctclibrispeech} 
&7.11 &- &-\\
w/o CPC~\citep{saif24_interspeech} 
& 11.8 &10.2 &11.0 \\
Init. (M2ASR)~\citep{saif24_interspeech} 
&7.3 &6.2 &6.7 \\
{LS-FT} &6.8 &5.9 &6.4 \\
FERERO-FT &\textbf{5.4} & {4.9} & \textbf{5.1}\\
FOOPS-FT &{5.7} & \textbf{4.7} & \textbf{5.1} \\
\bottomrule
\end{tabular}
\end{wraptable}
\textbf{Multi-patch image classification.} 
Following~\cite{momma2022multi}, we use  Multi-MNIST (Mt-M), Multi-Fashion (Mt-F), and Multi-Fashion+MNIST (Mt-F+M) for image classification.
The two tasks or objectives in all three datasets are to classify the top-left and the bottom-right images, respectively.
We use LeNet as the backbone neural network.
The losses of different methods given different preference vectors are plotted in Figure~\ref{fig:img_class_MF}.
Experiments for our method are repeated 5 times. Hypervolumes with means and standard deviations are reported in Table~\ref{tab:hyperv}.
The results for other methods in Table~\ref{tab:hyperv} are referenced from~\cite{momma2022multi,chen2024ferero}.
The results show that FOOPS is better at obtaining large hypervolumes, 
see Table~\ref{tab:hyperv}, 
but worse at aligning with preferences compared to other methods,
see Figure~\ref{fig:img_class_MF}.

\textbf{Multi-lingual speech recognition.}
We use the proposed method to fine-tune a pre-trained multi-lingual speech recognition model.
The datasets include Librispeech (100 hours)~\citep{panayotov2015librispeech}, and AISHELL v1~\citep{bu2017aISHELL}.
The model architecture is a conformer with 8 blocks. 
The speech recognition Connectionist Temporal Classification (CTC) losses in Chinese and English are denoted as $f_{t}^{\rm ch}$ and $f_{t}^{\rm en}$, respectively. We also use the self-supervised Contrastive Predictive Coding (CPC) loss $f_{p}$ for representation learning, 
i.e.,
{
\begin{subequations}
\begin{align} 
  & \min f_0(x) \coloneqq 
\| f_{t}^{\rm ch}(x) - f_{t}^{\rm en}(x) \|^2 \\
&\mathrm{s.t.}~~~
x\in \mathop{\arg\min}~ 
F(x) \coloneqq \big(f_{p}(x), 
f_{t}^{\rm ch}(x), f_{t}^{\rm en}(x) \big)^\top 
\end{align}
\end{subequations}
}
where the lower-level objective $f_p$ ensures the model learns a good representation, and the upper-level objective ensures the difference of the 
performances on both languages is small; see more details in Appendix~\ref{sec_app:experiment}.
The results are reported in Table~\ref{tab:lib_ais}, 
which show that FOOPS demonstrate competitive average performance on different languages, but is less good at optimizing the fairness preference function $f_0$. This observation is consistent with that in Figure~\ref{fig:img_class_MF}.

\section{Conclusions} 
\label{sec:conclusions}
In this work, we cast preference-guided multi-objective learning
as an optimization on the Pareto set problem, and proposes a 
first-order penalty method to solve the problem, where
the penalty function is the polynomial of a smoothed merit function.
For the theoretical analysis, first we prove the properties of the merit function, including its relation to weak Pareto optimality, and 
the  H\"{o}lderian error bound. Then we discuss the 
relation of solutions to the penalty reformulation and the original problem. Finally, we discuss the algorithms to solve the penalty problem and their convergence guarantees.
Our theoretical analysis is of independent interest to simple bilevel optimization problem.
For the experiment, we apply the proposed method to synthetic and real-world experiments, which demonstrate the effectiveness of the proposed method 
in finding preference-guided optimal solutions.

\clearpage

{\small
\bibliographystyle{abbrvnat}
\bibliography{myabrv,opt,MOO_opt,MTL,pareto_set_opt,bilevel,speech}

\begin{thebibliography}{86}
\providecommand{\natexlab}[1]{#1}
\providecommand{\url}[1]{\texttt{#1}}
\expandafter\ifx\csname urlstyle\endcsname\relax
  \providecommand{\doi}[1]{doi: #1}\else
  \providecommand{\doi}{doi: \begingroup \urlstyle{rm}\Url}\fi

\bibitem[Athan(1994)]{athan_quasi_1994}
T.~W. Athan.
\newblock \emph{A quasi-Monte Carlo method for multicriteria optimization}.
\newblock PhD thesis, University of Michigan, 1994.

\bibitem[Auslender(1976)]{Auslender1976}
A.~Auslender.
\newblock \emph{Optimisation : m{\'e}thodes num{\'e}riques, vol. 1}.
\newblock Masson, Paris, 1976.

\bibitem[Bai et~al.(2024)Bai, Zeng, Zhang, and Zhang]{bai2024_agils}
X.~Bai, S.~Zeng, J.~Zhang, and L.~Zhang.
\newblock Alternating gradient-type algorithm for bilevel optimization with
  inexact lower-level solutions via moreau envelope-based reformulation.
\newblock \emph{arXiv preprint arXiv:2412.18929}, 2024.

\bibitem[Benko and Mehlitz(2021)]{Benko2021Implicit_opt}
M.~Benko and P.~Mehlitz.
\newblock On implicit variables in optimization theory.
\newblock \emph{Journal of Nonsmooth Analysis and Optimization}, 2:\penalty0
  7215, 2021.

\bibitem[Bierstone and Milman(1988)]{Bierstone1988_Semianalytic}
E.~Bierstone and P.~D. Milman.
\newblock Semianalytic and subanalytic sets.
\newblock \emph{Publications Math{\'e}matiques de l'Institut des Hautes
  {\'E}tudes Scientifiques}, 67:\penalty0 5--42, 1988.

\bibitem[Bolintineanu(1993{\natexlab{a}})]{Bolintineanu1993_Necessary_ops}
S.~Bolintineanu.
\newblock Necessary conditions for nonlinear suboptimization over the
  weakly-efficient set.
\newblock \emph{Journal of Optimization Theory and Applications}, 78\penalty0
  (3):\penalty0 579--598, 1993{\natexlab{a}}.

\bibitem[Bolintineanu(1993{\natexlab{b}})]{Bolintineanu1993_Optimality_ops}
S.~Bolintineanu.
\newblock Optimality conditions for minimization over the (weakly or properly)
  efficient set.
\newblock \emph{Journal of Mathematical Analysis and Applications},
  173\penalty0 (2):\penalty0 523--541, 1993{\natexlab{b}}.

\bibitem[Bolte et~al.(2007)Bolte, Daniilidis, and Lewis]{Bolte2007_subanalytic}
J.~Bolte, A.~Daniilidis, and A.~Lewis.
\newblock The Łojasiewicz inequality for nonsmooth subanalytic functions with
  applications to subgradient dynamical systems.
\newblock \emph{SIAM Journal on Optimization}, 17\penalty0 (4):\penalty0
  1205--1223, 2007.

\bibitem[Bolte et~al.(2017)Bolte, Nguyen, Peypouquet, and
  Suter]{Bolte2017_eb_kl_convex}
J.~Bolte, T.~P. Nguyen, J.~Peypouquet, and B.~W. Suter.
\newblock From error bounds to the complexity of first-order descent methods
  for convex functions.
\newblock \emph{Mathematical Programming}, 165\penalty0 (2):\penalty0 471--507,
  2017.

\bibitem[Bonnel and Morgan(2006)]{Bonnel2006Semivectorial_penalty}
H.~Bonnel and J.~Morgan.
\newblock Semivectorial bilevel optimization problem: Penalty approach.
\newblock \emph{Journal of Optimization Theory and Applications}, 131\penalty0
  (3):\penalty0 365--382, 2006.

\bibitem[Bu et~al.(2017)Bu, Du, Na, Wu, and Zheng]{bu2017aISHELL}
H.~Bu, J.~Du, X.~Na, B.~Wu, and H.~Zheng.
\newblock Aishell-1: An open-source mandarin speech corpus and a speech
  recognition baseline.
\newblock In \emph{Conference of the oriental chapter of the international
  coordinating committee on speech databases and speech I/O systems and
  assessment}, pages 1--5, 2017.

\bibitem[Burke and Ferris(1993)]{Burke1993_weaksharpmin}
J.~V. Burke and M.~C. Ferris.
\newblock Weak sharp minima in mathematical programming.
\newblock \emph{SIAM Journal on Control and Optimization}, 31\penalty0
  (5):\penalty0 1340--1359, 1993.

\bibitem[Chen et~al.(2023)Chen, Fernando, Ying, and Chen]{chen2023three}
L.~Chen, H.~Fernando, Y.~Ying, and T.~Chen.
\newblock Three-way trade-off in multi-objective learning: Optimization,
  generalization and conflict-avoidance.
\newblock In \emph{Proc. Advances in Neural Information Processing Systems},
  New Orleans, LA, 2023.

\bibitem[Chen et~al.(2024{\natexlab{a}})Chen, Saif, Shen, and
  Chen]{chen2024ferero}
L.~Chen, A.~F.~M. Saif, Y.~Shen, and T.~Chen.
\newblock {FERERO}: A flexible framework for preference-guided multi-objective
  learning.
\newblock In \emph{Proc. Advances in Neural Information Processing Systems},
  Vancouver, Canada, 2024{\natexlab{a}}.

\bibitem[Chen et~al.(2024{\natexlab{b}})Chen, Shi, Jiang, and
  Wang]{chen2024penalty_simple_blo}
P.~Chen, X.~Shi, R.~Jiang, and J.~Wang.
\newblock Penalty-based methods for simple bilevel optimization under
  {Hölderian} error bounds.
\newblock In \emph{Proc. Advances in Neural Information Processing Systems},
  Vancouver, Canada, 2024{\natexlab{b}}.

\bibitem[Chen et~al.(2022)Chen, Crammer, He, Roth, and
  Su]{chen2022weighted_crosstask}
S.~Chen, K.~Crammer, H.~He, D.~Roth, and W.~J. Su.
\newblock Weighted training for cross-task learning.
\newblock In \emph{Proc. International Conference on Learning Representations},
  virtual, 2022.

\bibitem[Chen et~al.(2021)Chen, Sun, and Yin]{chen2021closing_gap}
T.~Chen, Y.~Sun, and W.~Yin.
\newblock Closing the gap: Tighter analysis of alternating stochastic gradient
  methods for bilevel problems.
\newblock In \emph{Proc. Advances in Neural Information Processing Systems},
  volume~34, pages 25294--25307, virtual, 2021.

\bibitem[Chen et~al.(2025)Chen, Zhang, Lin, Lin, Zhao, Zhang, and
  Kwok]{chen2025gradient_moo_review}
W.~Chen, X.~Zhang, B.~Lin, X.~Lin, H.~Zhao, Q.~Zhang, and J.~T. Kwok.
\newblock Gradient-based multi-objective deep learning: Algorithms, theories,
  applications, and beyond.
\newblock \emph{arXiv preprint arXiv:2501.10945}, 2025.

\bibitem[Clarke(1990)]{clarke1990optimization}
F.~H. Clarke.
\newblock \emph{Optimization and Nonsmooth Analysis}.
\newblock Classics in Applied Mathematics. Society for Industrial and Applied
  Mathematics, 1990.

\bibitem[Curtis et~al.(2023)Curtis, Liu, and Robinson]{curtis2023fair}
F.~E. Curtis, S.~Liu, and D.~P. Robinson.
\newblock Fair machine learning through constrained stochastic optimization and
  an epsilon-constraint method.
\newblock \emph{Optimization Letters}, pages 1--17, 2023.

\bibitem[Davis and Drusvyatskiy(2018)]{davis2018stochastic}
D.~Davis and D.~Drusvyatskiy.
\newblock Stochastic subgradient method converges at the rate {$O(k^{-1/4})$}
  on weakly convex functions.
\newblock \emph{arXiv preprint arXiv:1802.02988}, 2018.

\bibitem[Davis et~al.(2018)Davis, Drusvyatskiy, MacPhee, and
  Paquette]{davis2018subgradient}
D.~Davis, D.~Drusvyatskiy, K.~J. MacPhee, and C.~Paquette.
\newblock Subgradient methods for sharp weakly convex functions.
\newblock \emph{Journal of Optimization Theory and Applications}, 179:\penalty0
  962--982, 2018.

\bibitem[Dempe and Mehlitz(2020)]{Dempe2020Semivectorial_bilevel}
S.~Dempe and P.~Mehlitz.
\newblock Semivectorial bilevel programming versus scalar bilevel programming.
\newblock \emph{Optimization}, 69\penalty0 (4):\penalty0 657--679, 2020.

\bibitem[Dontchev and Rockafellar(2009)]{dontchev2009implicit}
A.~L. Dontchev and R.~T. Rockafellar.
\newblock \emph{{Implicit Functions and Solution Mappings}}.
\newblock Springer, 2009.

\bibitem[Doron and Shtern(2023)]{Doron2023_simplebilevel}
L.~Doron and S.~Shtern.
\newblock Methodology and first-order algorithms for solving nonsmooth and
  non-strongly convex bilevel optimization problems.
\newblock \emph{Mathematical Programming}, 201:\penalty0 521--558, 2023.

\bibitem[Dries and Miller(1996)]{VanDenDriesMiller1996_ominimal}
L.~Dries and C.~Miller.
\newblock Geometric categories and o-minimal structures.
\newblock \emph{Duke Mathematical Journal}, 84\penalty0 (2):\penalty0 497--540,
  1996.

\bibitem[Drusvyatskiy and Lewis(2018)]{drusvyatskiy2018_eb_qg}
D.~Drusvyatskiy and A.~S. Lewis.
\newblock Error bounds, quadratic growth, and linear convergence of proximal
  methods.
\newblock \emph{Mathematics of Operations Research}, 43\penalty0 (3):\penalty0
  919--948, 2018.

\bibitem[Ehrgott(2005)]{ehrgott_multicriteria_2005}
M.~Ehrgott.
\newblock \emph{Multicriteria optimization}.
\newblock Springer, Berlin; New York, 2nd ed edition, 2005.

\bibitem[Fliege et~al.(2019)Fliege, Vaz, and Vicente]{fliege2019complexity}
J.~Fliege, A.~I.~F. Vaz, and L.~N. Vicente.
\newblock Complexity of gradient descent for multi-objective optimization.
\newblock \emph{Optimization Methods and Software}, 34\penalty0 (5):\penalty0
  949--959, 2019.

\bibitem[Ghadimi and Wang(2018)]{ghadimi2018approximation_blo}
S.~Ghadimi and M.~Wang.
\newblock Approximation methods for bilevel programming.
\newblock \emph{arXiv preprint arXiv:1802.02246}, 2018.

\bibitem[Giang-Tran et~al.(2023)Giang-Tran, Ho-Nguyen, and
  Lee]{GiangTran2023_ircg}
K.-H. Giang-Tran, N.~Ho-Nguyen, and D.~Lee.
\newblock Projection-free methods for solving convex bilevel optimization
  problems.
\newblock \emph{arXiv preprint}, arXiv:2311.09738, 2023.

\bibitem[Giovannelli et~al.(2024)Giovannelli, Kent, and
  Vicente]{giovannelli2024bilevel_moo}
T.~Giovannelli, G.~D. Kent, and L.~N. Vicente.
\newblock Bilevel optimization with a multi-objective lower-level problem:
  Risk-neutral and risk-averse formulations.
\newblock \emph{Optimization Methods and Software}, pages 1--23, 2024.

\bibitem[Gulati et~al.(2020)Gulati, Qin, Chiu, Parmar, Zhang, Yu, Han, Wang,
  Zhang, Wu, et~al.]{gulati2020conformer}
A.~Gulati, J.~Qin, C.-C. Chiu, N.~Parmar, Y.~Zhang, J.~Yu, W.~Han, S.~Wang,
  Z.~Zhang, Y.~Wu, et~al.
\newblock Conformer: Convolution-augmented transformer for speech recognition.
\newblock \emph{arXiv preprint arXiv:2005.08100}, 2020.

\bibitem[Hardt et~al.(2018)Hardt, Ma, and Recht]{hardt2018gd_linear_dy_quasar}
M.~Hardt, T.~Ma, and B.~Recht.
\newblock Gradient descent learns linear dynamical systems.
\newblock \emph{Journal of Machine Learning Research}, 19\penalty0
  (29):\penalty0 1--44, 2018.

\bibitem[Hearn(1982)]{Hearn1982_gap_convex}
D.~Hearn.
\newblock The gap function of a convex program.
\newblock \emph{Operations Research Letters}, 1\penalty0 (2):\penalty0 67--71,
  1982.

\bibitem[Hinder et~al.(2020)Hinder, Sidford, and
  Sohoni]{hinder2020near_star_quasar}
O.~Hinder, A.~Sidford, and N.~Sohoni.
\newblock Near-optimal methods for minimizing star-convex functions and beyond.
\newblock In \emph{Proc. Conference on Learning Theory}, pages 1894--1938,
  virtual, 2020.

\bibitem[Hong et~al.(2023)Hong, Wai, Wang, and Yang]{hong2023two_time_blo}
M.~Hong, H.-T. Wai, Z.~Wang, and Z.~Yang.
\newblock A two-timescale stochastic algorithm framework for bilevel
  optimization: Complexity analysis and application to actor-critic.
\newblock \emph{SIAM Journal on Optimization}, 33\penalty0 (1):\penalty0
  147--180, 2023.

\bibitem[Hu et~al.(2023)Hu, Xian, Wu, Fan, Yin, and
  Zhao]{hu2024revisiting_ls_limit}
Y.~Hu, R.~Xian, Q.~Wu, Q.~Fan, L.~Yin, and H.~Zhao.
\newblock Revisiting scalarization in multi-task learning: A theoretical
  perspective.
\newblock In \emph{Proc. Advances in Neural Information Processing Systems},
  volume~36, New Orleans, LA, 2023.

\bibitem[Huang(2023)]{huang2023momentum_nc_blo}
F.~Huang.
\newblock On momentum-based gradient methods for bilevel optimization with
  nonconvex lower-level.
\newblock \emph{arXiv preprint arXiv:2303.03944}, 2023.

\bibitem[Ji et~al.(2021)Ji, Yang, and Liang]{ji2021bilevel}
K.~Ji, J.~Yang, and Y.~Liang.
\newblock Bilevel optimization: Convergence analysis and enhanced design.
\newblock In \emph{Proc. International Conference on Machine Learning}, pages
  4882--4892, virtual, 2021.

\bibitem[Jiang et~al.(2023)Jiang, Abolfazli, Mokhtari, and
  Hamedani]{Jiang2023_cgbio}
R.~Jiang, N.~Abolfazli, A.~Mokhtari, and E.~Y. Hamedani.
\newblock A conditional gradient-based method for simple bilevel optimization
  with convex lower-level problem.
\newblock In \emph{Proc. International Conference on Artificial Intelligence
  and Statistics}, pages 10305--10323, 2023.

\bibitem[Kamani et~al.(2021)Kamani, Forsati, Wang, and
  Mahdavi]{kamani2021pareto_fairness}
M.~M. Kamani, R.~Forsati, J.~Z. Wang, and M.~Mahdavi.
\newblock Pareto efficient fairness in supervised learning: From extraction to
  tracing.
\newblock \emph{arXiv preprint arXiv:2104.01634}, 2021.

\bibitem[Karimi et~al.(2016)Karimi, Nutini, and Schmidt]{karimi2016linear_pl}
H.~Karimi, J.~Nutini, and M.~Schmidt.
\newblock Linear convergence of gradient and proximal-gradient methods under
  the polyak-Łojasiewicz condition.
\newblock In \emph{Proceedings of the European Conference on Machine Learning
  and Principles and Practice of Knowledge Discovery in Databases (ECML PKDD)},
  pages 795--811, 2016.

\bibitem[Kiwiel(2004)]{kiwiel2004convergence}
K.~C. Kiwiel.
\newblock Convergence of approximate and incremental subgradient methods for
  convex optimization.
\newblock \emph{SIAM Journal on Optimization}, 14\penalty0 (3):\penalty0
  807--840, 2004.

\bibitem[Komatsu et~al.(2022)Komatsu, Fujita, Lee, Lee, Watanabe, and
  Kida]{ctclibrispeech}
T.~Komatsu, Y.~Fujita, J.~Lee, L.~Lee, S.~Watanabe, and Y.~Kida.
\newblock Better intermediates improve {CTC} inference.
\newblock \emph{arXiv preprint arXiv:2204.00176}, 2022.

\bibitem[Kosiba(2025)]{KOSIBA2025_suba_multif}
M.~Kosiba.
\newblock The generalized Łojasiewicz inequality for definable and subanalytic
  multifunctions.
\newblock \emph{Journal of Mathematical Analysis and Applications},
  543\penalty0 (2, Part 1), 2025.

\bibitem[Kurdyka(1998)]{Kurdyka1998_KL}
K.~Kurdyka.
\newblock On gradients of functions definable in o-minimal structures.
\newblock \emph{Annales de l'Institut Fourier}, 48\penalty0 (3):\penalty0
  769--783, 1998.

\bibitem[Kwon et~al.(2023)Kwon, Kwon, Wright, and
  Nowak]{kwon2023penalty_envelope}
J.~Kwon, D.~Kwon, S.~Wright, and R.~Nowak.
\newblock On penalty methods for nonconvex bilevel optimization and first-order
  stochastic approximation.
\newblock \emph{arXiv preprint arXiv:2309.01753}, 2023.

\bibitem[Lee and Valiant(2016)]{lee2016optimizing_star}
J.~C. Lee and P.~Valiant.
\newblock Optimizing star-convex functions.
\newblock In \emph{Proc. IEEE Annual Symposium on Foundations of Computer
  Science}, pages 603--612, 2016.

\bibitem[Liao et~al.(2024)Liao, Ding, and Zheng]{liao2024_eb_pl}
F.-Y. Liao, L.~Ding, and Y.~Zheng.
\newblock Error bounds, {PL} condition, and quadratic growth for weakly convex
  functions, and linear convergences of proximal point methods.
\newblock In \emph{Proceedings of the 6th Annual Learning for Dynamics \&
  Control Conference}, volume 242, pages 993--1005, 2024.

\bibitem[Lin et~al.(2019)Lin, Zhen, Li, Zhang, and Kwong]{lin2019pareto}
X.~Lin, H.-L. Zhen, Z.~Li, Q.-F. Zhang, and S.~Kwong.
\newblock Pareto multi-task learning.
\newblock In \emph{Proc. Advances in Neural Information Processing Systems},
  Vancouver, Canada, Dec. 2019.

\bibitem[Liu et~al.(2022{\natexlab{a}})Liu, Ye, Wright, Stone, and
  Liu]{liu2022bome}
B.~Liu, M.~Ye, S.~Wright, P.~Stone, and Q.~Liu.
\newblock {BOME! Bilevel} optimization made easy: A simple first-order
  approach.
\newblock In \emph{Proc. Advances in Neural Information Processing Systems},
  volume~35, pages 30612--30625, New Orleans, LA, 2022{\natexlab{a}}.

\bibitem[Liu et~al.(2021{\natexlab{a}})Liu, Liu, Yuan, Zeng, and
  Zhang]{liu21_interior_blo}
R.~Liu, X.~Liu, X.~Yuan, S.~Zeng, and J.~Zhang.
\newblock A value-function-based interior-point method for non-convex bi-level
  optimization.
\newblock In M.~Meila and T.~Zhang, editors, \emph{Proc. International
  Conference on Machine Learning}, volume 139 of \emph{Proceedings of Machine
  Learning Research}, pages 6882--6892, virtual, 18--24 Jul 2021{\natexlab{a}}.

\bibitem[Liu et~al.(2021{\natexlab{b}})Liu, Liu, Zeng, and
  Zhang]{liu2021nonconvex_blo}
R.~Liu, Y.~Liu, S.~Zeng, and J.~Zhang.
\newblock Towards gradient-based bilevel optimization with non-convex followers
  and beyond.
\newblock In \emph{Proc. Advances in Neural Information Processing Systems},
  volume~34, pages 8662--8675, virtual, 2021{\natexlab{b}}.

\bibitem[Liu et~al.(2022{\natexlab{b}})Liu, Gao, Zhang, Meng, and
  Lin]{liu2022survey_blo}
R.~Liu, J.~Gao, J.~Zhang, D.~Meng, and Z.~Lin.
\newblock Investigating bi-level optimization for learning and vision from a
  unified perspective: A survey and beyond.
\newblock \emph{IEEE Transactions on Pattern Analysis and Machine
  Intelligence}, 44\penalty0 (12):\penalty0 10045--10067, 2022{\natexlab{b}}.

\bibitem[Liu et~al.(2024)Liu, Liu, Yao, Zeng, and Zhang]{liu2024moreau}
R.~Liu, Z.~Liu, W.~Yao, S.~Zeng, and J.~Zhang.
\newblock Moreau envelope for nonconvex bi-level optimization: A single-loop
  and {Hessian}-free solution strategy.
\newblock \emph{arXiv preprint arXiv:2405.09927}, 2024.

\bibitem[Liu and Vicente(2021)]{liu2021stochastic}
S.~Liu and L.~N. Vicente.
\newblock The stochastic multi-gradient algorithm for multi-objective
  optimization and its application to supervised machine learning.
\newblock \emph{Annals of Operations Research}, pages 1--30, 2021.

\bibitem[Lu(2023)]{lu2023slm_blo}
S.~Lu.
\newblock {SLM}: A smoothed first-order lagrangian method for structured
  constrained nonconvex optimization.
\newblock In \emph{Proc. Advances in Neural Information Processing Systems},
  New Orleans, LA, 2023.

\bibitem[Luo et~al.(1996{\natexlab{a}})Luo, Pang, and
  Ralph]{luo1996_equilibrium}
Z.-Q. Luo, J.-S. Pang, and D.~Ralph.
\newblock \emph{Mathematical Programs with Equilibrium Constraints}.
\newblock Cambridge University Press, 1996{\natexlab{a}}.

\bibitem[Luo et~al.(1996{\natexlab{b}})Luo, Pang, Ralph, and
  Wu]{Luo1996_exact_penalty_mpec}
Z.-Q. Luo, J.-S. Pang, D.~Ralph, and S.-Q. Wu.
\newblock Exact penalization and stationarity conditions of mathematical
  programs with equilibrium constraints.
\newblock \emph{Mathematical Programming}, 75:\penalty0 19--76,
  1996{\natexlab{b}}.

\bibitem[Mahapatra and Rajan(2020)]{mahapatra2020multi}
D.~Mahapatra and V.~Rajan.
\newblock Multi-task learning with user preferences: Gradient descent with
  controlled ascent in {Pareto} optimization.
\newblock In \emph{Proc. International Conference on Machine Learning},
  virtual, 2020.

\bibitem[Mejía-De-Dios et~al.(2023)Mejía-De-Dios, Rodríguez-Molina, and
  Mezura-Montes]{MOBLO_Survey2023}
J.-A. Mejía-De-Dios, A.~Rodríguez-Molina, and E.~Mezura-Montes.
\newblock Multiobjective bilevel optimization: A survey of the
  state-of-the-art.
\newblock \emph{IEEE Transactions on Systems, Man, and Cybernetics: Systems},
  53\penalty0 (9):\penalty0 5478--5490, September 2023.

\bibitem[Merchav and Sabach(2023)]{Merchav2023_bisg}
R.~Merchav and S.~Sabach.
\newblock Convex bi-level optimization problems with nonsmooth outer objective
  function.
\newblock \emph{SIAM Journal on Optimization}, 33\penalty0 (4):\penalty0
  3114--3142, 2023.

\bibitem[Miettinen(1998)]{miettinen_nonlinear_1998}
K.~Miettinen.
\newblock \emph{Nonlinear {Multiobjective} {Optimization}}, volume~12.
\newblock Springer US, Boston, MA, 1998.

\bibitem[Momma et~al.(2022)Momma, Dong, and Liu]{momma2022multi}
M.~Momma, C.~Dong, and J.~Liu.
\newblock A multi-objective/multi-task learning framework induced by {Pareto}
  stationarity.
\newblock In \emph{Proc. International Conference on Machine Learning},
  Baltimore, MD, 2022.

\bibitem[Nesterov(2005)]{Nesterov2005_smooth_min}
Y.~Nesterov.
\newblock Smooth minimization of non-smooth functions.
\newblock \emph{Mathematical Programming}, 103\penalty0 (1):\penalty0 127--152,
  2005.

\bibitem[Oord et~al.(2018)Oord, Li, and Vinyals]{oord2018representation}
A.~v.~d. Oord, Y.~Li, and O.~Vinyals.
\newblock Representation learning with contrastive predictive coding.
\newblock \emph{arXiv preprint arXiv:1807.03748}, 2018.

\bibitem[Osyczka(1984)]{Osyczka1984Multicriterion}
A.~Osyczka.
\newblock \emph{Multicriterion optimization in engineering with FORTRAN
  programs}.
\newblock Ellis Horwood, 1984.

\bibitem[Panayotov et~al.(2015)Panayotov, Chen, Povey, and
  Khudanpur]{panayotov2015librispeech}
V.~Panayotov, G.~Chen, D.~Povey, and S.~Khudanpur.
\newblock Librispeech: an {ASR} corpus based on public domain audio books.
\newblock In \emph{Proc. International Conference on Acoustics, Speech and
  Signal Processing}, pages 5206--5210, 2015.

\bibitem[Roy et~al.(2023)Roy, So, and Ma]{roy2023opt_pareto_set}
A.~Roy, G.~So, and Y.-A. Ma.
\newblock Optimization on {Pareto} sets: On a theory of multi-objective
  optimization.
\newblock \emph{arXiv preprint arXiv:2308.02145}, 2023.

\bibitem[Saif et~al.(2024)Saif, Chen, Cui, Lu, Kingsbury, and
  Chen]{saif24_interspeech}
A.~F.~M. Saif, L.~Chen, X.~Cui, S.~Lu, B.~Kingsbury, and T.~Chen.
\newblock {M2ASR}: Multilingual multi-task automatic speech recognition via
  multi-objective optimization.
\newblock In \emph{Interspeech 2024}, pages 1240--1244, 2024.

\bibitem[Samadi et~al.(2024)Samadi, Burbano, and Yousefian]{Samadi2024_rapm}
S.~Samadi, D.~Burbano, and F.~Yousefian.
\newblock Achieving optimal complexity guarantees for a class of bilevel convex
  optimization problems.
\newblock In \emph{Proceedings of the 2024 American Control Conference (ACC)},
  pages 2206--2211, 2024.

\bibitem[Sener and Koltun(2018)]{sener2018multi}
O.~Sener and V.~Koltun.
\newblock {Multi-task learning as multi-objective optimization}.
\newblock In \emph{Proc. Advances in Neural Information Processing Systems},
  Montreal, Canada, Dec. 2018.

\bibitem[Shen et~al.(2023)Shen, Xiao, and Chen]{shen2023penalty}
H.~Shen, Q.~Xiao, and T.~Chen.
\newblock On penalty-based bilevel gradient descent method.
\newblock \emph{arXiv preprint arXiv:2302.05185}, 2023.

\bibitem[Stackelberg(1952)]{stackelberg1952market}
H.~F.~V. Stackelberg.
\newblock \emph{Market Structure and Equilibrium}.
\newblock Princeton University Press, 1952.

\bibitem[Tanabe et~al.(2019)Tanabe, Fukuda, and Yamashita]{tanabe2018proximal}
H.~Tanabe, E.~H. Fukuda, and N.~Yamashita.
\newblock Proximal gradient methods for multiobjective optimization and their
  applications.
\newblock \emph{Computational Optimization and Applications}, 72\penalty0
  (2):\penalty0 339--361, 2019.

\bibitem[Tanabe et~al.(2022)Tanabe, Fukuda, and Yamashita]{tanabe2022new}
H.~Tanabe, E.~H. Fukuda, and N.~Yamashita.
\newblock New merit functions for multiobjective optimization and their
  properties.
\newblock \emph{arXiv preprint arXiv:2010.09333}, 2022.

\bibitem[Vicente and Calamai(1994)]{vicente1994_blo_bib_review}
L.~N. Vicente and P.~H. Calamai.
\newblock Bilevel and multilevel programming: A bibliography review.
\newblock \emph{Journal of Global Optimization}, 5\penalty0 (3):\penalty0
  291--306, 1994.

\bibitem[Wang et~al.(2024)Wang, Zhang, Meng, Sun, Ma, and
  Chen]{wang2024convergence}
B.~Wang, H.~Zhang, Q.~Meng, R.~Sun, Z.-M. Ma, and W.~Chen.
\newblock On the convergence of adam under non-uniform smoothness: Separability
  from sgdm and beyond.
\newblock \emph{arXiv preprint arXiv:2403.15146}, 2024.

\bibitem[Wang and Wibisono(2023)]{wang2023continuized_acc_quasar}
J.-K. Wang and A.~Wibisono.
\newblock Continuized acceleration for quasar convex functions in non-convex
  optimization.
\newblock In \emph{Proc. International Conference on Learning Representations},
  Kigali, Rwanda, 2023.

\bibitem[Xiao et~al.(2023)Xiao, Lu, and Chen]{xiao2023alternating}
Q.~Xiao, S.~Lu, and T.~Chen.
\newblock An alternating optimization method for bilevel problems under the
  {Polyak-Łojasiewicz} condition.
\newblock In \emph{Proc. Advances in Neural Information Processing Systems},
  New Orleans, LA, 2023.

\bibitem[Ye et~al.(2021)Ye, Lin, Yue, Guo, Xiao, and Zhang]{ye2021MOML}
F.~Ye, B.~Lin, Z.~Yue, P.~Guo, Q.~Xiao, and Y.~Zhang.
\newblock Multi-objective meta learning.
\newblock In \emph{Proc. Advances in Neural Information Processing Systems},
  volume~34, pages 21338--21351, virtual, 2021.

\bibitem[Ye(2000)]{Ye1999_cq_eb_blo}
J.~J. Ye.
\newblock Constraint qualifications and necessary optimality conditions for
  optimization problems with variational inequality constraints.
\newblock \emph{SIAM Journal on Optimization}, 10\penalty0 (4):\penalty0
  943--962, 2000.

\bibitem[Ye et~al.(1995)Ye, Zhu, and Zhu]{Ye1995_exact_penalty_blo}
J.~J. Ye, D.~L. Zhu, and Q.~J. Zhu.
\newblock Exact penalization and necessary optimality conditions for
  generalized bilevel programming problems.
\newblock \emph{SIAM Journal on Optimization}, 4\penalty0 (3):\penalty0
  481--507, 1995.

\bibitem[Ye and Liu(2022)]{ye2022pareto_nav}
M.~Ye and Q.~Liu.
\newblock Pareto navigation gradient descent: a first-order algorithm for
  optimization in {Pareto set}.
\newblock In \emph{Uncertainty in Artificial Intelligence}, pages 2246--2255,
  2022.

\bibitem[Zhang et~al.(2024)Zhang, Lin, and Zhang]{zhang2024pmgda}
X.~Zhang, X.~Lin, and Q.~Zhang.
\newblock {PMGDA}: A preference-based multiple gradient descent algorithm.
\newblock \emph{arXiv preprint arXiv:2402.09492}, 2024.

\end{thebibliography}
}


\clearpage
\appendix
\onecolumn

\begin{center}
{\Large \bf Appendix}
\end{center}

\allowdisplaybreaks

Throughout the paper, we assume $f_m(x), m = 0, \ldots, M$ are twice continuously differentiable and bounded below, and the minimizer of $\lambda^\top F(x)$ exists for all $\lambda\in \Delta^M$.

For notation simplicity, we use $\mathrm{LSE}_\tau: \R^M \to \R$ to denote the Log-sum-exp function with parameter $\tau$.
We say a function is $\mu$-(weakly) convex, with $\mu \in \R$. If $\mu>0$, the function is strongly convex, if $\mu = 0$, the function is convex, and if $\mu<0$, the function is weakly convex.

\paragraph{Organization of the appendix.}
We organize the proof in the appendix as follows.

In Appendix~\ref{sec_app:related_work}, we discuss additional related work
on (semivectorial) bilevel optimization
and limitation of linear scalarization and preference as constraint formulations.
\begin{itemize}
  \item Appendix~\ref{sub_app:blo_work}: additional related work on bilevel optimization
  \item Appendix~\ref{sub_app:limitation_ls_kkt}: 
  limitations of linear scalarization and preference as constraint
\end{itemize}

In Appendix~\ref{sec_app:proof_property}, we prove the basic properties of the merit function.
\begin{itemize}
  \item Appendix~\ref{sub_app:cont_h_v}: continuity of the merit function 
  \item Appendix~\ref{sub_app:proof_v_weak_po}: 
  relations of $v_{l,\tau}$ and weak Pareto optimality
\end{itemize}

In Appendix~\ref{sec_app:subanalytic_property}, we discuss (global) subanalyticity and prove related properties such as the H\"{o}lderian error bound (HEB) of the merit function.
\begin{itemize}
  \item Appendix~\ref{sub_app:proof_global_suba_v}: proof of global subanalyticity of $v_{l,\tau}$
  \item Appendix~\ref{sub_app:examples_global_suba}: examples of globally subanalytic functions and their HEB
  \item Appendix~\ref{sub_app:relations_heb_qg}: relations of proximal error bound (EB), proximal KL, and HEB
\end{itemize}

In Appendix~\ref{sec_app:proof_relation}, we prove the relations of different formulations.
\begin{itemize}
  \item Appendix~\ref{sub_app:proof_relation_asymptotic}: global solutions relation in the asymptotic setting
  \item Appendix~\ref{sub_app:global_relation}: $\epsilon$-global solutions relation
  \item Appendix~\ref{sub_app:local_relation}: local solutions relation
  \item Appendix~\ref{sub_app:stationary_relation}: $\epsilon$-stationary solutions relation
\end{itemize}

In Appendix~\ref{sec_app:proof_convergence}, we prove the convergence of the proposed algorithm.

In Appendix~\ref{sec_app:experiment}, we provide implementation details and additional experiment results.

\section{Extended discussion of related works}
\label{sec_app:related_work}

In this section, we provide an extended review of recent works that are closely related to ours.

\subsection{Bilevel optimization}
\label{sub_app:blo_work}

Bilevel optimization (BLO) is a classical problem that dates back to~\cite{stackelberg1952market,vicente1994_blo_bib_review,luo1996_equilibrium}.
Gradient-based approaches with non-asymptotic convergence analysis and applications to machine learning were studied in e.g.,~\cite{ghadimi2018approximation_blo,chen2021closing_gap,ji2021bilevel,liu2022survey_blo,hong2023two_time_blo}.
These works focus on problems with (strongly) convex lower-level (LL) objectives.
Similarly, in simple bilevel optimization with shared optimization variable
in the upper- and lower-level objectives, 
a line of works focus on convex LL objectives, 
including e.g.,~\cite{Jiang2023_cgbio,Merchav2023_bisg,GiangTran2023_ircg,Samadi2024_rapm,Doron2023_simplebilevel}.
For nonconvex LL objectives, a pessimistic algorithm with asymptotic convergence guarantee was proposed in~\cite{liu2021nonconvex_blo}, 
a momentum-based algorithm with non-asymptotic analysis was proposed in~\cite{huang2023momentum_nc_blo},
the stationary metric was studied for BLO with lower-level PL objective, and an alternating descent algorithm  with non-asymptotic analysis  was proposed in~\cite{xiao2023alternating}.
In Table~\ref{tab:comparison_methods_simple_blo}, 
we provide a summary of the bilevel optimization works 
with possibly nonconvex lower-level objectives
which may satisfy the H\"{o}lderian error bound (HEB).

\paragraph{Semivectorial BLO.}
OPS can be seen as a semivectorial BLO problem, where the bilevel program 
has a vector-valued lower-level (LL) objective
and a scalar-valued upper-level (UL) objective.
To solve such problems, one straightforward approach is to convert the LL vector-valued objective to a scalar-valued objective through scalarization, and to optimize the scalarization parameter in the upper level, see, e.g.,~\cite{roy2023opt_pareto_set}. However, it has been shown that this reformulation could induce additional local minimizers or stationary solutions~\cite{Dempe2020Semivectorial_bilevel,Benko2021Implicit_opt}. Furthermore, this reformulation might require stronger constraint qualifications than the original problem~\cite{Benko2021Implicit_opt}.
Alternatively, penalty-based reformulations have been considered~\cite{Bonnel2006Semivectorial_penalty}, where the penalty function is defined as the maximum improvement amount of the vector-valued objective.
However, no practical implementation, or relation of solutions to the original problem in nonconvex settings, 
or non-asymptotic convergence analysis are provided for the reformulation.
Indeed, according to~\cite{giovannelli2024bilevel_moo}, 
very limited works study the computational algorithms 
for semivectorial BLO problem.
In~\cite{giovannelli2024bilevel_moo}, deterministic and stochastic risk-neutral and risk-averse algorithms are proposed, under the assumption that the LL objective is strictly convex w.r.t. the LL variable, and requiring the 
second-order derivative of the objective. 

Another line of research study the problem
with vector-valued upper-level (UL) objective, 
and scalar-valued LL objective~\cite{ye2021MOML}.
It is sometimes also referred to as the multi-objective BLO problem.
For a more detailed review, see a survey~\cite{MOBLO_Survey2023}
and the references therein.

\begin{table}
\centering
\fontsize{8}{9}\selectfont 
\caption{Comparison with existing methods for 
(simple) bilevel optimization with lower-level 
scalar objective, ``SC'' and ``C'' represent ``strongly convex'' and ``convex'', respectively; ``comp'' represents ``compact set'';
``Lip'' represents ``Lipschitz continuous''.
For non-simple bilevel optimization problem, 
the lower-level and upper-level properties are all w.r.t. 
the lower-level variable
for a meaningful comparison.
The lower-level objective in our problem is $v_{l,\tau}(x) + \tau\ln M$.
}
\label{tab:comparison_methods_simple_blo}
\begin{tabular}{c|cccc}
\toprule
Method  & lower-level HEB &other lower-level properties & upper-level & first-order  \\
\midrule
\multicolumn{5}{c}{non-simple bilevel optimization}\\
\hline
IAPTT-GM~\cite{liu2021nonconvex_blo} 
& - & smooth, comp &smooth, comp &\xmark\\
\hline
BOME~\cite{liu2022bome}  & $\eta=2$ &PL, Lip, smooth &Lip, smooth & \cmark\\
\hline
\multirow{2}{*}{PBGD~\cite{shen2023penalty}}
& \multirow{2}{*}{$\eta = 2$}
& PL, smooth &Lip, smooth
& \multirow{2}{*}{\cmark} \\
& & C, smooth &Lip, smooth\\
\hline
GALET~\cite{xiao2023alternating}
  &$\eta = 2$
&PL, smooth &Lip, smooth 
& \xmark \\
\hline
AGILS~\cite{bai2024_agils} 
& $\eta \geq 1$ & KL, weakly C (composite) & smooth &\cmark\\
\hline
\multirow{2}{*}{MEHA~\cite{liu2024moreau}} 
& \multirow{2}{*}{-}
& smooth & smooth
& \multirow{2}{*}{\cmark} \\
& &weakly C (composite) &smooth\\
\hline
SLM~\cite{lu2023slm_blo}  & $\eta=2$ &PL, smooth &smooth, comp &\cmark \\
\hline
\multicolumn{5}{c}{simple bilevel optimization}\\
\hline
\multirow{2}{*}{CG-BiO~\cite{Jiang2023_cgbio}} 
& \multirow{2}{*}{$\eta\geq 1$}
& C, smooth & C, smooth 
& \multirow{2}{*}{\cmark} \\
& &C, smooth &non-C, smooth
\\
\hline
R-APM~\cite{Samadi2024_rapm} 
& $\eta = 1$ & C, composite &C, smooth &\cmark \\
\hline 
\multirow{3}{*}{PB-APG~\cite{chen2024penalty_simple_blo}} 
& \multirow{3}{*}{$\eta\geq 1$}
& C, composite & C, composite 
& \multirow{3}{*}{\cmark} \\
& &C, composite &SC, composite\\
& &nonsmooth, Lip & nonsmooth, Lip
\\
\hline
\multirow{2}{*}{FOOPS (ours)} 
& \multirow{2}{*}{\makecell{
$\eta > 0$,\\
(modified by $\theta = \frac{\eta_p}{\eta} \geq \frac{1}{\eta}$)}}
& subanalytic (provable) &locally Lip & \multirow{2}{*}{\cmark} \\
& &KL (provable)   &locally Lip \\
\bottomrule
\end{tabular}
\end{table}

\subsection{Limitation of linear scalarization 
and preference as constraint}
\label{sub_app:limitation_ls_kkt}

Besides using empirical results on Example~\ref{exmp:kkt_suboptimal}, we also provide theoretical justifications to show the limitations of LS and KKT solutions for preference-guided MOL.

\paragraph{Limitation of linear scalarization.}
We first discuss the limitation of 
linear scalarization (LS) for preference-guided MOL.
It is known that LS is not good at handling 
nonconvex Pareto front.
Prior works have shown that 
the optimality condition of LS is not a necessary condition
for Pareto optimality,
see e.g.,~{\citep[Proposition~3.3]{athan_quasi_1994}}. 
As a result, the solution set of LS, even by enumerating all possible weights of objectives, does not include all Pareto optimal solutions, and thus it does not include Pareto optimal solutions under certain preferences.

Below we provide the proof of Proposition~\ref{prop:ls_fail_exmp},
which shows that under certain initializations, 
gradient descent on the LS objective does not converge to 
a preferred Pareto optimal solution in Example~\ref{exmp:kkt_suboptimal}.
\begin{proof}[Proof of Proposition~\ref{prop:ls_fail_exmp}]
In Example~\ref{exmp:kkt_suboptimal}, 
there exists a solution $x^* \in (-1, 1)$, which is an optimal solution to 
both~\eqref{eq:intro_bi_vector} and~\eqref{eq:constrained_vo}. 
And there exists $\lambda = [\lambda_1, \lambda_2]^\top \in \Delta^2$ 
with $\lambda_2 = \lambda_1 \frac{(1-x^*)e^{-(x^* - 1)^2}}{(1+x^*)e^{-(x^* + 1)^2}}$ that $\nabla F(x^*) \lambda = 0$.
To show this, let $c $ denote a positive constant, note that
\begin{align*}
\nabla F(x) \lambda 
=& 2\lambda_1 e^{-(x-1)^2} (x - 1)
+ 2\lambda_2 e^{-(x+1)^2} (x + 1)  \\
=& c \Big( 2 e^{-(x-1)^2} (x - 1) (1+x^*)e^{-(x^* + 1)^2} 
+ 2 e^{-(x+1)^2} (x + 1) (1-x^*)e^{-(x^* - 1)^2} \Big) \\
=& 2c  (x + 1) (x^* + 1) e^{-(x+1)^2 -(x^* + 1)^2}
\Big(  e^{ 4x} \frac{x - 1}{x + 1} 
-  \frac{ x^* - 1 }{x^* + 1} e^{ 4 x^*  } \Big) .
\numberthis 
\end{align*}
Let $r(x) = \frac{x - 1}{x + 1} e^{ 4x} $.
Then $ r(x) > r(x^*) \iff \nabla F(x) \lambda > 0$
and vice versa for $ r(x) < r(x^*)$.
The above equation implies that when $r(x') = r(x^*)$, 
$\nabla F(x') \lambda = 0$.
Therefore, $\nabla F(x^*) \lambda  = 0$.
Also observing that there exists different points 
$-1 < x_1 < x^* < x_2 < 1 $ that
\begin{align}
  r(x_1) = r(x_2) = r(x^*) .
\end{align}
This means that $x_1, x_2$ are all stationary points of $\lambda^\top F(x)$.
Furthermore, \\
1. for $x' \in (x_1, x^*) \cup (x_2, 1)$, $r(x') > r(x^*)$, 
thus $\nabla F(x') \lambda > 0$, then gradient descent (GD)  on $\lambda^\top F(x)$ starting from $x'\in (x_1, x^*)$ with sufficiently small step size converges to $x_1$, 
and it converges to $x_2$ if starting from  $x'\in (x_2, 1)$; \\
2. for $x' \in (-1, x_1) \cup (x^*, x_2)$, $r(x') < r(x^*)$, 
thus $\nabla F(x') \lambda < 0$, then GD  on $\lambda^\top F(x)$ starting from $x'\in (-1, x_1)$ with sufficiently small step size converges to $x_1$, 
and it converges to $x_2$ if starting from $x'\in (x^*, x_2)$.

This proves that they will not converge to $x^*$.
\end{proof}

\paragraph{Limitation of preference as constraint.}
We then discuss in more detail of the limitation
of modeling preference by constraints.
We verify Proposition~\ref{prop:kkt_not_ps} by constructing an example with a KKT but non-Pareto stationary solution. Using Example~\ref{exmp:kkt_suboptimal}, and noticing that one solution that PMTL converges to, denoted as $x^*$, has objective value that is approximately $F(x^*) \approx [0.80; 0.99]$, which satisfies the feasibility condition that $H(x^*) = 0$.
Furthermore, its gradient at $x^*$ can be computed approximately as
\begin{subequations}
\begin{align*}
\nabla f_1(x^*) \approx &
0.5062 ; \numberthis \\
\nabla f_2(x^*) \approx & 0.0002 . \numberthis 
\end{align*}      
\end{subequations}
Clearly, $x^*$ is not Pareto stationary.
By the KKT stationarity condition, we further have
\begin{align*}
&\nabla F(x) \lambda_f + \nabla H(x)\lambda_h 
= \nabla F(x) \Big(\lambda_f + 
\begin{bmatrix}
5 \\ -4
\end{bmatrix}\lambda_h \Big) \\
=& \nabla F(x) \Big(\lambda_f + 
\begin{bmatrix}
5\lambda_h \\ -4\lambda_h
\end{bmatrix} \Big) = 0, 
~~\text{for some}~~\lambda_f \in \Delta^2,
\lambda_h \in \R . \numberthis
\end{align*}
It can be verified that  
$0.5062(\lambda_{f,1} + 5\lambda_h)
+0.0002(1 - \lambda_{f,1} - 4\lambda_h ) = 0$
has solutions, e.g., $\lambda_f = [0; 1]$, 
$\lambda_h \approx -7.9 \times 10^{-5}$.
Therefore, $x^*$ is a KKT point.

We use another example with strongly convex objectives to prove  Proposition~\ref{prop:kkt_not_ps} that a KKT point 
to~\eqref{eq:constrained_vo} is not necessarily 
Pareto stationary.
\begin{example}\label{exmp:sc_kkt_not_ps}
Let $M=2, q = 2$, and $M_g=0, M_h=1$. Let ${\cal X} = \R^2$, and $x = [x_1; x_2]$. The objective $F$ and constraint $H$ is defined as
\begin{subequations}
\begin{align}
F(x) =& \big((x_1-1)^2+ x_2^2,  ~~0.5x_1^2+x_2^2 \big) \\
H(x) =& 9f_1(x) - 8f_2 (x)
\end{align}    
\end{subequations}
\end{example}

\begin{proof}[Proof of Proposition~\ref{prop:kkt_not_ps}]
\label{proof:kkt_not_ps}
In Example~\ref{exmp:sc_kkt_not_ps}, the gradient of $F$ can be computed by 
\begin{align}
\nabla F(x) = 
\begin{bmatrix}
2(x_1 - 1 )  & x_1 \\
2 x_2  & 2x_2 
\end{bmatrix}.
\end{align}
For $x = [3; 0]$, $\nabla F(x) = \begin{bmatrix}
4  & 3 \\0  & 0 
\end{bmatrix}$.
Let $\Delta^M$ denote the $(M-1)$-simplex.
Apparently, there exists no $\lambda\in \Delta^2$ such that
$\nabla F(x) \lambda = 0$. Therefore, $x$ is not Pareto stationary.

Then we check whether $x$ satisfies the KKT condition.
First, it satisfies the feasiblity condition since 
$H(x) =0 $.
Second, by invoking the KKT stationarity condition, 
for some $\lambda_f \in \Delta^2,
\lambda_h \in \R$, we have 
\begin{align*}
&\nabla F(x) \lambda_f + \nabla H(x)\lambda_h 
= \nabla F(x) \Big(\lambda_f + 
\begin{bmatrix}
9 \\ -8
\end{bmatrix}\lambda_h \Big) 
= \begin{bmatrix}
4  & 3 \\0  & 0 
\end{bmatrix} \Big(\lambda_f + 
\begin{bmatrix}
9\lambda_h \\ -8\lambda_h
\end{bmatrix} \Big) = 0
. \numberthis
\end{align*}
It can then be verified that the above holds true
when $\lambda_f = [0, 1] \in \Delta^2$, 
and $\lambda_h = -0.25$.

Therefore, $x$ is a KKT point but not a Pareto stationary point.
The proof is complete.
\end{proof}

\section{Proof of the properties of the merit functions} 
\label{sec_app:proof_property}
For convenience, we define the merit function $u_{l}(x)$ and restate 
the smoothed merit function $v_{l,\tau}(x)$ below. 
\begin{align}
\label{eq:ul_app}
u_{l}(x) \coloneqq &
\max_{y\in {\cal X}} \min_{m\in [M]}
\Big\{ f_m(x) -f_m(y) 
 - \frac{l}{2}\|x - y\|^2  \Big\} \\
\label{eq:vl_app}
v_{l,\tau}(x) \coloneqq &
- \min_{y\in {\cal X}} \Big\{ \tau\ln \Big(\sum_{m=1}^M e^{\frac{f_m(y) - f_m(x)}{\tau} }
\Big) + \frac{l}{2}\|x - y\|^2 \Big\} .
\end{align}
Correspondingly, we define $h_{l,\tau}(x,y)$ below for analysis.
Note that $v_{l,\tau}(x) = -\min_{y\in {\cal X}} h_{l,\tau}(x,y)$.
\begin{align}
\label{eq:h_ltau}
h_{l,\tau}(x,y) 
\coloneqq & \tau\ln \Big(\sum_{m=1}^M e^{\frac{f_m(y) - f_m(x)}{\tau} } \Big) + \frac{l}{2}\|x - y\|^2 .
\end{align}

\subsection{Auxiliary lemmas}
\label{sub_app:aux_property_v_po}

\begin{lemma}[Restatement of~{\citep[Theorems~3.1 and 3.3]{tanabe2022new}}]
\label{lemma:u_l0}
For $l\geq 0$, consider the merit function $u_{l}(x)$ defined in~\eqref{eq:ul_app}. 
Then $u_{l}(x) \geq 0$ for $l\geq 0$. 
$u_{0}(x)=0$ if and only if $x$ is weakly Pareto optimal.
For $l > 0$,
$x$ is weakly Pareto optimal implies $u_{l}(x)=0$;
furthermore, if $f_m(x)$ is convex for all $m\in [M]$,
then $u_{l}(x)=0$ implies $x$ is weakly Pareto optimal.
\end{lemma}

\begin{proposition}[Smoothness implies weak convexity]
\label{prop:smooth_imply_weak_convex}
If a locally Lipschitz function $f$ is $\ell_{f,1}$-smooth, 
then it is also $-\ell_{f,1}$-weakly convex.
\end{proposition}

\begin{lemma}[Log-sum-exp function preserves weak convexity]
\label{lemma:log_sum_exp_preserve_weak_convexity}
Let $f_m(x), m\in [M]$ be weakly convex with modulus $\mu_m \in \R$. 
Let $\bar{\mu} = \min_{m\in [M]}~\mu_m$.
Then $\ln \big( \sum_{m=1}^{M} e^{f_m(x)} \big)$ is weakly convex with modulus $\bar{\mu}$.
\end{lemma}

\begin{proof}[Proof of Lemma~\ref{lemma:log_sum_exp_preserve_weak_convexity}]
By definition, and since $\bar{\mu} = \min_{m\in [M]}~\mu_m$, we have $f_m(x) - \frac{\bar{\mu}}{2} \|x\|^2$ is convex for all $m\in [M]$.
Also because the Log-sum-exp function preserves convexity, we have that
$\ln \big( \sum_{m=1}^{M} e^{f_m(x) - \frac{\bar{\mu}}{2} \|x\|^2} \big)$ is convex.
Further rearranging this function, we have
\begin{align}
\ln \Big( \sum_{m=1}^{M} e^{f_m(x) - \frac{\bar{\mu}}{2} \|x\|^2} \Big)
= \ln \Big( e^{-\frac{\bar{\mu}}{2} \|x\|^2} \big( \sum_{m=1}^{M} e^{f_m(x)} \big) \Big)
= \ln \Big( \sum_{m=1}^{M} e^{f_m(x)} \Big) - \frac{\bar{\mu}}{2} \|x\|^2
\end{align}
which is convex.
By definition, this implies that 
$\ln \Big( \sum_{m=1}^{M} e^{f_m(x)} \Big)$ 
is weakly convex with modulus $\bar{\mu}$.
The proof is complete.
\end{proof}

\begin{corollary}
\label{crlr:h_ltau_unique_y}
If $f_m(x), m\in [M]$ is weakly convex with modulus $\mu_m\in \R$, and $l + \min_{m\in [M]}\mu_m > 0$, then $h_{l,\tau}(x,y)$ defined in~\eqref{eq:h_ltau} is strictly convex w.r.t. $y$, and the solution to $\min_{y\in {\cal X}} h_{l,\tau}(x,y)$ is a singleton.
\end{corollary}

\begin{proof}[Proof of Corollary~\ref{crlr:h_ltau_unique_y}]
By the $\mu_m$-weak convexity, $\frac{1}{\tau} f_m(y)$ is $\frac{\bar{\mu}}{\tau}$-weakly convex w.r.t. $y$.
Therefore, combining with Lemma~\ref{lemma:log_sum_exp_preserve_weak_convexity}, 
the function $\tau\ln \Big(\sum_{m=1}^M e^{\frac{f_m(y) - f_m(x)}{\tau} } \Big)$ is $\bar{\mu}$-weakly convex w.r.t. $y$, where $\bar{\mu} = \min_{m\in [M]}~\mu_m$.
Since $l + \min_{m\in [M]}\mu_m = l + \bar{\mu} > 0$, $h_{l,\tau}(x,y)$ is strictly convex w.r.t. $y$, and the solution to $\min_{y\in {\cal X}} h_{l,\tau}(x,y)$ is unique.
The proof is complete.
\end{proof}

\begin{lemma}\label{lemma:exist_min_v}
If $f_m(x), m\in [M]$ is continuous and weakly convex with modulus $\mu_m$,
and $l \geq -\min_{m\in [M]}\mu_m$, then there exists $x\in {\cal X}$ such that $v_{l,\tau}(x) = -\tau \ln M$.
\end{lemma}

\begin{proof}[Proof of Lemma~\ref{lemma:exist_min_v}]
From Corollary~\ref{crlr:h_ltau_unique_y}, since $l \geq -\min_{m\in [M]}\mu_m$,  $h_{l,\tau}(x,y)$ is convex w.r.t. $y$. 
Let $P $ be the indicator function defined on $\cal X$.
Then,  $0 \in \nabla_y h_{l,\tau}(x,y) + \partial P(y) $ if and only if $y = \arg\min_{y\in {\cal X}} h_{l,\tau}(x,y)$.

By the definition of $h_{l,\tau}(x,y)$, the gradient $\nabla_y h_{l,\tau}(x,y) $ can be derived as
\begin{align}
\nabla_y h_{l,\tau}(x,y) 
= \sum_{m=1}^M \frac{e^{\frac{f_m(y) - f_m(x)}{\tau}}}{\sum_{m=1}^M e^{\frac{f_m(y) - f_m(x)}{\tau}}} \nabla f_m(y) + l(y - x) .
\end{align}
When $y = x$, it can be further derived that
\begin{align}\label{eq:nabla_y_h_yeqx}
\nabla_y h_{l,\tau}(y,y) 
= \frac{1}{M}\sum_{m=1}^M \nabla f_m(y) .
\end{align}
Recall that $\lambda^\top F(x)$ is lower bounded for all $\lambda\in \Delta^M$.
And $\lambda^\top F(x)$ is continuous since $f_m$ are continuous
for all $m\in [M]$.
We assume either ${\cal X}$ is compact, or 
${\cal X} = \R^q$ and $f_m$ is coercive for all $m\in [M]$.
Then the solution to $\min_{x\in {\cal X}} \lambda^\top F(x)$ exists.
Let $x^* = \arg\min_{x\in {\cal X}} \frac{1}{M}\sum_{m=1}^M f_m(x) $, which implies
\begin{align}\label{eq:nabla_x_h_uniform}
  0 \in \frac{1}{M}\sum_{m=1}^M \nabla f_m(x^*) + \partial P(x^*).
\end{align}
Combining \eqref{eq:nabla_x_h_uniform} with \eqref{eq:nabla_y_h_yeqx}, we have that $0 \in \nabla_y h_{l,\tau}(x, y) + \partial P(y)\mid_{(x,y)=(x^*,x^*)}$.
Therefore, $x^* \in \arg\min_{y\in {\cal X}} h_{l,\tau}(x^*, y) $, and thus
\begin{align}
  \min_{y\in {\cal X}} h_{l,\tau}(x^*, y) = h_{l,\tau}(x^*, x^*)
  = \tau\ln \Big(\sum_{m=1}^M e^{\frac{f_m(x^*) - f_m(x^*)}{\tau} } \Big) 
  = \tau \ln M .
\end{align}
By definition,  $v_{l,\tau}(x^*) $ can be computed by
\begin{align}
  v_{l,\tau}(x^*) = - \min_{y\in {\cal X}} h_{l,\tau}(x^*, y) = - \tau \ln M
\end{align}
which completes the proof.
\end{proof}

\subsection{Continuity of the merit function}
\label{sub_app:cont_h_v}

\begin{lemma}[Continuity of $h_{l,\tau}$, $v_{l,\tau}$, and $p$]
\label{lemma:continuous_v}
1) If $f_m(x)$ is continuous for all $m\in [M]$, then the merit function $v_{l,\tau}(x)$ is lower semi-continuous. \\
2) 
Given a bounded set ${\cal X}_C$, suppose $f_m$ is 
$\ell_f$-Lipschitz continuous on ${\cal X}_C$ for $m=0,\ldots, M$.
Let $\ell_x = \sup_{x\in {\cal X}_C} \|x\|$. 
Then $h_{l,\tau}(x, y)$ is 
 $(\ell_f + 2 l \ell_x)$-Lipschitz continuous 
w.r.t. both $x\in {\cal X}_C$ and $y \in {\cal X}_C$,  $v_{l,\tau}(x)$ is 
$\ell_{v_{l,\tau}}$-Lipschitz continuous  on ${\cal X}_C$ with $\ell_{v_{l,\tau}} = \ell_f + 2l \ell_x$,
and for $\theta \geq 1$, $p(x)$ is $\ell_{p}$-Lipschitz continuous on ${\cal X}_C$ with $\ell_{p} = \theta\big(2 \ell_x \big)^{\theta - 1} \ell_{v_{l,\tau}}^{\theta}$.
\end{lemma}

\begin{proof}[Proof of Lemma~\ref{lemma:continuous_v}]
\emph{Proof of 1).}
Recall that $v_{l,\tau}(x) = -\min_{y\in {\cal X}} h_{l,\tau}(x,y)$, and 
$h_{l,\tau}(x,y) = \mathrm{LSE}_\tau \big( f_m(y) - f_m(x) \big) + \frac{l}{2}\|x - y\|^2$.
Since $f_m(x)$ is continuous for all $m\in [M]$, and the LSE function is continuous, we have $h_{l,\tau}(x,y)$ is continuous w.r.t. $x$ and $y$.

For any sequence $\{x_t\} \subseteq {\cal X}$ satisfying $\lim_{t\to\infty} x_t = \bar{x} \in$ ${\cal X}$, given any $\epsilon>0$, let $\bar{y} \in {\cal X}$ satisfy $h_{l,\tau}(\bar{x}, \bar{y}) \leq \min_{y\in {\cal X}} h_{l,\tau}(\bar{x}, y)+\epsilon$. As $h_{l,\tau}$ is continuous at $(\bar{x}, \bar{y})$, there exists $T>0$ such that
\begin{align}
\min_{y\in {\cal X}} h_{l,\tau}({x}_t, y) 
\leq h_{l,\tau}({x}_t, \bar{y}) 
\leq h_{l,\tau}(\bar{x}, \bar{y}) + \epsilon 
\leq \min_{y\in {\cal X}} h_{l,\tau}(\bar{x}, y)  + 2 \epsilon, \quad \forall t>T,  
\end{align}
and thus
\begin{align}
\limsup _{t \rightarrow \infty} 
~\min_{y\in {\cal X}} h_{l,\tau} (x_t, y ) 
\leq \min_{y\in {\cal X}} h_{l,\tau} (\bar{x}, y ) + 2 \epsilon .
\end{align}
As the above inequality holds for any $\epsilon > 0$, 
we obtain, 
\begin{align}
  \limsup _{t \rightarrow \infty} 
~\min_{y\in {\cal X}} h_{l,\tau} (x_t, y ) 
\leq \min_{y\in {\cal X}} h_{l,\tau} (\bar{x}, y ) 
\end{align}
which proves that $v_{l,\tau}(x)$ is lower semi-continuous.

\emph{Proof of 2).}
We prove the Lipschitz continuity of $h_{l,\tau}(x,y)$ below. 
We define
\begin{align}\label{eq:pi_m_xy}
\pi_m(x,y) \coloneqq 
\frac{e^{\frac{1}{\tau}(f_m(y) - f_m(x))}} {\sum_{m=1}^M e^{\frac{1}{\tau}(f_m(y) - f_m(x))} }.
\end{align}
Note that
\begin{align}
& \|\nabla_x  h_{l,\tau}(x,y) \| 
\leq \Big\| \sum_{m=1}^M \pi_m(x, y) \nabla f_m(x) \Big\| +  \| l(x - y) \|
\leq \ell_f + l ( \|x\| + \|y\|)
\leq \ell_f + 2 l \ell_x,  \\
& \|\nabla_y  h_{l,\tau}(x,y) \| 
\leq \Big\| \sum_{m=1}^M \pi_m(x, y) \nabla f_m(y) \Big\| +  \| l(x - y) \|
\leq \ell_f + l ( \|x\| + \|y\|)
\leq \ell_f + 2 l \ell_x .
\end{align}
Therefore, $h_{l,\tau}(x,y)$ is $(\ell_f + 2 l \ell_x)$-Lipschitz continuous 
w.r.t. both $x\in {\cal X}_C$ and $y \in {\cal X}_C$.

Next we prove the Lipschitz continuity of the merit function $v_{l,\tau}$ under additional assumptions.
From Assumption~\ref{assmp:local_lip_f}, the functions $f_m(x), m=0,\ldots, M$ are $\ell_f$-Lipschitz on a bounded set ${\cal X}_C$ where $\|x\|\leq \ell_x$. 
From~\eqref{eq:grad_v_main}, we can compute $\nabla v_{l,\tau}(x)$.
We then derive the bound of $\|\nabla v_{l,\tau}(x)\|$ below.
\begin{align*}
\|\nabla v_{l,\tau}(x)\| 
=& \Big\|\sum_{m=1}^M \pi_m(x) \nabla f_m(x) - l(x - y^*_{l,\tau}(x)) \Big\| \\
\leq & \ell_f + l \|x\| + l \|y^*_{l,\tau}(x)\|
\leq \ell_f + 2 l \ell_x  
\numberthis \label{eq:bound_nabla_v_ltau_1}
\end{align*}
which proves that $v_{l,\tau}(x)$ is $(\ell_f + 2l \ell_x)$-Lipschitz continuous on $\cal X$. 

Recall that $p(x) = \big(v_{l,\tau}(x) + \tau \ln M \big)^{\theta}$.
For $\theta \geq 1$, the gradient of $p(x)$ is given by
\begin{align}
\nabla p(x) = \theta \big(v_{l,\tau}(x) + \tau \ln M \big)^{\theta - 1}
\nabla v_{l,\tau}(x)
\end{align}
Note that $v_{l,\tau}(x) + \tau \ln M$ is bounded on a compact set since 
$v_{l,\tau}(x)$ is Lipschitz on this set, i.e.,
\begin{align}
v_{l,\tau}(x) + \tau \ln M 
\leq \ell_{v_{l,\tau}} \|x - x^*\| \leq 2 \ell_{v_{l,\tau}} \ell_x
~~\text{with}~~\ell_{v_{l,\tau}} = \ell_f + 2 l \ell_x .
\end{align}
Then $\|\nabla p(x)\|$ can be bounded by 
\begin{align}
\|\nabla p(x) \| 
\leq & \theta \ell_{v_{l,\tau}} 
\big(v_{l,\tau}(x) + \tau \ln M \big)^{\theta - 1}
\leq \theta \ell_{v_{l,\tau}} 
\big(v_{l,\tau}(x) + \tau \ln M \big)^{\theta - 1}  \\
\leq & \theta \ell_{v_{l,\tau}} 
\big(2 \ell_{v_{l,\tau}} \ell_x \big)^{\theta - 1}
= \theta\big(2 \ell_x \big)^{\theta - 1} \ell_{v_{l,\tau}}^{\theta} 
\end{align}
which proves that $p(x)$ is $\Big(\theta\big(2 \ell_x \big)^{\theta - 1} \ell_{v_{l,\tau}}^{\theta} \Big)$-Lipschitz continuous on $\cal X$. 
\end{proof}

\begin{corollary}[Lipschitz continuity of $\varphi_\gamma$]
\label{crlr:lip_cont_varphi}
Under the same settings as Lemma~\ref{lemma:f_smooth_boundset_imply_lip},
let $\ell_x = \sup_{x\in {\cal X}_C} \|x\|$. 
Given $\gamma > 0$, 
$\varphi_\gamma (x)$ is $\big( \gamma \ell_p + \ell_f \big)$-Lipschitz continuous on ${\cal X}_C$. 
\end{corollary}

\begin{proof}[Proof of Corollary~\ref{crlr:lip_cont_varphi}]
Recall that $\varphi_\gamma (x) = f_0(x) + \gamma p(x)$.
The proof directly follows by applying the
$\ell_p$-Lipschitz continuity of $p$
from Lemma~\ref{lemma:continuous_v},
and the $\ell_f$-Lipschitz continuity of $f_0$
from Assumption~\ref{assmp:local_lip_f}.
\end{proof}

\begin{lemma}[Lipschitz continuity of $y_{l,\tau}^*(x)$]
\label{lemma:lip_cont_ystar}
Under the same settings as Lemma~\ref{lemma:f_smooth_boundset_imply_lip},
recall that 
\begin{align}
y^*_{l,\tau}(x) \coloneqq \mathop{\arg\min}_{y\in {\cal X}} h_{l,\tau}(x,y) 
= \mathop{\arg\min}_{y\in {\cal X}} \Big\{ \tau\ln \Big(\sum_{m=1}^M e^{\frac{f_m(y) - f_m(x)}{\tau} } \Big) + \frac{l}{2}\|x - y\|^2 \Big\} .
\end{align}
For $l - \ell_{f,1} \geq \mu_{h_y} > 0$,
there exists $\ell_{y_{l,\tau}^*} = \frac{2M \ell_f}{\tau} \Big(\frac{\ell_f^2}{\tau} + \ell_{f,1} \Big) + \frac{4M \ell_f^3}{\tau^2} > 0$ that for all $x, x' \in {\cal X}_C $, the following holds
\begin{align}
\|y_{l,\tau}^*(x) - y_{l,\tau}^*(x')\| \leq \ell_{y_{l,\tau}^*} \|x - x'\| .
\end{align}
\end{lemma}

\begin{proof}
By Corollary~\ref{crlr:h_ltau_unique_y}, 
for $l + \min_{m\in [M]} \mu_m \geq \mu_{h_y} > 0$, 
the function $h_{l,\tau}(x,y) $ is $\mu_{h_y}$-strongly convex w.r.t. $y$.
Therefore, from~\citep[Theorem~2F.7]{dontchev2009implicit}, 
or using similar arguments for the proof in \citep[Lemma~15]{chen2023three},
we can derive that
\begin{align}
\|y_{l,\tau}^*(x) - y_{l,\tau}^*(x')\| 
\leq \mu_{h_y}^{-1} \|\nabla_{yy}^2 h_{l,\tau}(x,y) - \nabla_{yy}^2 h_{l,\tau}(x',y)\| .
\end{align}
Let $I_q \in \R^{q\times q}$ denote the identity matrix, then $\nabla_{yy}^2 h_{l,\tau}(x,y)$ can be further computed by
\begin{align}
  & \nabla_{yy}^2 h_{l,\tau}(x,y)
= \nabla_y \Bigg( \sum_{m=1}^M \underbrace{\frac{e^{\frac{f_m(y) - f_m(x)}{\tau}}}{\sum_{m=1}^M e^{\frac{f_m(y) - f_m(x)}{\tau}}}}_{\pi_m(x,y)} 
\nabla f_m(y)  + l (y - x) \Bigg) \nonumber \\
=& \nabla_y \Big(  \underbrace{\nabla F(y) \pi(x,y)}_{S(x,y)} + l (y - x) \Big)
~~\text{with}~~\pi(x,y) = [\pi_1(x,y), \ldots, \pi_M(x,y)]^\top
\nonumber \\
=& \frac{1}{\tau} \sum_{m=1}^M \pi_m(x,y) \nabla f_m(y) \nabla f_m(y)^\top
- \frac{1}{\tau} S(x,y)  S(x,y)^\top
+ \sum_{m=1}^M \pi_m (x,y) \nabla^2 f_m(y) + l I_q.
\end{align}
We first bound $\| S(x,y)  S(x,y)^\top -  S(x',y)  S(x',y)^\top \|$ by
\begin{align}
& \| S(x,y)  S(x,y)^\top -  S(x',y)  S(x',y)^\top \| 
\leq \Big(\|S(x,y) \| + \|S(x',y) \| \Big) \|S(x,y) - S(x',y)\| .
\end{align}
Then from Assumptions~\ref{assmp:smooth_f} and~\ref{assmp:local_lip_f}, we can bound $\|\nabla_{yy}^2 h_{l,\tau}(x,y) - \nabla_{yy}^2 h_{l,\tau}(x',y)\| $ by 
\begin{align*}
  & \|\nabla_{yy}^2 h_{l,\tau}(x,y) - \nabla_{yy}^2 h_{l,\tau}(x',y)\| \\
\leq & \frac{1}{\tau} \sum_{m=1}^M \|\pi_m (x,y) - \pi_m (x',y) \| \|\nabla f_m(y)\|^2
+ \frac{1}{\tau} \Big(\|S(x,y) \| + \|S(x',y) \| \Big) \|S(x,y) - S(x',y)\| \\
&+ \sum_{m=1}^M \|\pi_m (x,y) - \pi_m (x',y) \| \|\nabla^2 f_m(y)\|  \\
\leq & \sum_{m=1}^M \|\pi_m (x,y) - \pi_m (x',y) \| \Big(\frac{\ell_f^2}{\tau}
+ \ell_{f,1} \Big) + \frac{2\ell_f}{\tau}  \|S(x,y) - S(x',y)\|
\numberthis
\end{align*}
where $\|\pi_m (x,y) - \pi_m (x',y) \|$ can be further bounded by
\begin{align}
  \|\pi_m (x,y) - \pi_m (x',y) \| \leq \frac{2\ell_f}{\tau} \|x - x'\| .
\end{align}
Similarly, $\|S(x,y) - S(x',y)\|$ can be further bounded by
\begin{align}
\|S(x,y) - S(x',y)\| 
\leq \sum_{m=1}^{M} \big(\pi_m(x,y) - \pi_m(x',y) \big) \nabla f_m(y)
\leq \frac{2M \ell_f^2}{\tau} \|x - x'\| .
\end{align}
The proof is complete with $\ell_{y_{l,\tau}^*} = \frac{2M \ell_f}{\tau} \Big(\frac{\ell_f^2}{\tau} + \ell_{f,1} \Big) 
+ \frac{4M \ell_f^3}{\tau^2}$.
\end{proof}

\subsection{Proof of Proposition~\ref{prop:vl_property}: 
relations of $v_{l,\tau}$ and weak Pareto optimality}
\label{sub_app:proof_v_weak_po}

\begin{proof}[Proof of Proposition~\ref{prop:vl_property}]
We prove each property as follows.  
  
\emph{Property 1.} For the first argument, by the property of the Log-sum-exp function~\cite{Nesterov2005_smooth_min}, and since taking $\min_{y\in {\cal X}}$ preserves inequality, we have that
\begin{align}\label{eq:u0_vltau_ineq_relation}
  u_{l}(x) - \tau \ln M \leq v_{l,\tau} (x)\leq u_{l}(x) .
\end{align}
This implies that, as $\tau \downarrow 0$, $v_{l,\tau} (x)$ uniformly converges to $u_{l}(x)$. Also recall from Lemma~\ref{lemma:u_l0} that $x$ is weakly Pareto optimal if and only if $u_{l} = 0$. Therefore, $x$ is weakly Pareto optimal if and only if $\lim_{\tau\downarrow 0} v_{l,\tau}(x)=0$. The first argument is proved.

For the second argument, from~\cite{Nesterov2005_smooth_min}, we have that
\begin{align*}
& \tau\ln \Big(\sum_{m=1}^M e^{\frac{f_m(y) - f_m(x)}{\tau} } \Big) + \frac{l}{2}\|x - y\|^2 
\leq  \tau \ln M
+ \max_{m\in [M]} \{f_m(y) - f_m(x)\} + \frac{l}{2}\|x - y\|^2 .
\numberthis
\end{align*}

Since taking $\min_{y\in {\cal X}}$ preserves inequality, it implies that
\begin{align*}
& \min_{y\in {\cal X}} \Big\{ \tau\ln \Big(\sum_{m=1}^M e^{\frac{f_m(y) - f_m(x)}{\tau} } \Big) + \frac{l}{2}\|x - y\|^2 \Big\}
\leq  \tau \ln M
+ \min_{y\in {\cal X}} \Big\{\max_{m\in [M]} \{f_m(y) - f_m(x)\} + \frac{l}{2}\|x - y\|^2  \Big\}
\numberthis
\end{align*}
which proves that 
\begin{align*}
v_{l,\tau}(x) \geq &
- \min_{y\in {\cal X}} \Big\{\max_{m\in [M]} \{f_m(y) - f_m(x)\} + \frac{l}{2}\|x - y\|^2 \Big\}
- \tau \ln M \\
=& u_{l}(x) - \tau \ln M \geq - \tau \ln M 
\numberthis
\end{align*}
where the last inequality holds because $u_{l}(x) \geq 0$.
Furthermore, there exists $x\in {\cal X}$
such that $v_{l,\tau}(x) = - \tau \ln M$
by Lemma~\ref{lemma:exist_min_v}.
Then $\min_{x\in {\cal X}} v_{l,\tau}(x) = - \tau \ln M$.

\emph{Property 2.}
For the first argument, by Lemma~\ref{lemma:u_l0}, if $x$ is weakly Pareto optimal, then $u_{l}(x) = 0$. 
Furthermore, by Property-1, $u_{l}(x) \geq v_{l,\tau}(x)$,
which proves $v_{l,\tau}(x) \leq 0$.

Conversely, for the second argument, 
\emph{for condition a)}, $l=0$,
$v_{0,\tau}(x) \leq 0$ implies that
\begin{align}
\bar{u}(x) \leq v_{0,\tau}(x) + \tau \ln M \leq  \tau \ln M .
\end{align}
By the definition of $\bar{u}$, for all $z \in {\cal X}$,
it holds that
\begin{align}
\min_{m\in [M]} \{ f_m(x) - f_m(z) \} \leq \bar{u}(x) 
\leq \tau \ln M .
\end{align}
In other words, there exists no $z\in {\cal X}$ and $z \neq x$ such that, 
$F(z) < F(x) - \tau \ln M $.

\emph{For condition b)}, $l > 0$, $v_{l,\tau}(x) \leq -\tau \ln M$ implies that
\begin{align}
0 \leq u_l (x) \leq
v_{l,\tau}(x) + \tau \ln M \leq 0 .
\end{align}
By the definition of $u_l (x)$, it implies
\begin{align}
\min_{m\in [M]} \{ f_m(x) - f_m(y) \} 
- \frac{l}{2} \|x - y\|^2 \leq 0
\end{align}
By the convexity of $\cal X$, take $z \in {\cal X}$, and $t \in (0,1)$,
then $(1 -t)x +t z \in {\cal X}$.
Let $y = (1 -t)x +t z$, and plug it into the above inequality, we have
\begin{align}
\min_{m\in [M]} \{ f_m(x) - f_m( (1 -t)x +t z ) \} 
- \frac{l}{2} \|x - y\|^2 \leq 0
\end{align}
By the $(1,0)$-point-quasar convexity
of $f_m$ at $x$ for all $m\in [M]$, 
$f_m((1 -t) x + t z) \leq t f_m(z) + (1-t) f_m(x)$.
Therefore,
\begin{align}
\min_{m\in [M]} \{ t (f_m(x) - f_m(z)) \} - \frac{l}{2} \|t(z - x)\|^2 
\leq 0 .  
\end{align}
Dividing both sides by $t$ and letting $t \downarrow 0$, we have
that for all $z \in {\cal X}$, 
\begin{align}
\min_{m\in [M]} \{ f_m(x) - f_m(z) \} \leq 0.  
\end{align}
This proves $x$ is weakly Pareto optimal.

The proof of the properties of the smoothed merit function is complete. 
\end{proof}

\subsection{Examples of point quasar convex functions}
\label{sub_app:example_quasar}

By definition, $\mu$-(strongly) convex functions are 
$(1, \mu)$-(strongly) point quasar convex everywhere.
This quasar convexity property is also closely related to 
star convexity, restricted secant condition, variational coherence,
PL condition, invexity, etc.
For a more detailed discussion and more examples, refer to e.g.,~\citep[Appendix~A, D.2, D.3]{hinder2020near_star_quasar}.

\subsubsection{Examples condition 2-b) in Proposition~\ref{prop:vl_property}}

\begin{example}\label{exmp:v_f_SC}
For all $m\in [M]$, $f_m$ are strongly convex on ${\cal X}$.  
\end{example}
Example~\ref{exmp:v_f_SC} is the simplest case
covered by our method. But existing works which focus on such cases
usually require second-order information in the algorithm,
as summarized in Table~\ref{tab:comparison_methods}.

\begin{example}\label{exmp:x1x2_square}
For $x = [x_1; x_2] \in \R^2$, let ${\cal X} = \R^2$.
Define $f_1(x) = x_1^2 x_2^2$,
$f_2(x) = x_1^4 x_2^4$.  
\end{example}
In Example~\ref{exmp:x1x2_square}, $f_1, f_2$ are nonconvex but star-convex~\cite{hinder2020near_star_quasar} at $x^*= [0; 0]$.
Furthermore, $x^* $ satisfies that $v_{l,\tau}(x^*) = -\tau \ln M$.
And both $f_1, f_2$ satisfies the $(1,0)$-point quasar-convexity at $x^*$
within ${\cal X}$.

\subsection{Proof of gradient of the smoothed merit function}

\begin{lemma}[Gradient and directional derivative of $v_{l,\tau}$]
\label{lemma:direction_derivative_v}
The gradient  of $v_{l,\tau}$ can be computed by
\begin{align}
& \nabla v_{l,\tau}(x) = 
\sum_{m=1}^M \pi_m(x) \nabla f_m(x) - l(x - y^*_{l,\tau}(x)), 
~~\text{with}~~
\pi_m(x) \coloneqq \frac{e^{\frac{1}{\tau}(f_m(y^*_{l,\tau}(x)) - f_m(x))}} {\sum_{m=1}^M e^{\frac{1}{\tau}(f_m(y^*_{l,\tau}(x)) - f_m(x))} }.
\end{align}  

For all $z \in {\cal X}$,
the directional derivative of $v_{l,\tau}$,
denoted as $v_{l,\tau}'(x ; z-x)$, can be computed by
\begin{align}
v_{l,\tau}'(x ; z-x) = \sum_{m=1}^M \pi_m(x,y^*_{l,\tau}(x)) f_m^{\prime}(x ; z-x) - 
l \big(x- y^*_{l,\tau}(x) \big)^{\top}(z-x)  .
\end{align}  
\end{lemma}
    
\begin{proof}[Proof of Lemma~\ref{lemma:direction_derivative_v}]
Recall that we have defined 
\begin{align}
h_{l,\tau}(x,y)  = \tau\ln \Big(\sum_{m=1}^M e^{\frac{f_m(y) - f_m(x)}{\tau} } \Big) + \frac{l}{2}\|x - y\|^2.
\end{align}
By definition, $v_{l,\tau}(x) = -\min_{y\in {\cal X}} h_{l,\tau}(x,y)
= - h_{l,\tau}(x,y^*_{l,\tau}(x))$, with $y^*_{l,\tau}(x) \in \arg\min_{y\in {\cal X}} h_{l,\tau}(x,y)$.
If $l + \min_{m\in [M]}\mu_m \geq c > 0$,
by Corollary~\ref{crlr:h_ltau_unique_y}, $y^*_{l,\tau}(x)$ is unique.
Furthermore, $y^*_{l,\tau}(x)$ is continuous w.r.t. $x$.

By the extended Danskin-type theorem in e.g.,~\citep[Proposition~5]{shen2023penalty}, $v_{l,\tau}(x)$ is differentiable. Its gradient can be computed by
\begin{align*}
\nabla v_{l,\tau}(x) 
=& - \nabla_x h_{l,\tau}(x,y^*_{l,\tau}(x)) \\
=& \sum_{m=1}^M \frac{e^{\frac{f_m(y^*_{l,\tau}(x)) - f_m(x)}{\tau}}}{\sum_{m=1}^M e^{\frac{f_m(y^*_{l,\tau}(x)) - f_m(x)}{\tau}}} \nabla f_m(x) - l(x - y^*_{l,\tau}(x)) .
\numberthis\label{eq:grad_v_l_tau}
\end{align*}

Then given all $x, z \in {\cal X}$,
the directional derivative of $v_{l,\tau} (x)$
can be computed by
\begin{align}
  v'_{l,\tau}(x ; z - x)
= \sum_{m=1}^M \pi_m(x,y^*_{l,\tau}(x)) f_m^{\prime}(x ; z-x) - 
l \big(x- y^*_{l,\tau}(x) \big)^{\top}(z-x) \quad 
\text { for all } z \in {\cal X} 
\end{align}
where $\pi_m (x,y) = \frac{e^{\frac{f_m(y) - f_m(x)}{\tau}}}{\sum_{m=1}^M e^{\frac{f_m(y) - f_m(x)}{\tau}}} $.
The proof is complete. 
\end{proof}

\section{Subanalyticity and related properties}
\label{sec_app:subanalytic_property}

In this section, we first discuss some preliminaries on subanalyticity, 
and then prove the global subanalyticity of the merit function $v_{l,\tau}(x)$.

\begin{proposition}[\cite{Bolte2007_subanalytic}]
\label{prop:global_suba_bounded_suba}
Globally subanalytic sets are subanalytic, and conversely, any bounded subanalytic set is globally subanalytic.  
\end{proposition}

\begin{lemma}[{\cite{VanDenDriesMiller1996_ominimal}}]
\label{lemma:image_global}
The image or the preimage of a globally subanalytic set by a globally subanalytic function (respectively, globally subanalytic multivalued operator) is globally subanalytic.
\end{lemma}

\begin{lemma}[Projection theorem~{\cite{VanDenDriesMiller1996_ominimal}}]
\label{lemma:proj_thm}
Let $\Pi(x_1,\ldots ,x_{n+1}) = (x_1,\ldots ,x_n)$ be the canonical projection from $\R^{n+1}$ onto $\R^n$. If $S$ is a globally subanalytic subset of $\R^{n+1}$, then so is $\Pi(S)$ in $\R^n$.
\end{lemma}

\begin{lemma}[Lojasiewicz factorization lemma~{\citep[Theorem~6.4]{Bierstone1988_Semianalytic}}]
\label{lemma:holderian_v_suba}
If ${\cal X}^*_{v_{l,\tau}} \coloneqq \arg\min_{x\in {\cal X}} v_{l,\tau}(x)$ is globally subanalytic, and $v_{l,\tau}(x)$ is continuous and globally subanalytic, 
then the $(\varrho, \eta)$-H\"{o}lderian error bound holds for $v_{l,\tau}(x)$ for some $\varrho, \eta > 0$.
\end{lemma}

\subsection{Proof of global subanalyticity of $v_{l,\tau}$ and $p$}
\label{sub_app:proof_global_suba_v}

\begin{lemma}
\label{lemma:min_preserve_suba}
Let ${\cal X}$ and ${\cal Y}$ be two bounded subanalytic subsets in $\R^q$.
Let $f: {\cal X} \times {\cal Y} \to \R$ be a bounded subanalytic function. Then 1) the function $\phi$ below is bounded and subanalytic, thus globally subanalytic.
\begin{align}
\phi: {\cal X}  \ni x \mapsto \max_{y \in {\cal Y}} f(x, y) \in \R .
\end{align}
2) Given $x$, the solution set ${\cal Y}^*(x) \coloneqq \mathop{\arg\max}_{y\in {\cal Y}}f(x, y)$ is bounded and subanalytic, thus globally subanalytic. 
\end{lemma}

\begin{proof}
The proof mainly adopts the proof idea of~\citep[Lemma~4.18]{KOSIBA2025_suba_multif}. The key difference is that instead of using local boundedness of the functions, we start from global boundedness of $f$ on ${\cal X}$ and derive the global boundedness of $\phi$, which, combined with subanalyticity, leads to the  global subanalyticity of $\phi$ on ${\cal X}$. 

Let $\mathrm{Gr}(\cdot)$ denote the graph of a function.
Since $f$ is bounded, let $\overline{c_f}, \underline{c_f}$ denote its upper and lower bound, respectively.
And define ${\cal W} \coloneqq \{w\in \R \mid  \underline{c_f} \leq w \leq \overline{c_f} \}$.
Then ${\cal W} \subset \R$ is bounded.
Consider the following set
\begin{align}
A = \{(x,y,z,w)\in {\cal X}\times {\cal Y}\times \R \times {\cal W}
 \mid (x,y,z)\in \mathrm{Gr}(f(x,y)), z \leq w \} .
\end{align}
Note that $x,y$ are bounded, and so is $z$ because of the boundedness of $f$.
Furthermore, $w$ is bounded by definition. Thus $A$ is bounded. 
Also note that $A$ is subanalytic because $f$ is subanalytic, its graph is subanalytic, and the Cartesian product of two subanalytic sets is subanalytic.
Therefore, $A$ is globally subanalytic.

Define the following projection $\Pi_A$ and the set $B$ by canonical projection of $A$.
\begin{align}
& \Pi_A : {\cal X}\times {\cal Y}\times \R \times {\cal W}  
\ni (x,y,z,w) \mapsto (x,y,w) \in {\cal X}\times {\cal Y}\times {\cal W} . \\
& B \coloneqq \Pi_A (A) = \{  (x,y,w) \in {\cal X}\times {\cal Y}\times {\cal W}
\mid f(x,y)\leq w  \}.
\end{align}
Then  $B$ is globally subanalytic based on Lemma~\ref{lemma:proj_thm}.

We also define the following auxiliary sets
\begin{align}
& B_R = \{(x,y,w) \in {\cal X}\times {\cal Y}\times {\cal W} \},  \\
& B_R \backslash B = \{(x,y,w) \in {\cal X}\times {\cal Y}\times {\cal W}
\mid f(x,y) > w  \} .
\end{align}

Then we define the projection $\Pi_B$ and define the following set $C$
by canonical projection of $B$.
\begin{align}
&\Pi_B: {\cal X}\times {\cal Y}\times {\cal W}   
\ni (x,y, w) \mapsto (x, w) \in {\cal X} \times {\cal W} , \\
&C \coloneqq \Pi_B(B) \backslash \Pi_B(B_R\backslash B) 
= \{ (x,w) \in {\cal X} \times {\cal W} 
\mid \sup_{y\in {\cal Y}} f(x,y) \leq w   \} .
\end{align}
Since $C$ is bounded and subanalytic, it is globally subanalytic.

We further define the following subanalytic set and projection
\begin{align}
&D \coloneqq \{(x,w_1,w_2) \in {\cal X} \times {\cal W} \times {\cal W} 
\mid (x,w_1)\in C, (x,w_2)\in C , w_1 > w_2  \} ,\\
&\Pi_D: {\cal X} \times {\cal W} \times {\cal W} \ni (x,w_1,w_2)
\mapsto (x, w_1) \in {\cal X} \times {\cal W} .
\end{align}
Finally, observe that
\begin{align}
&C \backslash \Pi_D(D) \ni (x, w)
\iff  ~~\text{there exists no}~~ w' ~~\text{such that}~~ 
(x, w'), (x, w) \in C ~\text{and}~ w > w' \\
&\iff  ~~\text{there exists no}~~ w' ~~\text{such that}~~ 
\sup_{y\in {\cal Y}} f(x,y) \leq w'
~\text{and}~ \sup_{y\in {\cal Y}} f(x,y) \leq w ~\text{and}~ w > w' 
\end{align}
which means that $\mathrm{Gr}(\phi(x)) = C \backslash \Pi_D(D)$. 
Since $\mathrm{Gr}(\phi(x))$ is bounded and subanalytic, thus it is globally subanalytic.
By definition, $\phi$ is also globally subanalytic.

Next we proceed to prove the global subanalyticity of ${\cal Y}^*(x)$ given $x$.
Define the following set $E$ and projection $\Pi_E$
\begin{align}
& E \coloneqq \{ (x,y,w)\in {\cal X}\times {\cal Y}\times {\cal W} \mid (x,w)\in \mathrm{Gr}(\phi(x)), f(x,y) =w \} , \\
& \Pi_E : {\cal X}\times {\cal Y}\times {\cal W} \ni (x,y,w)
\mapsto (x, y) \in {\cal X}\times {\cal Y} .
\end{align}
Then $ \Pi_E(E) = \{ (x,y)\in {\cal X}\times {\cal Y} 
\mid y\in {\cal Y}^*(x) \} $ is globally subanalytic.
Furthermore, given $x = c_x \in {\cal X}$ for any $c_x \in {\cal X}$, we have
\begin{align}
  \Pi_E(E) \cap \{(x,y)\in {\cal X}\times {\cal Y}  \mid x = c_x\} 
=  \{ (x,y)\in {\cal X}\times {\cal Y} 
\mid y\in {\cal Y}^*(x), x = c_x \}
\end{align}
which is globally subanalytic. Taking projection $\Pi: {\cal X}\times {\cal Y} \ni (x,y ) \mapsto y \in {\cal Y}$ yields that ${\cal Y}^*(x)$ given $x = c_x$ for any $c_x \in {\cal X}$ is also globally subanalytic.
\end{proof}

\begin{proof}[Proof of Lemma~\ref{lemma:global_suba_X_v}]
\label{proof:global_suba_X_v}
\emph{Part 1:}
By Assumption~\ref{assmp:subanalyticity_f}, $f_m$ is globally subanalytic for all $m = 0, \ldots, M$. By Definition~\ref{def:subanalytic} and Lemma~\ref{lemma:image_global}, subanalyticity is preserved under the subanalytic LSE function.

Recall the definition of $h_{l,\tau}(x,y)$ in~\eqref{eq:h_ltau}.
By Assumption~\ref{assmp:local_lip_f} and 
Lemma~\ref{lemma:continuous_v}, $h_{l,\tau}(x,y)$ is 
Lipschitz continuous w.r.t. both $x,y \in {\cal X}_C$, 
thus $h_{l,\tau}(x,y)$ is bounded for $(x,y) \in {\cal X}_C\times {\cal X}_C$. Combining the above arguments, $h_{l,\tau}(x,y)$ is bounded and subanalytic on ${\cal X}_C$, thus globally subanalytic on ${\cal X}_C$ by Proposition~\ref{prop:global_suba_bounded_suba}.

\emph{Part 2:} Next, we prove that $v_{l,\tau}(x) = - \min_{y\in {\cal X}} h_{l,\tau}(x,y)$ and $p(x) = \big( v_{l,\tau}(x) + \tau \ln M \big)^{\theta}$
are both  globally subanalytic on ${\cal X}_C$.
This directly follows by applying Lemma~\ref{lemma:min_preserve_suba}-1).

\emph{Part 3:} Finally, we prove that ${\cal X}^*_{v_{l,\tau}} = \arg\min_{x\in {\cal X}} v_{l,\tau}(x )$ is also globally subanalytic on ${\cal X}_C$.
This directly follows by applying Lemma~\ref{lemma:min_preserve_suba}-2).

Consequently,
the merit function $v_{l,\tau}(x)$ satisfies the $(\varrho,\eta)$-H\"{o}lderian error bound for some $\varrho, \eta > 0$ on a bounded set ${\cal X}_C$ by Lemma~\ref{lemma:holderian_v_suba}.
The penalty function $p(x)$ also satisfies the $(\varrho_p,\eta_p)$-H\"{o}lderian error bound with $\varrho_p = \varrho^{\theta}$ and $\eta_p = \theta\eta$.
\end{proof}

\begin{lemma}[KL inequality~{\citep[Theorem~1]{Kurdyka1998_KL}}]
\label{lemma:kl_subanalytic}
Let $f: \Omega \to \mathbb{R}$ be a subanalytic function which is differentiable in $\Omega \backslash f^{-1}(0)$, where $\Omega$ is an open bounded subset of $\mathbb{R}^q$. Then there exist $c > 0, \nu > 0$ 
and $\alpha > 1$ such that:
$c \|\nabla f(x)\|^\alpha \geq |f(x)|$,
for each $x \in \Omega$ such that $|f(x)| \in(0, \nu)$. If in addition $\lim _{x \to a} f(x)=0$ for some $a \in \bar{\Omega}$, then the above inequality holds for each $x \in \Omega \backslash f^{-1}(0)$ close to $a$.
\end{lemma}

\subsection{Examples of globally subanalytic functions}
\label{sub_app:examples_global_suba}

We summarize some commonly used globally  subanalytic functions and their corresponding H\"{o}lderian error bound in Table~\ref{tab:HEB_example} below.
The first few examples of convex functions in Table~\ref{tab:HEB_example} are directly referenced from~\citep[Table~2]{Doron2023_simplebilevel}
and~\citep[Table~2]{chen2024penalty_simple_blo}.
\begin{table}[ht]
\centering
\fontsize{6.5}{7}\selectfont
\caption{Summary of some functions satisfying H\"{o}lderian error bound with corresponding exponents.}
\label{tab:HEB_example}
\begin{tabular}{ccccc}
\toprule
Functions & Remarks & Names & $\eta$ &$\varrho$ \\
\midrule
\multicolumn{5}{c}{Convex functions} \\
\hline 
$ \max _{i \in[m]}\{\langle {a}_i, x\rangle - b_i\}$ & ${a}_i \in \R^q,  b \in \R^m$ 
& Polytope & $1$ & $\max_{i \in[m]}\{\|a_i\|^{-1}\}$\\
$\|x-x_0\|_Q = \sqrt{(x-x_0)^{\top} Q (x-x_0)}$ & $Q \in \mathbb{S}^q, Q \succ 0, x_0 \in \R^q$ 
& $Q$-norm (Ellipsoid) & $1$ & $(\lambda_{\min}(Q))^{-\frac{1}{2}}$ \\
$\|x-x_0\|_p$ & $x_0 \in \R^q, p \geq 1$ 
& $\ell_p$-norm & $1$ &$1$ \\
\multirow{2}{*}{$\frac{1}{m} \sum_{i=1}^m \log (1+\exp (-{a}_i^{\top} x b_i))$} 
&\multirow{2}{*}{${a}_i \in \R^q, b \in \R^m, A \in \R^{m \times q}$} 
&\multirow{2}{*}{Logistic loss} & \multirow{2}{*}{$2$} 
&\multirow{2}{*}{\makecell{$(\lambda_{\min}(Q))^{-1}$,
$Q$ is \\a function of $a_i,b$}} \\ \\
$f(x)+\frac{\sigma}{2}\|x\|^2$ & $f$ convex, $\sigma>0$ 
& strongly-convex & $2$ & $\frac{2}{\sigma}$ \\
\hline
\multicolumn{5}{c}{Bounded and possibly nonconvex functions} \\
\hline
$\sin(x)$ &$x \in [0, \frac{1}{2}\pi]$ 
&Sine & $\geq 1$ &$(\frac{\pi}{2})^\eta$ \\
$x^p$ &$x \in [0, 1], p> 0$ &Polynomial & $p$ &$1$ \\
$e^x$ &$x \in [0, 1] $ &Exponential & $\geq 1$ &$1$ \\
$\ln(x+1)$ &$x \in [0, 1] $ &Logarithmic & $\geq 1$ &$\frac{1}{\ln(2)}$ \\
\bottomrule
\end{tabular}  
\end{table}

Below, we prove the H\"{o}lderian error bound of the last four bounded functions in Table~\ref{tab:HEB_example}.

\begin{proof}[Proof of HEB of bounded functions in Table~\ref{tab:HEB_example}]
We prove the $(\varrho, \eta)$-H\"{o}lderian error bound (HEB) of the last four  functions in 
Table~\ref{tab:HEB_example} one by one as follows.

\emph{1) Sine function.}
For $x \in [0, \frac{1}{2}\pi]$, $x^* = 0$ 
is the unique minimizer of $\sin (x)$, thus
for $\eta \geq 1$,
$\big( \mathrm{dist}(x, {\cal X}^*_{\sin}) \big)^\eta = | x |^\eta
\leq (\frac{\pi}{2})^\eta \sin (x)$ for all $x \in [0, \frac{1}{2}\pi]$.
Therefore, $\sin(x)$ satisfies $(\varrho, \eta)$-HEB  
for $x \in [0, \frac{1}{2}\pi]$
with $\eta \geq 1$ and $\varrho = (\frac{\pi}{2})^\eta$.

\emph{2) Polynomial function.}
For $x \in [0, 1]$, $x^* = 0$ 
is the unique minimizer of $x^p$.
Thus $\big( \mathrm{dist}(x, {\cal X}^*_{\sin}) \big)^p = | x |^p
\leq x^p$ for all $x \in [0, 1]$.
Therefore, $x^p$ satisfies $(\varrho, \eta)$-HEB 
for $x \in [0, 1]$ 
with $\eta = p$ and $\varrho = 1$.

\emph{3) Exponential function.}
For $x \in [0, 1]$, $x^* = 0$ 
is the unique minimizer of $e^x$.
Thus $\big( \mathrm{dist}(x, {\cal X}^*_{\sin}) \big)^\eta 
= | x |^\eta \leq e^x - e^0$ for all $x \in [0, 1]$,
and $\eta \geq 1$.
Therefore, $e^x$ satisfies $(\varrho, \eta)$-HEB 
for $x \in [0, 1]$  with 
$\eta \geq 1$ and $\varrho = 1$.

\emph{4) Logarithmic function.}
For $x \in [0, 1]$, $x^* = 0$ 
is the unique minimizer of $\ln (x+1)$.
Thus $\big( \mathrm{dist}(x, {\cal X}^*_{\sin}) \big)^\eta 
= | x |^\eta \leq \frac{1}{\ln(2)}\ln (x+1)$ for all $x \in [0, 1]$,
and $\eta \geq 1$.
Therefore, $e^x$ satisfies $(\varrho, \eta)$-HEB 
for $x \in [0, 1]$  with 
$\eta \geq 1$ and $\varrho = \frac{1}{\ln(2)}$.

The proof is complete.
\end{proof}

\begin{remark}
Note that, even though we have shown in the above that 
the exponential function satisfies HEB in a 
compact and subanalytic set $x\in [0,1]$,
it is known that the exponential function 
is not globally subanalytic on $\R$. 
In the following proofs which require the HEB property,
it suffices to show the global subanalyticity holds on a 
compact and subanalytic set constructed from the corresponding problem,
which leads to the HEB property on the compact and subanalytic set.
And HEB of $p(x)$ on a compact and subanalytic set within $\cal X$ is sufficient
to show the relations of global/local/stationary solutions of
the penalty problem to the bilevel problem
as long as there exists bounded solutions to $\min_{x\in {\cal X}} p(x)$,
even if $\cal X$ is not bounded.
\end{remark}

\subsection{Relations of H\"{o}lderian error bound, quadratic growth, proximal PL, and proximal error bound}
\label{sub_app:relations_heb_qg}

We further discuss the relations of H\"{o}lderian error bound (HEB), 
also known as H\"{o}lderian growth,
with other commonly used conditions such as quadratic growth (QG), proximal error bound (EB), proximal PL inequality, and strong convexity (SC). 
Below, we consider a general function for the discussion.
\begin{align}\label{eq:phi_f_g_def}
  \phi(x) \coloneqq f(x) + g(x)
\end{align}
with $f$ smooth and $g$ convex.
In our specific problem, $g$ is an indicator function on $\cal X$.
Let ${\cal X}^*_{\phi}$ denote the set of minimizers for $\phi$.
We first review the formal definitions of the above conditions, and then discuss their relations.
\begin{definition}[Quadratic growth]
The function~$\phi$ in~\eqref{eq:phi_f_g_def} satisfies quadratic growth (QG) if
\begin{align}
\varrho_\phi^{-1} \big( \mathrm{dist}(x, {\cal X}^*_{\phi}) \big)^2
\leq \phi(x) - \phi(x^*)
~~\text{with}~x^* \in {\cal X}^*_{\phi} . 
\end{align}
\end{definition}

\begin{definition}[Proximal error bound]
The function~$\phi$ in~\eqref{eq:phi_f_g_def} satisfies proximal error bound if
for $t > 0$, there exists $c_\phi > 0$ that
\begin{align}
\mathrm{dist}(x, {\cal X}^*_{\phi})
\leq c_\phi t^{-1} \|x - \mathrm{prox}_{t g} (x - t \nabla f(x)) \| .
\end{align}  
\end{definition}

\begin{definition}[Proximal PL]
The function~$\phi$ in \eqref{eq:phi_f_g_def} satisfies proximal PL inequality if  for $t > 0$, there exists $c > 0$ that
\begin{align}
& \phi(x) - \phi(x^*)
\leq c \mathcal{D}_g(x, t)
~~\text{with}~x^* \in {\cal X}^*_{\phi},
~\text{and} \nonumber \\
& \mathcal{D}_g(x, t) \coloneqq -2 t \min_y
\big[\langle\nabla f(x), y-x\rangle + \frac{t}{2}\|y-x\|^2
+ g(y) - g(x) \big]  .
\end{align}  
\end{definition}

Furthermore, when the function $f$ is convex, we have that
\begin{align*}
\text{proximal EB/PL} \stackrel{(d)}{\iff} &\text{QG} \\
\Downarrow \qquad\qquad &~\Downarrow \\
\text{proximal EB/KL} 
\stackrel{(f)}{\iff} &\text{HEB} 
\numberthis
\end{align*}
where $(d)$ has been proved in e.g.,~\citep[Corollary~3.6]{drusvyatskiy2018_eb_qg};
$(f)$ has been proved in e.g.,~\citep[Theorem~5]{Bolte2017_eb_kl_convex}.
In this paper, we provide a proof of
the relations between KL, HEB, and EB in~\eqref{eq:kl_heb_relation}
for nonconvex $f$ and constant $g$
in Lemma~\ref{lemma:kl_imply_heb} and Lemma~\ref{lemma:kl_heb_imply_eb}.

\section{Proof of relations of different formulations} 
\label{sec_app:proof_relation}

In this section, we prove the relations of the solutions of the bilevel problem and the penalty reformulation.
Recall that we denote the penalty function with exponent $\theta$ as
\begin{align}
p(x) = \big( v_{l,\tau}(x) + \tau \ln M \big)^{\theta}. 
\end{align}
For convenience, we restate the $\epsilon$-approximate bilevel program and the penalty-based reformulation below.
\begin{minipage}{0.45\textwidth}
\begin{align}
&\min_{x\in {\cal X}}~ f_0(x), 
~\mathrm{s.t.}~ p(x) \leq \epsilon
\tag{BP$_\epsilon$}
\label{eq:BPv_eps_app}
\end{align}
\end{minipage}
\hfill
\begin{minipage}{0.45\textwidth}
\begin{align}
\min_{x\in {\cal X}}~ f_0(x) + \gamma p (x)
\tag{PP$_\gamma$}
\label{eq:PP_gamma_app}
\end{align}
\end{minipage}

\subsection{Proof of Theorem~\ref{thm:global_solution_relation}: relation in the asymptotic setting}
\label{sub_app:proof_relation_asymptotic}

Below, Theorem~\ref{thm:global_solution_relation} presents the relation of global solutions of the penalty reformulation and~\eqref{eq:bilevel_opt_vltau}
in the asymptotic regime as $\gamma \to \infty$.

\begin{theorem}[Relation of global solutions]
\label{thm:global_solution_relation}
Suppose $f_m (x), m=0, 1, \ldots, M$ are continuous over ${\cal X}$.
Let $\gamma_j \to \infty$, and 
\begin{align}
x_j \in \mathop{\arg\min}_{x\in {\cal X}}~ \varphi_{\gamma_j}(x) 
\coloneqq f_0(x) + \gamma_j p(x).
\end{align}
Then for any limit point $\bar{x}$ of the sequence $\{x_j\}$, $\bar{x}$ is a solution to problem~\eqref{eq:bilevel_opt_vltau}.
\end{theorem}

\begin{proof}[Proof of Theorem~\ref{thm:global_solution_relation}]
\label{proof:relation_solution_penalty}
For any limit point $\bar{x}$ of the sequence $\{x_j\}$, 
let $x_{j_k}$ be the subsequence of $x_j$ such that $x_{j_k} \to \bar{x} \in {\cal X}$. 
For any $\epsilon > 0$, let $x_\epsilon$ be a point that satisfies the following
\begin{align}
p(x_\epsilon ) \leq 0,
~~\text{and}~~ 
f_0 (x_\epsilon )
\leq \inf_{x\in {\cal X}^*_{v_{l,\tau}}} f_0(x) + \epsilon,
~~\text{with}~~
{\cal X}^*_{v_{l,\tau}} \coloneqq \{x\in {\cal X}\mid v_{l,\tau}(x) + \tau\ln M \leq 0\} .
\end{align}
Then by definition that $x_{j_k} \in {\arg\min}~ \varphi_{\gamma_{j_k}}(x)$, we have that for any $\epsilon > 0$,
\begin{align*}
f_0(x_{j_k}) \leq f_0(x_{j_k}) + \gamma_{j_k} p(x_{j_k}) 
\leq & f_0(x_{\epsilon}) + \gamma_{j_k} p(x_{\epsilon})  \\
\leq & f_0(x_{\epsilon}) \leq \inf_{x\in {\cal X}^*_{v_{l,\tau}}} 
f_0(x) + \epsilon .
\numberthis 
\label{eq:xjk_xeps}
\end{align*}
Since $ p(x_{j_k})  \geq 0$, 
taking $\gamma_{j_k}\to \infty$ gives
\begin{align}
\mathop{\lim\sup}_{j_k \to \infty} ~
p (x_{j_k}) \leq 0 .
\end{align}
By the lower continuity of $v_{l,\tau}(x)$ from Lemma~\ref{lemma:continuous_v}, and thus the lower semi-continuity of $p(x)$, we have
\begin{align}\label{eq:p_bar_x_feasible}
  p(\bar{x})  \leq 0 .
\end{align}
Combining the above with \eqref{eq:xjk_xeps} yields that for any $\epsilon > 0$,
\begin{align}
f_0(\bar{x}) 
\leq  \inf_{x\in {\cal X}^*_{v_{l,\tau}}} f_0(x) + \epsilon .
\end{align}
The above inequality holds for any $\epsilon > 0$, which implies
\begin{align}\label{eq:f0_bar_x_opt}
  f_0(\bar{x}) 
\leq \inf_{x\in {\cal X}^*_{v_{l,\tau}}} f_0(x) .
\end{align}
Therefore, \eqref{eq:p_bar_x_feasible} and \eqref{eq:f0_bar_x_opt} together
show that $\bar{x}$ is a solution to problem~\eqref{eq:bilevel_opt_vltau}.

The proof is complete.
\end{proof}

Next, we prove the relations of the global, local, and stationary solutions of the bilevel problem and its penalty reformulation.

\subsection{Proof of Theorem~\ref{thm:global_solution_relation_eps}: the $\epsilon$-global solutions relation}
\label{sub_app:global_relation}

Recall that we let ${\cal X}^*_{v_{l,\tau}} \coloneqq 
\arg\min_{x\in {\cal X}} v_{l,\tau}(x)$.
Similarly, we define ${\cal X}^*_{\varphi_\gamma} \coloneqq 
\arg\min_{x\in {\cal X}} \varphi_\gamma(x)$.
For $\delta > 0$, define the $\delta$-approximate solution set below
\begin{align}
{\cal X}_\delta \coloneqq \{x \in {\cal X} \mid v_{l,\tau}(x) 
+ \tau \ln M \leq \delta \} .
\end{align}
Then ${\cal X}^*_{v_{l,\tau}} = {\cal X}_0 \subseteq {\cal X}_\delta $. 
We use ${\cal X}_C$ to denote a compact set within ${\cal X}$.

We first present some auxiliary lemmas below, then prove the main results.

\begin{lemma}
\label{lemma:heb_cont}
Let ${\cal X}_C \subseteq {\cal X}$ be a compact subanalytic set with ${\cal X}_C \cap {\cal X}^*_{v_{l,\tau}} \neq \emptyset$. 
Suppose $f_0(x)$ is $\ell_f$-Lipschitz continuous on ${\cal X}_C$
with some $\ell_f > 0$.
If $v_{l,\tau}$ is globally subanalytic on ${\cal X}_C$,
then it satisfies the $(\varrho, \eta)$-HEB on ${\cal X}_C$
with some $\varrho, \eta > 0$,
and $p(x)$ satisfies the $(\varrho_p, \eta_p)$-HEB on ${\cal X}_C$
with $\varrho_p = \varrho^{\theta}$ and $\eta_p = \theta\eta$.
Let $ x \in {\cal X}_C$ and 
$z(x) \in \arg\min_{z\in {\cal X}^*_{v_{l,\tau}} \cap {\cal X}_C} \|z-x\|$.
It holds that
\begin{align}
f_0(x) + \gamma p(x) - f_0(z(x)) 
\geq -\epsilon_\gamma 
\coloneqq
\begin{cases}
\ell_f \Big(\frac{\ell_f\varrho_p}{\gamma \eta_p}\Big)^{\frac{1}{\eta_p-1}} (\frac{1}{\eta_p} - 1), &~~\eta_p > 1, \gamma > 0 ; \\
0, &~~ \eta_p = 1, \gamma \geq \varrho_p\ell_f .
\end{cases}
\end{align}
\end{lemma}

\begin{proof}[Proof of Lemma~\ref{lemma:heb_cont}]
Since $v_{l,\tau}(x)$ and $p(x)$ are globally subanalytic on ${\cal X}_C$,
there exists $\varrho, \eta > 0$ that 
$v_{l,\tau}(x)$ satisfies the $(\varrho, \eta)$-HEB,
thus $p(x)$ satisfies the $(\varrho_p, \eta_p)$-HEB
with $\varrho_p = \varrho^{\theta}$ and $\eta_p = \theta\eta$ on ${\cal X}_C$, 
which yields
\begin{align}
p(x) \geq \varrho_p^{-1} \|z(x)-x\|^{\eta_p} .
\end{align}
Since ${\cal X}_C$ is bounded, we have that 
for all $x,x' \in {\cal X}_C$,
there exists $\ell_f > 0$ such that
$f_0(x) - f_0(x') \geq - \ell_f \|x-x'\|$. 
Combined with the above inequality, it holds that
\begin{align}
f_0(x) + \gamma p(x) - f_0(z(x)) 
\geq & \frac{\gamma}{\varrho_p}\|z(x)-x\|^{\eta_p} 
- \ell_f\|z(x) - x\| 
\geq \underbrace{
\inf_{\zeta\in \R_+} \frac{\gamma}{\varrho_p} 
\zeta^{\eta_p} - \ell_f \zeta  }_{-\epsilon_\gamma} .
\end{align}
Analyzing $-\epsilon_\gamma$ separately under $\eta_p > 1$
and $\eta_p = 1$  proves the result.
When $\eta_p > 1$, the result is obtained by 
solving the optimal $\zeta$ through the first-order optimality condition.
When $\eta_p = 1$, the optimal value is achieved at $\zeta = 0$.
\end{proof}

\begin{proof}[Proof of Theorem~\ref{thm:global_solution_relation_eps}]
\label{proof:global_solution_relation_eps}
1) Given $\gamma>0$, let $x_\gamma$ be a bounded $\epsilon$-global solution of~\eqref{eq:PP_gamma_app}, then for all $x\in {\cal X}$, 
\begin{align}\label{eq:f_x_gamma}
  f_0(x_\gamma) + \gamma p(x_\gamma) 
  \le f_0(x) + \gamma p(x) + \epsilon .
\end{align}
Let $x^* \in {\cal X}$ denote a bounded global solution of~\eqref{eq:BPv_eps_app} with $\epsilon=0$.
Then there exists a compact and subanalytic set ${\cal X}_C \subseteq {\cal X}$ such that $x_\gamma, x^* \in {\cal X}_C$ and ${\cal X}_C \cap {\cal X}^*_{v_{l,\tau}} \neq \emptyset$.
Let $z(x) \in \arg\min_{z\in {\cal X}^*_{v_{l,\tau}} \cap {\cal X}_C} \|z-x\|$.
By Lemma~\ref{lemma:heb_cont}, given $\gamma' > 0$, we have that
\begin{align}
f_0(x_\gamma)+\gamma' p(x_\gamma) 
- f_0(z(x_\gamma)) 
\label{eq:bound_f_x_zx} 
\geq {-\epsilon_{\gamma'}} .
\end{align}

Because $f_0(z(x_\gamma)) \geq f_0(x^*)$,
\eqref{eq:bound_f_x_zx} indicates that $f_0(x_\gamma)
+\gamma' p(x_\gamma) \geq f_0(x^*) - \epsilon_{\gamma'} $. 
Plugging $x=x_\gamma$ in \eqref{eq:f_x_gamma}, and combining with the above inequality, we obtain that
\begin{align}
f_0(x_\gamma) + \gamma p(x_\gamma) 
\le & f_0(x^*)  + \epsilon
\qquad \text{since } p(x^*) \leq 0
\nonumber\\
\le & f_0(x_\gamma) + \gamma' p (x_\gamma) + \epsilon_{\gamma'} + \epsilon
\qquad \text{from~\eqref{eq:bound_f_x_zx}}
\end{align}
which further implies
\begin{align}\label{eq:epsilon_gm_fsb_v_x_gamma}
(\gamma - \gamma') p (x_\gamma) 
\leq \epsilon_{\gamma'}  + \epsilon.
\end{align}
Define $\epsilon_\gamma \coloneqq p(x_\gamma) $.
The above inequality implies $\epsilon_\gamma \leq \frac{\epsilon_{\gamma'}
 + \epsilon}{\gamma - \gamma'}$.
Further, for any $x_{\epsilon_\gamma} \in {\cal X}$ that satisfies 
$p(x_{\epsilon_\gamma}) \leq \epsilon_\gamma$, 
from~\eqref{eq:f_x_gamma} we have
\begin{align}
f_0(x_\gamma) + \gamma p(x_\gamma) 
\le f_0(x_{\epsilon_\gamma}) + \gamma p(x_{\epsilon_\gamma}) 
+ \epsilon
\le f_0(x_{\epsilon_\gamma}) + \gamma \epsilon_\gamma  
+ \epsilon .
\end{align}
Since the above holds for any $x_{\epsilon_\gamma}\in {\cal X}_{\epsilon_\gamma}$ with ${\cal X}_{\epsilon_\gamma}\coloneqq \{x \in {\cal X} \mid p(x_{\epsilon_\gamma}) \leq \epsilon_\gamma \}$, by rearranging the above inequality, we have 
\begin{align}\label{eq:epsilon_p_opt_f0_x_gamma}
f_0(x_\gamma) - \inf_{x\in {\cal X}_{\epsilon_\gamma}} f_0(x) 
\le -\gamma \epsilon_\gamma  
+ \gamma \epsilon_\gamma  + \epsilon
= \epsilon.
\end{align}
Combining \eqref{eq:epsilon_gm_fsb_v_x_gamma} and \eqref{eq:epsilon_p_opt_f0_x_gamma}
proves that $x_\gamma$ is an $(\epsilon, \epsilon_\gamma)$-approximate global solution for \eqref{eq:BPv_eps_app} with $\epsilon_\gamma \leq \frac{\epsilon_{\gamma'}
+ \epsilon}{\gamma - \gamma'}$.

When $\eta_p > 1$, choosing $\gamma' = \frac{\gamma}{2}$, $\gamma \geq 2\ell_f  \delta ^{-\frac{\eta_p-1}{\eta_p}}  (\frac{\varrho_p}{\eta_p})^{\frac{1}{\eta_p}}  (1 - \frac{1}{\eta_p})^{\frac{\eta_p-1}{\eta_p}} $, and $\epsilon \leq \ell_f \delta^{\frac{1}{\eta_p}}  (\frac{\varrho_p}{2\eta_p})^{\frac{1}{\eta_p}} (1-\frac{1}{\eta_p})^{\frac{\eta_p-1}{\eta_p}}$,
we further have that
\begin{align}
\epsilon_\gamma \leq \frac{\epsilon_{\gamma'}
+ \epsilon}{\gamma - \gamma'}
= \Big(\frac{2\ell_f }{\gamma}\Big)^{\frac{\eta_p}{\eta_p-1}}  \Big(\frac{\varrho_p}{ \eta_p}\Big)^{\frac{1}{\eta_p-1}} (1 - \frac{1}{\eta_p} ) + \frac{2\epsilon}{\gamma} 
\leq \delta 
\end{align}
which proves that the $\epsilon$-global solution of~\eqref{eq:PP_gamma_app} 
with $\gamma \geq 2\ell_f  \delta ^{-\frac{\eta_p-1}{\eta_p}}  (\frac{\varrho_p}{\eta_p})^{\frac{1}{\eta_p}}  (1 - \frac{1}{\eta_p})^{\frac{\eta_p-1}{\eta_p}} $ and $\epsilon \leq \ell_f \delta^{\frac{1}{\eta_p}}  (\frac{\varrho_p}{2\eta_p})^{\frac{1}{\eta_p}} (1-\frac{1}{\eta_p})^{\frac{\eta_p-1}{\eta_p}}$, is an $(\epsilon, \delta)$-global solution for \eqref{eq:BPv_eps_app}.

When $\eta_p = 1$, from Lemma~\ref{lemma:heb_cont}, 
for $\gamma' \geq \varrho_p\ell_f $, $\epsilon_{\gamma'} = 0$.
Choosing  $\gamma \geq \varrho_p\ell_f + 1$,
we have $\epsilon_\gamma \leq \epsilon$.
Therefore, the $\epsilon$-global solution of~\eqref{eq:PP_gamma_app} 
is an $(\epsilon, \epsilon)$-global solution for \eqref{eq:BPv_eps_app}.

2) Next we prove the converse. Define $x_{\epsilon_b}$ to be a bounded $(\epsilon_b, \epsilon)$-global solution for \eqref{eq:BPv_eps_app}, and define ${\cal X}_\epsilon \coloneqq \{x\in {\cal X}\mid p(x) \leq \epsilon\}$. 
Then we have
\begin{align}
p(x_{\epsilon_b})  \leq \epsilon ,
~~\text{and}~~
f_0(x_{\epsilon_b}) \leq \inf_{x\in {\cal X}_\epsilon} f_0(x) + \epsilon_b .
\end{align}
Therefore, 
\begin{align}
f_0(x_{\epsilon_b}) + \gamma p(x_{\epsilon_b}) 
\leq & \inf_{x\in {\cal X}_\epsilon} f_0(x) + \epsilon_b
+ \gamma \epsilon .
\label{eq:bound_PP_x_eps_b}
\end{align}

Let ${\cal X}_C \subseteq {\cal X}$ be a compact subanalytic set that
${\cal X}_C \cap {\cal X}^*_{v_{l,\tau}} \neq \emptyset$,
and  ${\cal X}_C \cap {\cal X}^*_{\varphi_\gamma} \neq \emptyset$.
The set ${\cal X}_C$ exists because there exist bounded points in 
${\cal X}^*_{v_{l,\tau}}$ and ${\cal X}^*_{\varphi_\gamma}$, respectively.
For $x\in {\cal X}_C$, define $z(x) \in \arg\min_{z\in {\cal X}^*_{v_{l,\tau}} \cap {\cal X}_C} \|z-x\|$.
Let $x^*_C \in \arg\min_{x\in {\cal X}_C} \varphi_\gamma(x) $.
Then $\inf_{x\in {\cal X}_\epsilon} f_0(x)$ can be further bounded by
\begin{align}
\inf_{x\in {\cal X}_\epsilon} f_0(x)
\leq &\inf_{x\in {\cal X}^*_{v_{l,\tau}} \cap {\cal X}_C} f_0(x)
\qquad\text{since }{\cal X}^*_{v_{l,\tau}}\cap {\cal X}_C \subseteq {\cal X}_\epsilon  \nonumber\\
\leq & f_0(z(x^*_C))
\qquad\text{since } z(x^*_C)\in {\cal X}^*_{v_{l,\tau}}\cap {\cal X}_C  \nonumber\\
\leq & \inf_{x\in {\cal X}_C} f_0(x)+\gamma p(x) + \epsilon_{\gamma}
\qquad \text{from~Lemma~\ref{lemma:heb_cont}} \nonumber\\
\label{eq:bound_inf_S_eps_f0}
=& \inf_{x\in {\cal X} } f_0(x)+\gamma p(x) + \epsilon_{\gamma} .
\qquad \text{since ${\cal X}_C \cap {\cal X}^*_{\varphi_\gamma} \neq \emptyset$}
\end{align}
Plugging~\eqref{eq:bound_inf_S_eps_f0} into~\eqref{eq:bound_PP_x_eps_b} yields
\begin{align}
f_0(x_{\epsilon_b}) + \gamma p(x_{\epsilon_b}) 
\leq & \inf_{x\in {\cal X}} f_0(x)+\gamma p(x)  + \epsilon_{\gamma} + \epsilon_b + \gamma \epsilon \nonumber\\
\leq & \inf_{x\in {\cal X}} f_0(x)+\gamma p(x)  + \delta
\end{align}
where the last inequality holds by 
choosing $\epsilon \leq \frac{\delta}{3\gamma}$,
$\epsilon_b \leq \frac{\delta}{3 }$,
and $\gamma=\frac{\ell_f \varrho_p}{\eta_p}\Big(\frac{3 \ell_f(1-\frac{1}{\eta_p})}{\delta} \Big)^{\eta_p-1}$
when $\eta_p > 1$, $\gamma = \varrho_p \ell_f$ when $\eta_p = 1$.
The converse is proved.
\end{proof}

\subsection{Proof of Theorem~\ref{thm:relation_local_solution}
and Proposition~\ref{prop:local_relation_condition}: 
the local solutions relation}
\label{sub_app:local_relation}

\begin{proof}[Proof of Theorem~\ref{thm:relation_local_solution}]
For $\delta > 0$, ${\cal X}^*_{v_{l,\tau}} \subseteq {\cal X}_\delta \neq \emptyset$.
Furthermore, ${\cal X}_\delta$ is closed by the continuity of $v_{l,\tau}(x)$ and the closeness of $\cal X$.

Since $x_\gamma$ is a bounded local solution of~\eqref{eq:approximate_penalty} on $\mathcal{N}( x_\gamma,  r)$, it holds for any $x \in \mathcal{N}( x_\gamma,  r) \cap {\cal X}$ that 
\begin{align}
f_0(x_\gamma ) + \gamma p(x_\gamma) 
\leq f_0(x ) + \gamma p(x).  
\end{align}
As ${\cal X}_\epsilon$ for $\epsilon \geq 0$ is closed and nonempty, 
there exists $x_\epsilon \in \mathop{\arg\min}_{x\in {\cal X}_\epsilon} 
\| x - x_\gamma \|$.
Also recall that $\bar{x} \in \mathcal{N}( x_\gamma,  r) \cap 
\mathcal{X}_\epsilon  $, then
\begin{align}
 \|x_\epsilon - x_\gamma \| = \min_{x\in {\cal X}_\epsilon} 
 \| x - x_\gamma \| 
 \leq \| \bar{x} - x_\gamma \| \leq r
\end{align}
which implies that $x_\epsilon \in \mathcal{N}( x_\gamma,  r) \cap 
\mathcal{X}_\epsilon \subseteq \mathcal{X} $, therefore,
\begin{align}
f_0(x_\gamma ) + \gamma p(x_\gamma) 
\leq f_0( x_\epsilon ) + \gamma p( x_\epsilon ) 
\leq f_0( x_\epsilon ) + \gamma \epsilon.  
\end{align}
Furthermore, since there exists $\ell_f > 0$ such that  $f_0$ is $\ell_f$-Lipschitz continuous on the bounded set $\mathcal{N}(x_\gamma, r)$,  rearranging the above inequality, we have
\begin{align}
\gamma p(x_\gamma) 
- \ell_f \| x_\epsilon - x_\gamma \| - \gamma \epsilon \leq 0 .
\end{align}
Let ${\cal X}_C \subseteq {\cal X}$ be a compact set that
${\cal X}_C \cap {\cal X}^*_{v_{l,\tau}} \neq \emptyset$,
${\cal X}_C \cap {\cal X}^*_{\varphi_\gamma} \neq \emptyset$,
and $x_\epsilon, x_\gamma \in {\cal X}_C$.
The set ${\cal X}_C$ exists because there exist bounded points in 
${\cal X}^*_{v_{l,\tau}}$ and ${\cal X}^*_{\varphi_\gamma}$, respectively.
By the $(\varrho_p, \eta_p)$-HEB of $p(x)$ on ${\cal X}_C$, 
\begin{align}
\|x_\epsilon - x_\gamma \| = \mathrm{dist}(x_\gamma, \mathcal{X}_\epsilon)
\leq \mathrm{dist}(x_\gamma, \mathcal{X}^*_{v_{l,\tau}}\cap {\cal X}_C )
\leq \varrho_p^{\frac{1}{\eta_p}} \big( p(x_\gamma) \big)^{\frac{1}{\eta_p}} .  
\end{align}
Plugging this into the above inequality, we have
\begin{align}
\gamma p(x_\gamma)  
- \ell_f \varrho_p^{\frac{1}{\eta_p}} \big( p(x_\gamma) \big)^{\frac{1}{\eta_p}} - \gamma \epsilon \leq 0 .  
\end{align} 
For $\eta_p > 1$,
the above implies that if we choose 
$\gamma \geq 
\ell_f (2\varrho_p)^{\frac{1}{\eta_p}} \epsilon^{\frac{1}{\eta_p} - 1}$, 
there exists 
$\bar{\epsilon}_\gamma = 
\Big(\frac{\ell_f}{\gamma}\Big)^{\frac{\eta_p}{\eta_p - 1}} 
\varrho_p^{\frac{1}{\eta_p - 1}} + 2 \epsilon > 0$ that
$p(x_\gamma) \leq \bar{\epsilon}_\gamma $. 
To show this, we define a function 
$ q(z) \coloneqq \gamma z -\ell_f \varrho_p^{\frac{1}{\eta_p}} z^{\frac{1}{\eta_p}}  - \gamma \epsilon $.
Then $q(z_0) = - \gamma \epsilon$ with $z_0 = \Big(\frac{\ell_f}{\gamma}\Big)^{\frac{\eta_p}{\eta_p - 1}} \varrho_p^{\frac{1}{\eta_p - 1}}$.
Furthermore, 
\begin{align*}
q(\bar{\epsilon}_\gamma) - q(z_0 ) =
q(z_0 + 2\epsilon) - q(z_0 ) =& 2\gamma \epsilon
- \ell_f \varrho_p^{\frac{1}{\eta_p}} [(z_0 + 2\epsilon)^{\frac{1}{\eta_p}}
- z_0^{\frac{1}{\eta_p}} ] \\
\stackrel{(a)}{\geq} & 2\gamma \epsilon
- \ell_f \varrho_p^{\frac{1}{\eta_p}} (2 \epsilon)^{\frac{1}{\eta_p}} 
\stackrel{(b)}{\geq} \gamma \epsilon
\end{align*}
where $(a)$ follows from the subadditivity of the function $z^{\frac{1}{\eta_p}}$ with $\eta_p > 1$, $(b)$ follows from that $\gamma \geq 
\ell_f (2\varrho_p)^{\frac{1}{\eta_p}} \epsilon^{\frac{1}{\eta_p} - 1}$.
This shows $q(\bar{\epsilon}_\gamma) \geq 0$.
By the property of $q(z)$, and that $q(p(x_\gamma)) \leq 0$,
we have $p(x_\gamma) \leq \bar{\epsilon}_\gamma$.

For $\eta_p = 1$,
the above implies that there exists 
$\bar{\epsilon}_\gamma = \gamma \epsilon $ 
with $\gamma = \varrho_p \ell_f + 1$
that $p(x_\gamma) \leq \bar{\epsilon}_\gamma $. 

Choose $\epsilon_\gamma = p(x_\gamma)$, then it holds for all 
$x \in \mathcal{N}(x_\gamma, r) \cap {\cal X}_{\epsilon_\gamma}$ that 
\begin{align}
f_0(x_\gamma) - f_0(x ) \leq \gamma ( p(x) - \epsilon_\gamma) \leq 0
\end{align}
which together with $x_\gamma \in {\cal X}_{\epsilon_\gamma}$ implies
that $x_\gamma$ is a local solution of \eqref{eq:BPv_eps_app} with 
$\epsilon = \epsilon_\gamma \leq \bar{\epsilon}_\gamma$.

The proof is complete.
\end{proof}

\begin{lemma}\label{lemma:penalty_v_stationary}
If $x_\gamma \in {\cal X}$ is an $\epsilon$-stationary solution 
to~\eqref{eq:approximate_penalty},
and $\|\nabla f_0(x_\gamma) \| \leq \ell_f$, 
then it is also an $\epsilon_\gamma$-stationary solution to 
$\min_{x\in {\cal X}} p(x)$ with $\epsilon_\gamma = \frac{\epsilon + \ell_f}{\gamma}$.
Furthermore, for $\theta \geq 1$, let $c_v(x_\gamma) \coloneqq \theta (v_{l,\tau}(x_\gamma) + \tau \ln M)^{\theta-1} $.
Then either one of the following two conditions holds.\\
1) $c_v(x_\gamma) = 0$ and $x_\gamma$ is an optimal solution, 
thus also a stationary solution to 
$\min_{x\in {\cal X}} v_{l,\tau}(x)$;\\ 
2) $c_v(x_\gamma) > 0$, and
$x_\gamma$ is an $\epsilon_{\gamma}'$-stationary solution to 
$\min_{x\in {\cal X}} v_{l,\tau}(x)$ with $\epsilon_\gamma' = \frac{\epsilon + \ell_f}{\gamma c_v(x_\gamma)}$.
\end{lemma}

\begin{proof}[Proof of Lemma~\ref{lemma:penalty_v_stationary}]
Since $x_\gamma \in {\cal X}$ is an $\epsilon$-stationary solution 
to~\eqref{eq:approximate_penalty}, for $\alpha > 0$, we have
\begin{align}
\frac{1}{\alpha}\|x_\gamma - \mathrm{Proj}_{\cal X}
(x_\gamma - \alpha \nabla \varphi_\gamma (x_\gamma))\| \leq \epsilon.
\end{align}
By the definition that $\varphi_\gamma(x) = f_0(x) + \gamma p(x)$,
we further have that
\begin{align*}
&\|x_\gamma - \mathrm{Proj}_{\cal X}
(x_\gamma - \alpha \gamma \nabla p (x_\gamma))\| \\
\leq & \|x_\gamma - \mathrm{Proj}_{\cal X}
(x_\gamma - \alpha \nabla \varphi_\gamma (x_\gamma)) \|
+ \| \mathrm{Proj}_{\cal X}
(x_\gamma - \alpha \nabla \varphi_\gamma (x_\gamma))
- \mathrm{Proj}_{\cal X}
(x_\gamma - \alpha \gamma \nabla p (x_\gamma))\| \\
\leq & \|x_\gamma - \mathrm{Proj}_{\cal X}
(x_\gamma - \alpha \nabla \varphi_\gamma (x_\gamma)) \|
+ \|\alpha \nabla \varphi_\gamma (x_\gamma)
- \alpha \gamma \nabla p (x_\gamma) \| \\
= &\|x_\gamma - \mathrm{Proj}_{\cal X}
(x_\gamma - \alpha \nabla \varphi_\gamma (x_\gamma)) \|
+ \alpha \|\nabla f_0(x_\gamma) \| .
\numberthis
\end{align*}
Dividing both sides by $\frac{1}{\alpha \gamma}$ of the above inequality yields
\begin{align}
\frac{1}{\alpha \gamma}\|x_\gamma - \mathrm{Proj}_{\cal X}
(x_\gamma - \alpha \gamma \nabla p (x_\gamma))\| 
\leq \frac{\epsilon + \ell_f}{\gamma} 
\end{align}
which proves that $x_\gamma$ is an $\epsilon_\gamma$-stationary solution to 
$\min_{x\in {\cal X}} p(x) $ with $\epsilon_\gamma = \frac{\epsilon + \ell_f}{\gamma}$.
By the definition that $p(x_\gamma) = (v_{l,\tau}(x_\gamma) + \tau\ln M)^{\theta}$, for $\theta \geq 1$, we further have
\begin{align}
\frac{1}{\alpha \gamma }\|x_\gamma - \mathrm{Proj}_{\cal X}
(x_\gamma - \alpha \gamma \underbrace{\theta (v_{l,\tau}(x_\gamma) + \tau \ln M)^{\theta-1}}_{c_v(x_\gamma)} \nabla v_{l,\tau} (x_\gamma))\| 
\leq \frac{\epsilon + \ell_f}{\gamma } 
\end{align}
where $c_v(x_\gamma) \geq 0$. This implies
that either $v_{l,\tau}(x_\gamma) + \tau \ln M = 0$, or
$c_v(x_\gamma) > 0$, and
\begin{align}
\frac{1}{\alpha \gamma c_v(x_\gamma)}\|x_\gamma - \mathrm{Proj}_{\cal X}
(x_\gamma - \alpha \gamma c_v(x_\gamma)\nabla v_{l,\tau} (x_\gamma))\| 
\leq \frac{\epsilon + \ell_f}{\gamma c_v(x_\gamma)} .
\end{align}
The proof is complete.
\end{proof}

\begin{lemma}\label{lemma:value_gap_y_star}
Suppose Assumption~\ref{assmp:smooth_f} holds.
Recall that $y^*_{l,\tau}(x) \coloneqq 
\mathop{\arg\min}_{y\in {\cal X}} h_{l,\tau} (x,y)$.
For $l - \ell_{f,1} > 0$, we have
\begin{align}\label{eq:value_gap_y_star}
\frac{l - \ell_{f,1}}{2 }
\|x - y^*_{l,\tau}(x)\|^2 \leq
v_{l,\tau}(x) + \tau \ln M
\leq \frac{3\ell_{h_{l,\tau},1}^2 + 6 \alpha^{-2}}{2 (l - \ell_{f,1})} 
\|x - y^*_{l,\tau}(x) \|^2 .   
\end{align}
\end{lemma}

\begin{proof}[Proof of Lemma~\ref{lemma:value_gap_y_star}]
By definition, 
$v_{l,\tau}(x) = - h_{l,\tau}(x, y^*_{l,\tau}(x))$, 
and $h_{l,\tau}(x, x) = \tau \ln M$.
Therefore,
\begin{align*}
v_{l,\tau}(x) + \tau \ln M 
=& - h_{l,\tau}(x, y^*_{l,\tau}(x)) + h_{l,\tau}(x, x) 
\geq 
\frac{l - \ell_{f,1}}{2 }
\|x - y^*_{l,\tau}(x)\|^2
\numberthis
\end{align*}
where the last inequality follows from the 
$(l - \ell_{f,1})$-strong convexity  of $h_{l,\tau}(x, \cdot)$,
thus the quadratic growth.
The first inequality in~\eqref{eq:value_gap_y_star} is proved.

For the second inequality in~\eqref{eq:value_gap_y_star}, we have
\begin{align*}
  & v_{l,\tau}(x) + \tau \ln M 
=- h_{l,\tau}(x, y^*_{l,\tau}(x)) + h_{l,\tau}(x, x) 
\stackrel{(a)}{\leq} 
\frac{1}{2 (l - \ell_{f,1}) \alpha^2}
\|x - \mathrm{Proj}_{\cal X}(x - \alpha\nabla_y h_{l,\tau}(x, x)) \|^2 \\
\stackrel{(b)}{\leq} & 
\frac{1}{2 (l - \ell_{f,1}) \alpha^2}
\|x - \mathrm{Proj}_{\cal X}(x - \alpha\nabla_y h_{l,\tau}(x, x)) 
- y^*_{l,\tau}(x) + 
\mathrm{Proj}_{\cal X}(y^*_{l,\tau}(x) - \alpha\nabla_y h_{l,\tau}(x, y^*_{l,\tau}(x))) \|^2 \\
\stackrel{(c)}{\leq} & \frac{3\ell_{h_{l,\tau},1}^2 + 6 \alpha^{-2}}{2 (l - \ell_{f,1})} 
\|x - y^*_{l,\tau}(x)\|^2
\numberthis
\end{align*}
where $(a)$ follows from the 
$(l - \ell_{f,1})$-strong convexity  of $h_{l,\tau}(x, \cdot)$,
thus the proximal-PL inequality,
$(b)$ follows from the optimality condition of 
$\min_{y\in {\cal X}} h_{l,\tau}(x,y)$ at $y^*_{l,\tau}(x)$,
and $(c)$ follows from 
the $\ell_{h_{l,\tau},1}$-smoothness of $h_{l,\tau}(x, \cdot)$,
and the non-expansiveness of projection.
\end{proof}

\begin{corollary}
Lemma~\ref{lemma:value_gap_y_star} implies that,
for $l - \ell_{f,1} > 0$,
$v_{l,\tau}(x)$ achieves its minimum value
if and only if $y_{l,\tau}^*(x) = x$.  
\end{corollary}

\begin{proof}[Proof of Proposition~\ref{prop:local_relation_condition}]
Since $x_\gamma$ is a local solution to~\eqref{eq:approximate_penalty}, 
it is a stationary solution to~\eqref{eq:approximate_penalty}, thus an $\epsilon$-stationary solution to $\min_{x\in {\cal X}} v_{l,\tau}$
by Lemma~\ref{lemma:penalty_v_stationary}.

We then show the $\epsilon$-stationary solution of $\min_{x\in {\cal X}} v_{l,\tau}$ is also an $\epsilon'$-optimal solution of
$\min_{x\in {\cal X}} v_{l,\tau}$ under additional assumptions of $F$.
By Lemma~\ref{lemma:direction_derivative_v},
the directional derivative of $v_{l,\tau}$,
denoted as $v_{l,\tau}'(x ; z-x)$, 
for all $z \in {\cal X}$,
can be computed by
\begin{align}
v_{l,\tau}'(x ; z-x) = 
\sum_{m=1}^M \pi_m (x, y^*_{l,\tau}(x) ) f_m^{\prime}(x ; z-x)- 
l \Big(x- y^*_{l,\tau}(x) \Big)^{\top}(z-x) 
\geq -\epsilon \|z - x\|.  
\end{align}  
Plugging $z = y^*_{l,\tau}(x)$ into the above inequality yields
\begin{align}\label{eq:z_ystar_eps}
\sum_{m=1}^M \pi_m(x, y^*_{l,\tau}(x) ) f_m^{\prime}(x ; y^*_{l,\tau}(x) - x) 
+ l \|y^*_{l,\tau}(x) - x \|^2 
\geq -\epsilon \|y^*_{l,\tau}(x) - x\| . 
\end{align}
Furthermore, let $P(x)$ be the indicator function on $\cal X$, 
by the optimality condition of $h_{l,\tau}(x,\cdot)$ at $y^*_{l,\tau}(x)$, we have
\begin{align}
0 \in \sum_{m=1}^M \pi_m(x, y^*_{l,\tau}(x)) \nabla f_m(y^*_{l,\tau}(x)) 
+ l \big( y^*_{l,\tau}(x) - x \big)  + \partial P(y^*_{l,\tau}(x)) .
\end{align}

For analysis, we construct an auxiliary function 
$R( y) \coloneqq \lambda^\top F(y) + P(y)$,
where $\lambda = [\lambda_1; \lambda_2; \ldots; \lambda_M] \in \Delta^M$ is a hyperparameter.
Then $R(y)$ is $(1,0)$-point quasar convex at $y = y^*_{l,\tau}(x)$.
By the definition of the subgradient of $R(y)$, and the $(1,0)$-point quasar convexity, for all $z \in {\cal X}$, we have
\begin{align}
\Big\langle \sum_{m=1}^M \lambda_m  \nabla f_m(y^*_{l,\tau}(x)) + \partial P(y^*_{l,\tau}(x)), 
z - y^*_{l,\tau}(x) \Big\rangle  \leq 
\sum_{m=1}^M \lambda_m \big(f_m (z) - f_m (y^*_{l,\tau}(x) ) \big) .
\end{align}
Letting $\lambda_m = \pi_m(x, y^*_{l,\tau}(x))$ given that $x\in {\cal X}$ is a stationary point of $\min_{x\in {\cal X}} v_{l,\tau}(x)$,
and plugging $y = y^*_{l,\tau}(x)$,  $l (x - y^*_{l,\tau}(x)) \in \sum_{m=1}^M \pi_m(x, y^*_{l,\tau}(x)) \nabla_y f_m(y^*_{l,\tau}(x)) + \partial P(y^*_{l,\tau}(x))$ into the above inequality yield
\begin{align}
l \langle x - y^*_{l,\tau}(x), 
z - y^*_{l,\tau}(x) \rangle \leq 
\sum_{m=1}^M \pi_m(x, y^*_{l,\tau}(x)) \Big( f_m(z) - f_m(y^*_{l,\tau}(x)) \Big).
\end{align}
Substituting $z = x$ into the above inequality, 
then $l \|x - y^*_{l,\tau}(x)\|^2 \leq \sum_{m=1}^M \pi_m(x, y^*_{l,\tau}(x)) 
\Big( f_m(x) - f_m(y^*_{l,\tau}(x)) \Big) $, 
which combined with~\eqref{eq:z_ystar_eps} yields
\begin{align}
  & \sum_{m=1}^M \pi_m(x, y^*_{l,\tau}(x)) 
\Big( f_m(y^*_{l,\tau}(x)) - f_m(x) \Big) - \epsilon \|y^*_{l,\tau}(x) - x\| 
\leq \sum_{m=1}^M \pi_m(x,y^*_{l,\tau}(x)) f_m' (x ; y^*_{l,\tau}(x) - x) 
\nonumber \\
\leq  & \sum_{m=1}^M \pi_m(x, y^*_{l,\tau}(x)) 
\Big( f_m(y^*_{l,\tau}(x)) - f_m(x) \Big) - \mu \|y^*_{l,\tau}(x) - x\|^2 
\end{align}
where the last inequality holds since $f_m(x)$ are $(1,\mu)$-strong point quasar convex for all $m\in [M]$ at $x$. 
The above inequality implies $\mu \|y^*_{l,\tau}(x) - x \| \leq \epsilon$.

Applying Lemma~\ref{lemma:value_gap_y_star} and letting $\alpha=O(1)$, we have
\begin{align}
v_{l,\tau}(x) + \tau \ln M \leq 
\frac{3\ell_{h_{l,\tau},1}^2 + 6 \alpha^{-2}}{2 (l - \ell_{f,1})} 
\|x - y^*_{l,\tau}(x) \|^2
\leq \frac{3\ell_{h_{l,\tau},1}^2 + 6 \alpha^{-2}}{2 (l - \ell_{f,1}) \mu^2}
\epsilon^2 = O(\epsilon^2). 
\end{align}
Then there exists $\bar{x} \in {\cal N}(x_\gamma, r)$ that 
$p(\bar{x}) = O(\epsilon^{2\theta})$.
\end{proof}

\subsection{Proof of Theorem~\ref{thm:relation_stationary_solution}: the $\epsilon$-stationary solutions relation}
\label{sub_app:stationary_relation}

We first discuss the stationary condition of~\eqref{eq:bilevel_opt_vltau}
when ${\cal X} = \R^q$, 
and its constraint qualifications.
Then we prove Theorem~\ref{thm:relation_stationary_solution}, the relation of $\epsilon$-stationary solutions to~\eqref{eq:approximate_penalty}
and~\eqref{eq:bilevel_opt_vltau}.
Consider a general constrained problem below
\begin{align}\label{eq:constrained_calm}
\min _{x \in \R^q} f_0(x) \quad \text { s.t. } H(x)=0
\end{align}
where $f_0: \R^{q} \to \R$, and $H: \R^{q} \to \R^{d_h}$ with $d_h \geq 1$. 

\begin{definition}[KKT condition of~\eqref{eq:constrained_calm}]
\label{def:kkt_constrained_H}
The KKT condition of~\eqref{eq:constrained_calm} is
\begin{align}
\underbrace{ H(x) = 0}_{\text{feasibility}},
\quad 
\underbrace{\nabla f_0(x) + \nabla H(x) w = 0}_{
  \text{stationarity}},
~~\text{with}~~w\in \R^{d_h}.
\end{align}
Correspondingly, the $(\epsilon', \epsilon)$-KKT condition 
of~\eqref{eq:constrained_calm} is
\begin{align}
\|  H(x)  \| \leq \epsilon,
\quad 
\|\nabla f_0(x) + \nabla H(x) w \| \leq \epsilon',
~~\text{with}~~w\in \R^{d_h}.
\end{align}
\end{definition}

\begin{definition}[Calmness~{\citep[Definition~6.4.1]{clarke1990optimization}}]
\label{def:calm}
Let $x^*$ be the global minimizer of~\eqref{eq:constrained_calm}.
If there exist 
$\epsilon, c > 0$ such that for any $u \in \R^{d_h}$ with $\|u\| \leq \epsilon$ and any $x$ that $\|x - x^* \| \leq \epsilon$ which satisfies $H(x) + u = 0$, one has
\begin{align}
f_0(x ) - f_0(x^* ) + c\|u\| \geq 0 .
\end{align}
Then the problem~\eqref{eq:constrained_calm} is said to be calm with $c$.  
\end{definition}

\begin{lemma}[{\citep[Theorem~3.6]{Ye1999_cq_eb_blo}}]
If the problem~\eqref{eq:constrained_calm}
is calm at a global solution $x^*$, then 
$x^*$ satisfies the KKT condition in 
Definition~\ref{def:kkt_constrained_H}.
\end{lemma}

Below is a lemma to show that if the objective is Lipschitz 
and the constraint satisfies error bound with exponent no
greater than one, then the calmness condition holds.
Similar results have been discussed 
in~\citep[Proposition~4.2]{Ye1999_cq_eb_blo} with exponent equal to one.
This result connects error bound with the calmness condition,
and thus the necessity of KKT condition.
\begin{lemma}
\label{lemma:calm_pl}
Let $x^*$ be a global minimizer of 
problem~\eqref{eq:constrained_calm}.
For $\epsilon < 1$, consider any $u\in \R^{d_h}$ 
and $\|u\| \leq \epsilon $,
and any $x$ that $\|x - x^*\| \leq \epsilon $ and $H(x) + u =0$.
Define $x_p = \mathrm{Proj}_{{\cal X}_H^*} (x)$,
where ${\cal X}_H^* = \{ x \in \R^q \mid H(x) = 0 \}$.
If $H(x)$ satisfies an error bound that 
$\|H(x)\| \geq \varrho_h\| x - x_p \|^{\eta_h}$
and $f_0$ is $\ell_f$-Lipschitz
for all $x \in \mathcal{N}(x^*, 2\epsilon) $ with $\epsilon < 1$, 
then the calmness condition in Definition~\ref{def:calm}
for problem~\eqref{eq:constrained_calm} holds.
\end{lemma}

\begin{proof}[Proof of Lemma~\ref{lemma:calm_pl}]
By definition, 
for $x \in \mathcal{N}(x^*, \epsilon)$
with $\epsilon \leq 1$,
\begin{align}
\epsilon \geq \| u \| = 
\|H(x)\| \geq \varrho_h\| x - x_p \|^{\eta_h} .
\end{align}  
Since $x^*$ is a global minimizer,
and $\|u\|\leq \epsilon \leq 1$, 
for $\eta_h \leq 1$,
\begin{align}
f_0(x) - f_0(x^*) \geq f_0(x) - f_0(x_p)   
\stackrel{(a)}{\geq} -\ell_f \|x - x_p\| 
\geq - \frac{\ell_f}{\varrho_h} \|u\|^{\frac{1}{\eta_h}}
\geq - \frac{\ell_f}{\varrho_h} \|u\| .
\end{align}
where $(a)$ holds because of the $\ell_f$-Lipschitz continuity
of $f_0$ on a bounded set that includes $x$ and $x_p$.
Therefore, the calmness condition in Definition~\ref{def:calm}
holds with $c = \frac{\ell_f}{\varrho_h}$.
\end{proof}

Next we show the connection of the 
KL and HEB conditions, and the relation of their exponents.

\begin{lemma}[KL implies HEB]
\label{lemma:kl_imply_heb}
Consider a function $p: \R^q \to \R$,
and $\min p(x) = 0$.
If $p$ satisfies the $(c_p,\alpha_p)$-KL inequality on $\Omega$
with exponent $\alpha_p > 1$,
then it also satisfies the $(\varrho_p, \eta_p)$-HEB on $\Omega$ 
with exponent $\eta_p = \frac{\alpha_p}{\alpha_p - 1}$,
and $\varrho_p = 
\Big(1 - \frac{1}{\alpha_p}\Big)^{- \frac{\alpha_p}{\alpha_p - 1}} 
\Big( \frac{1}{c_p} \Big)^{- \frac{1}{\alpha_p - 1}}$.
\end{lemma}

\begin{proof}[Proof of Lemma~\ref{lemma:kl_imply_heb}]
We first define an auxiliary function 
\begin{align}
g(x) \coloneqq \big( p(x) \big)^{1 - \frac{1}{\alpha_p}} .
\end{align}  
Since $p$ satisfies the KL inequality, then 
for any $x \notin \mathcal{X}_p^*$, and thus $p(x)\neq 0$, we have
\begin{align}
\|\nabla g(x)\|^2 
= \Big(1 - \frac{1}{\alpha_p}\Big)^2 
\Bigg\|\frac{\nabla p(x)}{ (p(x))^{\frac{1}{\alpha_p}} }  \Bigg\|^2
= \Big(1 - \frac{1}{\alpha_p}\Big)^2 
\frac{\|\nabla p(x)\|^2}{ ( p(x) )^{\frac{2}{\alpha_p}}} 
\geq \underbrace{\Big(1 - \frac{1}{\alpha_p}\Big)^2 
\Big( \frac{1}{c_p} \Big)^{\frac{2}{\alpha_p}}}_{\mu_g} .
\label{eq:grad_g_bounded_mu_g}
\end{align}
Also $p$ satisfies the KL inequality implies that
$p$ is an invex function and thus $g$ is a non-negative invex function
with a closed optimal solution set and zero optimal value.
or any point $x_0 \notin \mathcal{X}_g^*$, consider solving the following differential equation for $x(t) \notin \mathcal{X}^*_p$:
\begin{align}
\frac{d x(t)}{d t} & =-\nabla g(x(t)) \\
x(t=0) & =x_0
\end{align}
Following similar arguments for proving PL implies QG,
see e.g.,~\citep[Theorem~2, Appendix~A]{karimi2016linear_pl}, there exists a $T$ such that $x(T) \in \mathcal{X}_g^*$ (and at this point the differential equation ceases to be defined). Then
\begin{align}
g(x_0)-g(x_t) 
& =\int_{x_t}^{x_0}\langle\nabla g(x), d x\rangle 
=-\int_{x_0}^{x_t}\langle\nabla g(x), d x\rangle 
=-\int_0^T\langle\nabla g(x(t)), \frac{d x(t)}{d t}\rangle d t \\
& =\int_0^T\|\nabla g(x(t))\|^2 d t 
\geq \int_0^T  \mu_g d t = \mu_g T .
\label{eq:g_0_t_grad_square}
\end{align}
As $g(x_t ) \geq 0$, this shows we need to have $T \leq  g(x_0 ) / \mu_g$, so there must be a $T$ with $x(T) \in \mathcal{X}_g^*$. The length of the orbit $x(t)$ starting at $x_0$, which we'll denote by $\mathcal{L} (x_0 )$, is given by
\begin{align}
\mathcal{L} (x_0 ) 
=\int_0^T \|d x(t) / d t\| d t=\int_0^T\|\nabla g(x(t))\| d t 
\geq \|x_0-x_p \| 
\label{eq:grad_path_x_dist}
\end{align}
where $x_p$ is the projection of $x_0$ onto $\mathcal{X}_g^*$ and the inequality follows because the orbit is a path from $x_0$ to a point in $\mathcal{X}_g^*$ (and thus it must be at least as long as the projection distance).

Then we can further bound 
\begin{align}
  g(x_0) = g(x_0) - g(x_T) 
& \stackrel{\eqref{eq:g_0_t_grad_square}}{=} \int_0^T\|\nabla g(x(t))\|^2 d t 
\stackrel{\eqref{eq:grad_g_bounded_mu_g}}{\geq} \sqrt{\mu_g } \int_0^T\|\nabla g(x(t))\| d t 
\stackrel{\eqref{eq:grad_path_x_dist}}{\geq} \sqrt{\mu_g } \|x_0 - x_p\| .
\end{align}
The proof is complete.
\end{proof}

\begin{lemma}[KL and HEB imply EB]
\label{lemma:kl_heb_imply_eb}
For $p(x)$ that satisfies the KL inequality with exponent $\alpha_p$, 
and the HEB with exponent $\eta_p$, i.e., 
\begin{align}
\text{KL:}~ c_p \|\nabla p(x)\|^{\alpha_p} \geq p(x) ,
\quad
\text{HEB:}~ p(x)  \geq \varrho_p^{-1} 
\big( \mathrm{dist}(x, {\cal X}^*_p) \big)^{\eta_p}
\end{align}
Then it holds that
\begin{align}
\|\nabla p(x)\| \geq \varrho_h^{-1} 
\big( \mathrm{dist}(x, {\cal X}^*_p) \big)^{\eta_h}
~~\text{with}~\eta_h = \frac{\eta_p}{\alpha_p},
~\text{and}~\varrho_h = \big( \varrho_p  c_p 
\big)^{\frac{1}{\alpha_p}}.
\end{align}
\end{lemma}

\begin{proof}[Proof of Lemma~\ref{lemma:kl_heb_imply_eb}]
The proof directly follows from combining the 
two inequalities from KL and HEB that
\begin{align}
c_p \|\nabla p(x)\|^{\alpha_p} \geq p(x)
\geq \varrho_p^{-1} 
\big( \mathrm{dist}(x, {\cal X}^*_p) \big)^{\eta_p}  .
\end{align}  
Rearranging the above inequalities proves the result.
\end{proof}

\begin{proposition}
\label{prop:v_kl_imply_p_kl}
Recall that $p(x) = \big( v_{l,\tau}(x) + \tau \ln M \big)^{\theta}$
with $\theta > 0$.
If $v_{l,\tau}$  satisfies the $(c_v, \alpha_v)$-KL inequality on $\Omega$
with $\alpha_v > 1$,
then $p$ satisfies the $(c_p, \alpha_p)$-KL inequality on $\Omega$
with $\alpha_p = \frac{\theta}{\theta - 1 + \frac{1}{\alpha_v}} > 1$, 
and $c_p = \theta^{-\frac{\theta}{\theta-1 + \frac{1}{\alpha_v}}} 
\cdot c_v^{\frac{\theta}{(\theta-1) \alpha_v + 1}} $.
\end{proposition}

\begin{proof}[Proof of Proposition~\ref{prop:v_kl_imply_p_kl}]
For $x \in \Omega$ and that $v_{l,\tau}(x) + \tau \ln M > 0$,
we have
\begin{align*}
\|\nabla p(x)\| 
= & \theta (v_{l,\tau}(x) + \tau \ln M)^{\theta-1} 
\|\nabla v_{l,\tau}(x) \| \\
\geq & \theta ( c_v )^{- \frac{1}{\alpha_v}} 
\big( v_{l,\tau}(x) + \tau \ln M \big)^{\theta-1 + \frac{1}{\alpha_v}} 
= \theta ( c_v )^{- \frac{1}{\alpha_v}} 
\big( p(x) \big)^{\frac{\theta-1 + \frac{1}{\alpha_v}}{\theta}} .
\numberthis
\end{align*}  
Rearranging the above inequality proves the result.
\end{proof}

Theorem~\ref{thm:relation_stationary_solution} requires the assumption
of the $\ell_{v,2}$-smoothness of $\nabla v_{l,\tau}$ on a bounded set.
Below we provide a sufficient condition, 
which shows that under additional assumptions of $f_m, m\in[M]$, 
the $\ell_{v,2}$-smoothness of $\nabla v_{l,\tau}$ on a bounded set
${\cal X}_C$ can be justified. 
\begin{lemma}[Smoothness of $\nabla v_{l,\tau}$]
\label{lemma:smooth_nabla_vltau}
Suppose Assumption~\ref{assmp:weak_c_f} holds,
and $l + \mu > 0$.
If $\nabla f_m$ is $\ell_{f,2}$-smooth for all $m\in[M]$
on a bounded set ${\cal X}_C$, 
and there exists $x' \in {\cal X}_C$ that 
$\nabla^2 f_m$ for all $m\in[M]$ is bounded,
then $\nabla v_{l,\tau}$ is $\ell_{v,2}$-smooth
on ${\cal X}_C$,
with $\ell_{v,2} = \ell_{h_{xx},2} (1 + \ell_{y^*_{l,\tau}})
+ \ell_{y^*_{l,\tau}, 1}\ell_{h_{xy},1}
+ \ell_{y^*_{l,\tau}} \ell_{h_{xy},2} (1 + \ell_{y^*_{l,\tau}})$.
\end{lemma}

\begin{proof}[Proof of Lemma~\ref{lemma:smooth_nabla_vltau}]
With similar arguments as Lemma~\ref{lemma:f_smooth_boundset_imply_lip},
since $\nabla f_m$ is $\ell_{f,2}$-smooth for all $m\in[M]$
on a bounded set ${\cal X}_C$, 
and there exists $x' \in {\cal X}_C$ that 
$\nabla^2 f_m(x')$ for all $m\in[M]$ is bounded,
therefore, there exists $\ell_{f,1} > 0$
such that for $x \in {\cal X}_C$, and for all $m \in [M]$,
$\|\nabla^2 f_m(x)\| \leq \ell_{f,1}$.

Under Assumption~\ref{assmp:weak_c_f},
and with $l + \mu > 0$, 
first recall from~\eqref{eq:grad_v_main} 
that $\nabla v_{l,\tau}(x)$ can be computed by
\begin{align}
\nabla v_{l,\tau}(x)
= -\nabla_x h(x,y) \mid_{y = y^*_{l,\tau}(x)}
= \sum_{m=1}^M \pi_m(x, y) \nabla f_m(x) - l(x - y ) 
\mid_{y = y^*_{l,\tau}(x)}.
\end{align}
Because of the twice continuous differentiability of $f_m$ for $m\in [M]$, $\nabla^2 v_{l,\tau}(x)$ exists and can be computed by
\begin{align*}
\nabla^2 v_{l,\tau}(x)
=& -\nabla_{xx}^2 h(x,y) 
- \nabla y^*_{l,\tau}(x) \nabla_{xy}^2 h(x,y) \mid_{y = y^*_{l,\tau}(x)} .
\numberthis
\end{align*}
For simplicity, we simplify $y^*_{l,\tau}(x)$ as $y^*(x)$,
and $h_{l,\tau}(x, y)$ as $h(x, y)$ in 
the following derivations.
Then $\|\nabla^2 v_{l,\tau}(x) - \nabla^2 v_{l,\tau}(x')\|$
can be bounded by
{\small\begin{align}
\|\nabla^2 v_{l,\tau}(x) - \nabla^2 v_{l,\tau}(x')\| 
\leq & 
\underbrace{\|\nabla_{xx}^2 h(x,y^*(x)) - \nabla_{xx}^2 h(x',y^*(x'))\|}_{J_1} 
+ \underbrace{\|\nabla y^*(x) - \nabla y^*(x')\| 
\|\nabla_{xy}^2 h(x,y^*(x))\|}_{J_2} \nonumber \\
&+ \underbrace{\|\nabla y^*(x)\| 
\|\nabla_{xy}^2 h(x,y^*(x)) - \nabla_{xy}^2 h(x',y^*(x'))\|}_{J_3}
\end{align}}
where $J_1$ can be further bounded by 
\begin{align}
J_1 \leq \ell_{h_{xx},2} \big(  \|x - x'\| 
+  \|y^*(x) - y^*(x')\| \big)
\leq \ell_{h_{xx},2} (1 + \ell_{y^*_{l,\tau}})  \|x - x'\|  
\end{align}
with $\ell_{h_{xx},2}$ denoting the Lipschitz continuity
of $\nabla_{xx}^2 h(x,y)$ w.r.t. $[x;y]$.
Similarly, with $\ell_{h_{xy},2}$ denoting the Lipschitz continuity
of $\nabla_{xy}^2 h(x,y)$ w.r.t. $[x;y]$, $J_3$ can be bounded by
\begin{align}
J_3 \leq \ell_{y^*_{l,\tau}} 
\ell_{h_{xy},2} (1 + \ell_{y^*_{l,\tau}})  \|x - x'\| .
\end{align}
And $J_2$ can be bounded by
\begin{align}
J_2 \leq \ell_{y^*_{l,\tau}, 1}\ell_{h_{xy},1} \|x - x'\| .
\end{align}
Therefore, 
\begin{align}\label{eq:ell_v_2}
\|\nabla^2 v_{l,\tau}(x) - \nabla^2 v_{l,\tau}(x')\| 
\leq & \underbrace{\Big(\ell_{h_{xx},2} (1 + \ell_{y^*_{l,\tau}})
+ \ell_{y^*_{l,\tau}, 1}\ell_{h_{xy},1}
+ \ell_{y^*_{l,\tau}} \ell_{h_{xy},2} (1 + \ell_{y^*_{l,\tau}}) 
  \Big)}_{\ell_{v,2}} \|x - x'\| .
\end{align}
The derivation of $\ell_{y^*_{l,\tau}}$ is discussed in Lemma~\ref{lemma:lip_cont_ystar}.
We next discuss the derivation for $\ell_{h_{xy},1},\ell_{h_{xx},2},
\ell_{h_{xy},2}$ and $\ell_{y^*_{l,\tau}, 1}$.

To compute $\ell_{y^*_{l,\tau}, 1}$, from implicit differentiation, 
$\nabla y^*_{l,\tau}(x)$ can be computed by
\begin{align}
\nabla y^*_{l,\tau}(x) = - \nabla_{x y}^2 h(x, y^*(x)) 
\big[ \nabla_{yy}^2 h(x, y^*(x)) \big]^{-1}  .
\end{align}
Then $\|\nabla y^*_{l,\tau}(x) - \nabla y^*_{l,\tau}(x')\|$ can be bounded by
\begin{align*}
  & \|\nabla y^*_{l,\tau}(x) - \nabla y^*_{l,\tau}(x')\|
\leq \|\nabla_{x y}^2 h(x, y^*(x)) - \nabla_{x y}^2 h(x', y^*(x')) \|
\big \|\big[ \nabla_{yy}^2 h(x, y^*(x)) \big]^{-1} \big\| \\
&+ \| \nabla_{x y}^2 h(x', y^*(x')) \|
\|[\nabla_{yy}^2 h(x, y^*(x))]^{-1}\| \|[\nabla_{yy}^2 h(x', y^*(x'))]^{-1}\| 
\Big\|  \nabla_{yy}^2 h(x, y^*(x))
-  \nabla_{yy}^2 h(x', y^*(x')) \Big\| \\
\leq & \underbrace{\Big( \ell_{h_{xy},2} (l+\mu)^{-1} 
+ \ell_{h_{xy},1} (l+\mu)^{-2} \ell_{h_{yy},2} 
\Big) (1+\ell_{y^*_{l,\tau}})}_{\ell_{y^*_{l,\tau},1}}\|x - x'\|.
\numberthis\label{eq:ell_y_star_1}
\end{align*}
We then proceed to bound $\ell_{h_{xy,1}}$.
\begin{align}
\nabla_{xy}^2 h(x,y ) 
=& \sum_{m=1}^M \nabla_y \pi_m(x, y) \nabla f_m(x)^\top + l I_q  \nonumber \\
=& \frac{1}{\tau} \nabla F(y)\Big( \pi(x,y) \pi(x,y)^\top
-\mathrm{diag}(\pi(x,y))\Big) \nabla F(x)^\top 
- l I_q .
\end{align}
We have that
\begin{align}\label{eq:ell_h_xy_1}
\|\nabla_{xy}^2 h(x,y)\| \leq \ell_f + l \coloneqq \ell_{h_{xy},1}.
\end{align}
Furthermore, to bound $\ell_{h_{xy},2}$, we have
\begin{align*}
\|\nabla_{xy}^2 h(x,y) - \nabla_{xy}^2 h(x',y')\| 
\leq & \frac{1}{\tau}
\Big( \|\nabla F(y) - \nabla F(y')\|
\big\|\big( \pi(x,y) \pi(x,y)^\top
-\mathrm{diag}(\pi(x,y))\big) \nabla F(x)^\top \big\| \\
&+\|\nabla F(y')\| \|\pi(x,y) - \pi(x',y')\|
(\|\pi(x,y)\|+\|\pi(x',y')\|+1) \| \nabla F(x)\| \\
&+\big\|\nabla F(y') 
\big( \pi(x',y') \pi(x',y')^\top
-\mathrm{diag}(\pi(x',y'))\big) \big\|
\| \nabla F(x) - \nabla F(x')\|
\Big) \\
\leq & \underbrace{\frac{1}{\tau} \Big(2 \ell_{f,1} 
+ 3 \ell_f \ell_{\pi} 
\Big) M \ell_f }_{\ell_{h_{xy},2}} 
\big(\|x - x'\| + \|y - y'\| \big).
\numberthis\label{eq:ell_h_xy_2}
\end{align*}
To compute $\ell_{\pi}$, recall that
\begin{align}
\nabla_x \pi(x,y) =& -\frac{1}{\tau} \nabla F(x) \big( \mathrm{diag}(\pi(x,y))
- \pi(x,y) \pi(x,y)^\top \big) , \\
\nabla_y \pi(x,y) =& \frac{1}{\tau} \nabla F(y) \big( \mathrm{diag}(\pi(x,y))
- \pi(x,y) \pi(x,y)^\top \big)
\end{align}
from which we have
\begin{align}\label{eq:ell_pi}
\max\{ \|\nabla_x \pi(x,y)\|,  \|\nabla_y \pi(x,y)\|  \}
\leq \frac{2}{\tau} \sqrt{M} \ell_f \coloneqq \ell_{\pi}.
\end{align}
Next we bound $\ell_{h_{xx},2}$, and $\ell_{h_{yy},2}$.
The Hessian of $h(x, y)$ can be computed by
\begin{align*}
\nabla_{y y}^2 h(x, y) 
=&  -\frac{1}{\tau} \nabla F(y)\Big( \pi(x,y) \pi(x,y)^\top
-\mathrm{diag}(\pi(x,y))\Big) \nabla F(y)^\top
+ \sum_{m=1}^M \pi_m (x,y) \nabla^2 f_m(y) + l I_q, 
\numberthis \\
\nabla_{x x}^2 h(x, y) 
=& -\frac{1}{\tau} \nabla F(x)\Big( \pi(x,y) \pi(x,y)^\top
-\mathrm{diag}(\pi(x,y))\Big) \nabla F(x)^\top
+ \sum_{m=1}^M \pi_m (x,y) \nabla^2 f_m(x) + l I_q .
\numberthis
\end{align*}
Then we have
\begin{align*}
\|\nabla_{y y}^2 h(x, y) - \nabla_{y y}^2 h(x', y') \|
\leq & \frac{1}{\tau}
\Big( \|\nabla F(y) - \nabla F(y')\|
\big\|\big( \pi(x,y) \pi(x,y)^\top
-\mathrm{diag}(\pi(x,y))\big) \nabla F(y)^\top \big\| \\
&+\|\nabla F(y')\| \|\pi(x,y) - \pi(x',y')\|
(\|\pi(x,y)\|+\|\pi(x',y')\|+1) \| \nabla F(y)\| \\
&+\big\|\nabla F(y') 
\big( \pi(x',y') \pi(x',y')^\top
-\mathrm{diag}(\pi(x',y'))\big) \big\|
\| \nabla F(y) - \nabla F(y')\|
\Big) \\
&+\|\nabla^2 F(y) - \nabla^2 F(y')\| \|\pi(x,y)\|
+ \|\nabla^2 F(y')\| \|\pi(x,y) - \pi(x',y')\| \\
\leq & \underbrace{ \frac{1}{\tau} \Big(\big(2 \ell_{f,1} 
+ 3 \ell_f \ell_{\pi}  \big) 
M \ell_f 
+ \sqrt{M}\ell_{f,2} + \sqrt{M}\ell_{f,1}\ell_\pi \Big)}_{\ell_{h_{yy},2}} 
\big(\|x - x'\| + \|y - y'\| \big).
\numberthis\label{eq:ell_h_yy_2}
\end{align*}
With similar derivations as the above, we have
$\ell_{h_{xx},2} = \ell_{h_{yy},2}$.

Collecting the results 
in \eqref{eq:ell_v_2}, \eqref{eq:ell_y_star_1}, \eqref{eq:ell_h_xy_1},
\eqref{eq:ell_h_xy_2}, \eqref{eq:ell_pi}, \eqref{eq:ell_h_yy_2}
completes the proof.
\end{proof}

\begin{theorem}[Extended version of Theorem~\ref{thm:relation_stationary_solution}, 
relation of $\epsilon$-stationary solutions]
\label{thm:relation_stationary_solution_app}
Let ${\cal X} = \R^q$, and $\theta \geq 1$. 
Let $x_{\gamma}$ be a bounded $\epsilon$-stationary solution
to~\eqref{eq:approximate_penalty}. 
Define ${\cal X}_H^* = \{ x \in \R^q \mid H(x) = 0 \}$.
Suppose that on a compact subanalytic set ${\cal X}_C \subset \R^q$
with ${\cal X}_C \cap {\cal X}_H^* \neq \emptyset$
and $x_\gamma \in {\cal X}_C$, 
$\nabla p(x) $ exists and is $\ell_{p, 2}$-smooth,
$\nabla^2 p(x) $ exists, and $f_0$ is $\ell_f$-Lipschitz.
If on ${\cal X}_C$, $v(x)$ satisfies one of the following:\\
1. the $(\varrho,\eta)$-HEB, and $\theta = \frac{1}{\eta}$,
then $p(x)$ satisfies the $(\varrho_p,\eta_p)$-HEB with $\eta_p = 1$,
and if $p(x_\gamma) \leq \epsilon$, let $H(x) = p(x)$,
and $\gamma = \Theta(1)$, and $\epsilon = \delta$;\\
2. $(c_v,\alpha_v)$-KL with $\alpha_v \geq 2$, and $\theta = 1$,
then $p(x)$ satisfies the $(c_p,\alpha_p)$-KL 
with $\alpha_p \geq 2$, let $H(x) = \nabla p(x)$,
and $\gamma = \Omega(\delta^{-1})$, and $\epsilon = \delta$; \\
then  $x_{\gamma}$ is an $(\epsilon, \delta)$-KKT point to the 
problem~\eqref{eq:constrained_calm}
with a bounded $w \in \R^q$, i.e.,
\begin{equation}\label{eq:blo_kkt_hessianp_app}
\| H(x_\gamma) \| \leq \delta,  
~~ \| \nabla f_0(x_\gamma) + \nabla H(x_\gamma) w\| 
\leq \epsilon .
\end{equation}        
\end{theorem}

\begin{proof}[Proof of Theorem~\ref{thm:relation_stationary_solution_app}]
Since $x_\gamma$ is an $\epsilon$-stationary solution 
to~\eqref{eq:approximate_penalty}, thus
\begin{align}\label{eq:penalty_eps_stationary}
\| \nabla f_0(x_\gamma) + \gamma \nabla p(x_\gamma)\| 
\leq  \epsilon .
\end{align}
In the first condition, $H(x) = p(x)$,
thus $\eta_h = \eta_p = 1$,
and the calmness CQ holds based on Lemma~\ref{lemma:calm_pl}. 
This also corresponds to the exact penalization, and thus 
we can choose $w = \gamma = \Theta(1)$.
The above inequality with $ \gamma = \Theta(1)$ directly implies the KKT stationarity condition to the 
problem~\eqref{eq:constrained_calm}.

In the second condition, $H(x) = \nabla p(x)$,
by Lemma~\ref{lemma:penalty_v_stationary}, it is also an 
$\epsilon_\gamma$-stationary solution to 
$\min_{x\in \R^q} p(x) $, and thus 
\begin{align}
\|\nabla p(x_\gamma )\| 
\leq \epsilon_\gamma = \frac{\epsilon + \ell_f}{\gamma} .
\end{align}
By Lemma~\ref{lemma:kl_heb_imply_eb}, the KL condition implies that 
$\varrho_h \big( \mathrm{dist}(x_\gamma, {\cal X}_p^*\cap {\cal X}_C) \big)^{ \eta_h }
\leq \|\nabla p(x_\gamma)\| $. 
And since  ${\cal X}_p^*\cap {\cal X}_C$ is closed,
the above implies 
that there exists $x^* \in {\cal X}_p^* \cap {\cal X}_C$ such that
\begin{align}
\|x_\gamma - x^* \| = 
\mathrm{dist}(x_\gamma, {\cal X}_p^*\cap {\cal X}_C ) 
\leq \varrho_h^{-\frac{1}{\eta_h}} \|\nabla p(x_\gamma)\|^{\frac{1}{\eta_h}} 
= O(\epsilon_\gamma^{\frac{1}{\eta_h}} ) .
\end{align}
Taking Taylor expansion of $\nabla p(x)$ at $x^*$
and by the $\ell_{p,2}$-smoothness of $\nabla p(x)$ on ${\cal X}_C$, we have
\begin{align*}
\| \nabla p(x_\gamma) - \nabla^2 p(x^*) (x_\gamma - x^*) \| 
=&  \| \nabla p(x_\gamma) - \nabla p(x^*) - \nabla^2 p(x^*) (x_\gamma - x^*) \| \\ 
\leq & \ell_{p, 2} \|x_\gamma - x^*\|^2 
= O(\epsilon_\gamma^{\frac{2}{\eta_h}} ) .
\numberthis
\end{align*}
Plugging the above inequality into~\eqref{eq:penalty_eps_stationary},
we have
\begin{align}
\|\nabla f_0(x_\gamma) + \gamma \nabla^2 p(x^*) 
(x_\gamma - x^*) \|
= O( \epsilon + \gamma \epsilon_\gamma^{\frac{2}{\eta_h}} )
= O( \epsilon + \gamma^{1 - \frac{2}{\eta_h}} ).
\end{align}
Letting $w = \gamma (x_\gamma - x^*)$,
then $\|w\| = O(\gamma^{1 - \frac{1}{\eta_h}}) \leq O(1)$ is bounded
since $\eta_h \leq 1$.
We can further bound 
$\|\nabla f_0(x_\gamma) + \nabla^2 p(x_\gamma) w\|$ by
\begin{align*}
\|\nabla f_0(x_\gamma) + \nabla^2 p(x_\gamma) w\|
\leq & \|\nabla f_0(x_\gamma) + \nabla^2 p(x^*) w\|
+ \|\nabla^2 p(x_\gamma) - \nabla^2 p(x^*)\| \|w\| \\
\leq & \|\nabla f_0(x_\gamma) + \nabla^2 p(x^*) w\|
+ \ell_{p,2} \gamma \|x_\gamma - x^* \|^2  \\
=& O(\epsilon + \gamma^{1 - \frac{2}{\eta_h}} ) .
\numberthis
\end{align*}
Recall that $\eta_h \leq 1$,
thus choosing $\gamma  = \Omega (\delta^{- 1})$,
and $\epsilon = \delta$,
we have $\epsilon_\gamma \leq \delta$,
and $\|\nabla f_0(x_\gamma) + \nabla^2 p(x_\gamma) w\| = O (\epsilon)$,
which proves the result.

In the second condition, note that we require
\begin{align}
  1 \geq \eta_h = \frac{\eta_p}{\alpha_p} = \frac{1}{\alpha_p - 1}  
\end{align}
which implies $\alpha_p \geq 2$, and thus 
$\frac{\theta}{\theta - 1 + \frac{1}{\alpha_v}} \geq 2$ from Proposition~\ref{prop:v_kl_imply_p_kl}, implying $ \theta < 2$.
\end{proof}

\begin{remark}
Note that, the exponent of  the HEB
may not be unique (c.f. Appendix~\ref{sub_app:examples_global_suba}). 
In such cases, there may exist $0 < \eta_p \neq \frac{\alpha_p}{\alpha_p -1}$.
And the above theorem still requires 
$\eta_h = \frac{\eta_p}{\alpha_p} \leq 1$, thus  
$\alpha_p \geq \eta_p$.  
\end{remark}

\section{Proof of convergence of algorithms} 
\label{sec_app:proof_convergence}

In this section, we prove the convergence of the proposed algorithms
with a specific instantiation of $\theta = 1$,
thus $p(x) = v_{l,\tau}(x) + \ln M$.
For convenience, we define
$\varphi_{fh,\gamma}(x,y) 
\coloneqq f_0(x) - \gamma h_{l,\tau}(x,y)$.

The details of other oracles including 
Nesterov's acceleration
and Adam updates are given below.

\textbf{Nesterov's acceleration. } Define $\mathrm{U}(w,\Delta w;\alpha_t, t)=\operatorname{Proj}_{\mathcal{X}}(v_{t+1}+\tilde\alpha (v_{t+1}-v_t))$, where $v_{t+1}=w-\alpha_t\Delta w$ and $\tilde\alpha$ is the lookahead coefficient. 

\textbf{Adam updates. } Define $\mathrm{U}(w, \Delta w; \alpha, t)=\operatorname{Proj}_{\mathcal{X}}(w-\alpha \frac{\hat{m}_t}{\sqrt{\hat{v}_t}+\epsilon})$, where $m_t=\beta_1 m_{t-1}+(1-\beta_1) \Delta w$ denotes the moving average of gradients, $v_t=\beta_2 v_{t-1}+(1-\beta_2)(\Delta w^2)$ denotes moving average of squared gradients, $\hat{m}_t=\frac{m_t}{1-\beta_1^t}$ is bias-corrected first moment, $\hat{v}_t=\frac{v_t}{1-\beta_2^t}$ is bias-corrected second moment, with $\beta_1,\beta_2 \in (0, 1)$, $\epsilon$ is the small error.

In our convergence analysis, 
we use the PGD algorithm as an oracle. 

\subsection{Auxiliary lemmas}
\label{subs_app:auxiliary_lemma_convergence}

We first present the auxiliary lemmas to prove the convergence of the algorithms.

\begin{assumption}[Smoothness of functions]
\label{assmp:smooth_f}
For all $m\in \{0,\ldots, M\}$, 
$f_m$ is $\ell_{f,1}$-smooth on $\cal X$.  
\end{assumption}

\begin{assumption}
\label{assmp:bound_alg_x}
The sequence $\{x_t\}$ generated by Algorithm~\ref{alg:meta_v_double_loop}
is bounded on the trajectory.
\end{assumption}
The above assumption combined with Lemma~\ref{lemma:f_smooth_boundset_imply_lip}, implies that
$\{f_m(x_t)\}, m = 0,\ldots, M$ are $\ell_f$-Lipschitz 
on the trajectory.  

\begin{lemma}
\label{lemma:f_smooth_boundset_imply_lip}
Suppose Assumption~\ref{assmp:smooth_f} holds.
Given a bounded set ${\cal X}_C \subseteq {\cal X}$
such that $\|x\|\leq \ell_x$ for all $x \in {\cal X}_C$,
if there exists $x' \in {\cal X}_C$ that $\|\nabla f_m(x')\| \leq \bar{\ell}_f$,
and $|f_m(x')| < \bar{c}_f$ for all $m=0,\ldots, M$.
Then $f_m(x)$ is bounded and $\ell_f$-Lipschitz continuous for all $x \in {\cal X}_C$ with $\ell_f = \bar{\ell}_f +  2 \ell_{f,1}\ell_x$.
\end{lemma}

\begin{proof}[Proof of Lemma~\ref{lemma:f_smooth_boundset_imply_lip}]
We first prove that $f_m$  is $\ell_f$-Lipschitz continuous on ${\cal X}_C$.
For all $x \in {\cal X}_C$, it holds that
\begin{align}
\|\nabla f_m(x)\| \leq \|\nabla f_m(x')\| + \|\nabla f_m(x) - \nabla f_m(x')\|
\leq \bar{\ell}_f + \ell_{f,1} \|x - x'\|
\leq \underbrace{\bar{\ell}_f +  2 \ell_{f,1}\ell_x}_{\ell_f}.
\end{align}
Then we can further bound $|f_m(x)|$ by
\begin{align}
|f_m(x)| \leq |f_m(x')| + |f_m(x) - f_m(x')| 
\leq  \bar{c}_f + \ell_f \|x - x'\|
\leq  \bar{c}_f + 2\ell_f \ell_x .
\end{align}
The above holds for all $m=0,\ldots, M$, the proof is complete.
\end{proof}

\begin{lemma}[Smoothness of $h_{l,\tau}$ and $v_{l,\tau}$]
\label{lemma:smooth_h_v}
Under the same settings as Lemma~\ref{lemma:f_smooth_boundset_imply_lip},
then $h_{l,\tau}(x,y)$ is $\ell_{h_{l,\tau},1}$-smooth on ${\cal X}_C$ 
w.r.t. both 
$x$ and $y$ with  $\ell_{h_{l,\tau},1} = l+\ell_{f,1} + 2 e^{\frac{\ell_f \ell_x}{\tau}} \frac{\ell_f^2 }{\tau}  $.
And $v_{l,\tau}(x)$ is $\ell_{v_{l,\tau},1}$-smooth on ${\cal X}_C$
with $ \ell_{v_{l,\tau},1}= \ell_{h_{l,\tau},1}( 
  1 + \ell_{y^*_{l,\tau}} )$.
\end{lemma}

\begin{proof}[Proof of Lemma~\ref{lemma:smooth_h_v}]
The gradient of $h_{l,\tau}(x,y)$ w.r.t. $x$ can be computed by 
\begin{align}
\nabla_x h_{l,\tau}(x,y)
=  - \sum_{m=1}^M \pi_m(x, y) \nabla f_m(x) + l(x - y) .
\end{align} 
Given $x, x', y \in {\cal X}$, we can bound 
$\|\nabla_x h_{l,\tau}(x,y) - \nabla_x h_{l,\tau}(x',y)\|$ by
\begin{align*}
& \|\nabla_x h_{l,\tau}(x,y) - \nabla_x h_{l,\tau}(x',y)\| \\
\leq &  l\|x - x' \|
+ \sum_{m=1}^M \pi_m(x,y) \|\nabla f_m(x) - \nabla f_m(x')\|
+ \sum_{m=1}^M \big \| \pi_m(x,y) - \pi_m(x',y) \big \| \|\nabla f_m(x')\| \\
\leq & l\|x - x' \|
+ \ell_{f,1} \|x - x' \| + 
\ell_f \sum_{m=1}^M \big \| \pi_m(x,y) - \pi_m(x',y) \big \|
\numberthis
\end{align*}
where we bound $\| \pi_m(x,y) - \pi_m(x',y) \big \|$ by 
\begin{align*}
  & \| \pi_m(x,y) - \pi_m(x',y) \|
= \Bigg\| \frac{e^{\frac{f_m(y) - f_m(x)}{\tau}}}
{\sum_{m=1}^M e^{\frac{f_m(y) - f_m(x)}{\tau}}}
- \frac{e^{\frac{f_m(y) - f_m(x')}{\tau}}}
{\sum_{m=1}^M e^{\frac{f_m(y) - f_m(x')}{\tau}}} \Bigg\| \\
\leq & \Bigg\| \frac{e^{\frac{f_m(y) - f_m(x)}{\tau}}}
{\sum_{m=1}^M e^{\frac{f_m(y) - f_m(x)}{\tau}}}
- \frac{e^{\frac{f_m(y) - f_m(x')}{\tau}}}
{\sum_{m=1}^M e^{\frac{f_m(y) - f_m(x)}{\tau}}} \Bigg\|
+ \Bigg\| \frac{e^{\frac{f_m(y) - f_m(x')}{\tau}}}
{\sum_{m=1}^M e^{\frac{f_m(y) - f_m(x)}{\tau}}}
- \frac{e^{\frac{f_m(y) - f_m(x')}{\tau}}}
{\sum_{m=1}^M e^{\frac{f_m(y) - f_m(x')}{\tau}}} \Bigg\| \\
\leq & 
\Bigg\| \frac{e^{\frac{f_m(y) - f_m(\tilde{x}) }{\tau}}}
{\sum_{m=1}^M e^{\frac{f_m(y) - f_m(x)}{\tau}}} \Bigg\| 
\frac{\ell_f}{\tau} \|x - x'\| +
\Bigg\| \frac{e^{\frac{f_m(y) - f_m(x')}{\tau}}
\sum_{m=1}^M e^{\frac{f_m(y) - f_m(\tilde{x}) }{\tau}} }
{\Big( \sum_{m=1}^M e^{\frac{f_m(y) - f_m(x)}{\tau}}\Big)
\Big( \sum_{m=1}^M e^{\frac{f_m(y) - f_m(x')}{\tau}}\Big)} \Bigg\|
\frac{\ell_f}{\tau} \|x - x'\| 
\numberthis
\end{align*}
where $\tilde{x}$ is on the line segment of $x$ and $x'$.
Taking $\sum_{m=1}^M$ of the above inequality yields
\begin{align*}
& \sum_{m=1}^M \| \pi_m(x,y) - \pi_m(x',y) \| \\
\leq & 
\frac{\sum_{m=1}^M e^{\frac{f_m(y) - f_m(\tilde{x}) }{\tau}}}
{\sum_{m=1}^M e^{\frac{f_m(y) - f_m(x)}{\tau}}} 
\cdot \frac{\ell_f}{\tau} \|x - x'\|+
\frac{ \sum_{m=1}^M e^{\frac{f_m(y) - f_m(\tilde{x}) }{\tau}} }
{ \sum_{m=1}^M e^{\frac{f_m(y) - f_m(x)}{\tau}} } 
\cdot \frac{\ell_f}{\tau} \|x - x'\| \\
\leq & 2 e^{\frac{\ell_f \ell_x}{\tau}}
\frac{\ell_f}{\tau} \|x - x'\|
\numberthis
\end{align*}
where the last inequality uses the fact that 
$\|f_m(\tilde{x}) - f_m(x)\| \leq \ell_f \ell_x$.
Combining the above arguments yields
\begin{align}
\|\nabla_x h_{l,\tau}(x,y) - \nabla_x h_{l,\tau}(x',y)\| 
\leq &  \ell_{h_{l,\tau},1} \|x - x' \|
\end{align}
with $\ell_{h_{l,\tau},1} = l+\ell_{f,1} + 2 e^{\frac{\ell_f \ell_x}{\tau}}
\frac{\ell_f^2 }{\tau}  $.

Similarly, given $x, y, y' \in {\cal X}$, we can bound 
$\|\nabla h_{l,\tau}(x,y) - \nabla h_{l,\tau}(x,y')\|$ by
\begin{align}
\|\nabla_x h_{l,\tau}(x,y) - \nabla_x h_{l,\tau}(x,y')\| 
\leq &  \ell_{h_{l,\tau},1} \|y - y' \| \\ 
\|\nabla_y h_{l,\tau}(x,y) - \nabla_y h_{l,\tau}(x,y')\| 
\leq &  \ell_{h_{l,\tau},1} \|y - y' \| .
\end{align}

The gradient of $v_{l,\tau}(x)$  can be computed by 
\begin{align}
\nabla v_{l,\tau}(x )
= \sum_{m=1}^M \pi_m(x, y^*_{l,\tau}(x)) \nabla f_m(x) 
+ l(x - y^*_{l,\tau}(x) ) .
\end{align} 
Given $x, x' \in {\cal X}$, we can bound 
$\|\nabla v_{l,\tau}(x) - \nabla v_{l,\tau}(x')\|$ by
\begin{align*}
& \|\nabla v_{l,\tau}(x) - \nabla v_{l,\tau}(x')\| 
\leq \ell_{h_{l,\tau},1}(\|x - x' \|
+ \|y^*_{l,\tau}(x) - y^*_{l,\tau}(x')\|)
\leq \ell_{h_{l,\tau},1}( 
  1 + \ell_{y^*_{l,\tau}} ) \|x - x' \| .
\numberthis
\end{align*}
The proof is complete.
\end{proof}

\begin{lemma}[Smoothness of the penalized function]
\label{lemma:smoothness_penalized_func}
Suppose Assumptions~\ref{assmp:smooth_f}, \ref{assmp:bound_alg_x} hold.
Then $\varphi_{fh, \gamma}$ is $\ell_{\varphi_{fh, \gamma},1}$-smooth w.r.t. $x$ and $y$ on the trajectory,
with $\ell_{\varphi_{fh, \gamma},1} = 
\ell_{f,1} + \gamma \ell_{h_{l,\tau},1}$. 
\end{lemma}

\begin{proof}[Proof of Lemma~\ref{lemma:smoothness_penalized_func}]
Given $x, x', y, y'\in {\cal X}$, we can bound the difference $\|\nabla_x \varphi_{fh, \gamma}(x,y) - \nabla_x \varphi_{fh, \gamma}(x',y) \|$ by
\begin{align}
\|\nabla_x \varphi_{fh, \gamma}(x,y) - \nabla_x \varphi_{fh, \gamma}(x',y) \| 
\leq & \|\nabla f_0(x) - \nabla f_0(x')\| 
+ \gamma \| \nabla h_{l,\tau}(x,y) - \nabla h_{l,\tau}(x',y) \|  
\nonumber \\
\leq & \ell_{f,1} \|x - x'\| + \gamma \ell_{h_{l,\tau},1}\|x - x'\| .
\end{align}
Similarly,  we can bound the difference $\|\nabla_y \varphi_{fh, \gamma}(x,y) - \nabla_y \varphi_{fh, \gamma}(x,y') \|$ by
\begin{align}
\|\nabla_y \varphi_{fh, \gamma}(x,y) - \nabla_y \varphi_{fh, \gamma}(x,y') \| 
\leq &  \gamma \| \nabla h_{l,\tau}(x,y) - \nabla h_{l,\tau}(x,y') \| 
\leq \gamma \ell_{h_{l,\tau},1} \|y - y'\| .  
\end{align}
The proof is complete.
\end{proof}

\begin{lemma}[Contraction of $y_{t,k}$]
\label{lemma:contraction_ytk}
Suppose Assumptions~\ref{assmp:smooth_f}, \ref{assmp:bound_alg_x} hold, and $l - \ell_{f,1} \geq \mu_{h_y} > 0$.
Recall that $y^*_{l,\tau}(x) \coloneqq {\arg\min}_y h_{l,\tau}(x,y)$.
The sequence $\{y_{t,k}\}_{k=1}^K$ produced by Algorithm~\ref{alg:meta_v_double_loop} satisfies
\begin{align}
\|y_{t,k+1} - y_{l,\tau}^*(x_t)\|^2
\leq (1 - \mu_{h_y} \beta_{t,k} ) \|y_{t,k} - y_{l,\tau}^*(x_t)\|^2 .
\end{align} 
\end{lemma}

\begin{proof}[Proof of Lemma~\ref{lemma:contraction_ytk}]
\label{proof:contraction_ytk}
Recall that the update of $y_{t,k}$ in \eqref{eq:update_ytk} takes the projected gradient descent (PGD) on $h_{l,\tau}(x,y) $.
By Corollary~\ref{crlr:h_ltau_unique_y}, 
for $l + \min_{m\in [M]} \mu_m \geq \mu_{h_y} > 0$, 
the function $h_{l,\tau}(x,y) $ is $\mu_{h_y}$-strongly convex w.r.t. $y$.

Leveraging the convergence result of PGD on strongly convex functions,
we have
\begin{align}
\|y_{t,k+1} - y_{l,\tau}^*(x_t)\|^2
\leq (1 - \mu_{h_y} \beta_{t,k} ) \|y_{t,k} - y_{l,\tau}^*(x_t)\|^2 .
\end{align} 
The proof is complete.
\end{proof}

\begin{corollary}
\label{crlr:convergence_ytk}
Suppose Assumptions~\ref{assmp:smooth_f},~\ref{assmp:bound_alg_x}  hold, and $l - \ell_{f,1} \geq \mu_{h_y} > 0$.
Recall that $y^*_{l,\tau}(x) \coloneqq {\arg\min}_y h_{l,\tau}(x,y)$.
The sequence $\{y_{t,k}\}_{k=1}^K$ produced by Algorithm~\ref{alg:meta_v_double_loop} satisfies
\begin{align}
\|y_{t,K} - y_{l,\tau}^*(x_t)\|^2 \leq 
\prod_{k=0}^{K-1}(1 - \mu_{h_y} \beta_{t,k} ) \|y_{t,0} - y_{l,\tau}^*(x_t)\|^2 .
\end{align}
\end{corollary}

\begin{proof}[Proof of Corollary~\ref{crlr:convergence_ytk}]
The result directly follows from the update of $y_{t,k}$ in \eqref{eq:update_ytk}, and by applying Lemma~\ref{lemma:contraction_ytk}
iteratively from $k=0,\ldots, K-1$.
\end{proof}

\subsection{Convergence of the meta algorithm}
\label{subs_app:meta_alg_convergence}

\begin{theorem}[Convergence of Algorithm~\ref{alg:meta_v_double_loop}
  with projected gradient descent]
\label{thm:convergence_meta_alg}
Suppose Assumptions~\ref{assmp:smooth_f} and~\ref{assmp:bound_alg_x} hold.
The sequence $\{x_{t}, y_{t}\}_{t=0}^T$ produced by Algorithm~\ref{alg:meta_v_double_loop} with $\alpha_t = \alpha = \Theta(1)$, 
$\beta_t = \beta = \Theta(1)$, 
$\gamma_t = O(1 + t)$, $K_t = O(1+t)$ satisfies
{\begin{align}\label{eq:convergence_PPv_total}
& \frac{1}{T}\sum_{t=0}^{T-1} \frac{1}{\alpha_t^2}
\Big\|x_t - \mathrm{Proj}_{\cal X} \big( x_t  - \alpha_t \nabla \varphi_{\gamma_t}(x_t) \big) \Big\|^2
= O \Big( \frac{1}{T} \Big),~\text{and} \nonumber \\
& \frac{1}{T} \sum_{t=0}^{T-1} \frac{1}{\gamma_t^2 \alpha_t^2}
\Big\|x_t - \mathrm{Proj}_{\cal X} \big( x_t  - \alpha_t \gamma_t 
\nabla v_{l,\tau}(x_t) \big) \Big\|^2
= O \Big( \frac{1}{T} \Big) .
\end{align}}
\end{theorem}

\begin{proof}[Proof of Theorem~\ref{thm:convergence_meta_alg}]
We first prove~\eqref{eq:convergence_PPv_total}, the convergence of the penalty reformulation~\eqref{eq:approximate_penalty}.
Recall that at each outer-loop iteration, Algorithm~\ref{alg:meta_v_double_loop} does the following update
\begin{align}\label{eq:x_tplus_update_proof}
&x_{t+1} = \mathrm{Proj}_{\cal X} 
\big(x_{t} - \alpha_{t} (\nabla f_0(x_t ) -\gamma_t \nabla_x h_{l,\tau}(x_t , y_{t+1})) \big)  
\end{align}
where $y_{t+1} = y_{t,K}$ approximates $y_{l,\tau}^*(x_t)$ with sufficiently large $K$ based on Corollary~\ref{crlr:convergence_ytk}.
Choosing $\beta_{t,k} = \beta_t \leq 1/\mu_{h_y}$ for all $k = 0,\ldots, K-1$, 
 it then follows that 
\begin{align}\label{eq:y_t_converge_inner}
\|y_{t+1} - y_{l,\tau}^*(x_t)\|^2 =
\|y_{t,K} - y_{l,\tau}^*(x_t)\|^2 \leq 
(1 - \mu_{h_y} \beta_{t} )^K \|y_{t,0} - y_{l,\tau}^*(x_t)\|^2 .
\end{align}
Let $\ell_{fh,1,t}$ denote the smoothness constant for 
$ f_0(x ) - \gamma_t h_{l,\tau}(x , y)$.
Define the Lyapunov function $\mathbb{V}_t$ to be
\begin{align}
\mathbb{V}_t \coloneqq 
f_0(x_t ) - \gamma_t h_{l,\tau}(x_t , y_{t+1}) .
\end{align}
Applying the convergence of PGD for general nonconvex smooth objective yields
\begin{align}\label{eq:PGD_xt_bound}
\mathbb{V}_{t+1} - \mathbb{V}_t  \leq 
\langle \nabla f_0(x_t ) - \gamma_t \nabla h_{l,\tau}(x_t , y_{t+1}) , 
x_{t+1} - x_t \rangle
+ \frac{\ell_{fh,1, t}}{2 } \|x_{t+1} - x_t \|^2 .
\end{align} 
By the property of projection, and the update of $x_t$,
 we further have 
\begin{align}\label{eq:proj_property_descent}
\langle \nabla f_0(x_t ) - \gamma_t \nabla h_{l,\tau}(x_t , y_{t+1}) , 
x_{t+1} - x_t \rangle
\leq -\frac{1}{\alpha_t} \| x_{t+1} - x_t \|^2 
\end{align}
Since $\alpha_t \leq 1/\ell_{fh,1, t}$, plugging the above inequality back into~\eqref{eq:PGD_xt_bound} and rearranging yield
\begin{align}\label{eq:V_diff_x_diff}
  \mathbb{V}_{t+1} - \mathbb{V}_t 
\leq  - \frac{1}{2 \alpha_t} \|x_{t+1} - x_t \|^2  
\end{align}
Recall from \eqref{eq:x_tplus_update_proof} and \eqref{eq:y_t_converge_inner} that $x_{t+1} $ 
is an approximation to $\mathrm{Proj}_{\cal X} \big( x_t  - \alpha_t \nabla \varphi_{\gamma_t}(x_t) \big)  =  \mathrm{Proj}_{\cal X} \big( x_t  - \alpha_t (\nabla f_0(x_t ) -\gamma_t \nabla_x h_{l,\tau}(x_t , y_{l,\tau}^*(x_t))) \big)$, 
since $y_{t+1}$ is an approximation to $y_{l,\tau}^*(x_t)$.
Therefore, the term $\big\| x_t - \mathrm{Proj}_{\cal X} \big( x_t  - \alpha_t \nabla \varphi_{\gamma_t}(x_t) \big) \big\|^2 $   can be further decomposed as
{\begin{align*}
  & \big\| x_t - \mathrm{Proj}_{\cal X} \big( x_t  - \alpha_t \nabla \varphi_{\gamma_t}(x_t) \big) \big\|^2 =  
\Big\| x_t - \mathrm{Proj}_{\cal X} 
\big(x_{t} - \alpha_{t} (\nabla f_0(x_t ) - \gamma_t \nabla_x h_{l,\tau}(x_t , y_{l,\tau}^*(x_t ) )) \big)  \Big \|^2 \\
\leq & 2 \Big\| x_t - \mathrm{Proj}_{\cal X} 
\big(x_{t} - \alpha_{t} (\nabla f_0(x_t ) - \gamma_t \nabla_x h_{l,\tau}(x_t , y_{t+1} )) \big)  \Big \|^2
+ 2 \gamma_t^2 \| \nabla_x h_{l,\tau}(x_t , y_{t+1} )
- \nabla_x h_{l,\tau}(x_t , y_{l,\tau}^*(x_t ) ) \|^2 \\
\leq &  2 \| x_t -   x_{t+1} \|^2
+ 2 \gamma_t^2 \ell_{h_{l,\tau}, 1}^2 \|y_{t+1} - y_{l,\tau}^*(x_t )\|^2 \\
\leq &  2 \| x_t - x_{t+1}  \|^2
+ 2 \gamma_t^2 \ell_{h_{l,\tau}, 1}^2 \epsilon_{y,t}^2 
\numberthis \label{eq:bound_y_approx_err}
\end{align*}}
where $\epsilon_{y,t}^2 = (1 - \mu_{h_y} \beta_{t} )^{K_t} \|y_{t,0} - y_{l,\tau}^*(x_t)\|^2  $ .

Plugging~\eqref{eq:bound_y_approx_err} into~\eqref{eq:V_diff_x_diff}
and rearranging, we have
\begin{align*}
\big\| x_t - \mathrm{Proj}_{\cal X} \big( x_t  - \alpha_t \nabla \varphi_{\gamma_t}(x_t) \big) \big\|^2
\leq & 2 \| x_t - x_{t+1}  \|^2
+ 2 \gamma_t^2 \ell_{h_{l,\tau}, 1}^2 \epsilon_{y,t}^2 \\
\leq & 4 \alpha_t ( \mathbb{V}_t - \mathbb{V}_{t+1}  ) 
+ 2 \gamma_t^2 \ell_{h_{l,\tau}, 1}^2 \epsilon_{y,t}^2 .
\numberthis
\end{align*}
Taking telescoping sum of the above inequality, 
and choosing $\alpha_t = \alpha = \Theta(1)$, 
$\beta_t = \beta = \Theta(1)$, 
$\gamma_t = O(1 + t)$, $K_t = O(1+t)$
prove the first result.

For the second result, we can derive that
\begin{align*}
& \big\| x_t - \mathrm{Proj}_{\cal X} \big( x_t  - \alpha_t \gamma_t \nabla v_{l,\tau}(x_t) \big) \big\|^2 \\
\leq & 2 \big\| x_t - \mathrm{Proj}_{\cal X} \big( x_t  - \alpha_t \nabla \varphi_{\gamma_t}(x_t) \big) \big\|^2
+ 2 \big\| \mathrm{Proj}_{\cal X} \big( x_t  - \alpha_t \nabla \varphi_{\gamma_t}(x_t) \big) - \mathrm{Proj}_{\cal X} \big( x_t  - \alpha_t \gamma_t \nabla v_{l,\tau}(x_t) \big) \big\|^2 \\
\leq & 2 \big\| x_t - \mathrm{Proj}_{\cal X} \big( x_t  - \alpha_t \nabla \varphi_{\gamma_t}(x_t) \big) \big\|^2
+ 2 \big\| \alpha_t \nabla f_0(x_t)  \big\|^2 \\
\leq & 8 \alpha_t ( \mathbb{V}_t - \mathbb{V}_{t+1}  ) 
+ 4 \gamma_t^2 \ell_{h_{l,\tau}, 1}^2 \epsilon_{y,t}^2
+ 2 \alpha_t^2 \ell_f^2 .
\numberthis
\end{align*}
Dividing both sides of the above inequality by $\gamma_t^2 \alpha_t^2$,
and taking telescoping sum prove the second result.
\end{proof}

\section{Implementation details and additional experiments}
\label{sec_app:experiment}

In this section, we report the additional implementation details and additional experimental results and discussion omitted from the main text.

\paragraph{Computation.}
All experiments were conducted on a server with an Intel i9-7920X CPU, and one NVIDIA A5000 GPU. Some experiments require CPU only.

For all the experiments reported in the main text except for the multi-lingual speech recognition experiment, we exactly follow the settings from~\citep{mahapatra2020multi}.
For the multi-lingual speech recognition experiment, we follow 
the settings from~\cite{chen2024ferero}.
The implementations of the baselines including LS, PMTL, and EPO  are from the official code of the EPO~\citep{mahapatra2020multi} and that of FERERO~\cite{chen2024ferero} is from its official code with their default hyperparameters.
The results of XWC-MGDA are directly referenced from the paper.

\paragraph{Synthetic data.}
For the results in both Figure~\ref{fig:syn_init_close} and Figure~\ref{fig:synthetic_concave_exponential},  the model parameter $x$ has dimension $q = 20$, the number of objectives is $M = 2$.
The angles between the preference vectors and the horizontal axis are generated  between $[\frac{1}{20}\pi, \frac{9}{20}\pi]$ with equal angular distance.
The optimization methods are all deterministic in this experiment.
We use the default parameters for all baseline methods.
For the FOOPS method, we use hyperparameters 
$\theta=1, l = 1, \tau = 0.01$.
The penalty parameter $\gamma_t = \min\{0.05 + 0.01 t, 1.5 \}$.
The inner-loop parameters are $K = 100, \beta = 0.1$.

\begin{table}[ht]
\caption{Summary of hyper-parameters for the synthetic data experiments in Figure~\ref{fig:synthetic_concave_exponential}.}
\label{tab:hp-synthetic}
\centering
\begin{tabular}{c| cccc ccc }
\toprule 
Hyperparameters &LS & MGDA & PMTL & EPO 
& FERERO & FOOPS (ours)  \\
\midrule
step size $\alpha_t$ 
& 0.1 & 0.2 & 0.2 &0.1 &0.05 &0.2 \\ 
max iterations &150 &150 &150 &100 &100 &100  \\
\bottomrule
\end{tabular}
\end{table}

\begin{figure*}[ht]
\centering
\begin{subfigure}[b]{0.16\textwidth}
  \centering
  \includegraphics[width=.98\linewidth]{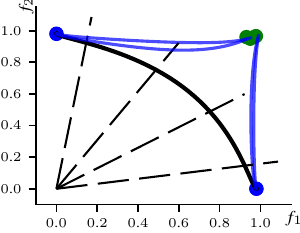}
  \caption{LS}
  \label{sfig:ls_toy}
\end{subfigure}
\begin{subfigure}[b]{0.16\textwidth}
  \centering
  \includegraphics[width=.98\linewidth]{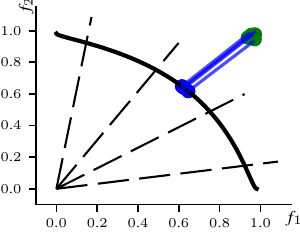}
  \caption{MGDA}
  \label{sfig:MGDA_toy}  
\end{subfigure}
\begin{subfigure}[b]{0.16\textwidth}
  \centering
  \includegraphics[width=.98\linewidth]{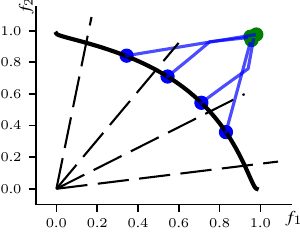}
  \caption{PMTL}
  \label{sfig:PMTL_toy}  
\end{subfigure}
\begin{subfigure}[b]{0.16\textwidth}
  \centering
  \includegraphics[width=.98\linewidth]{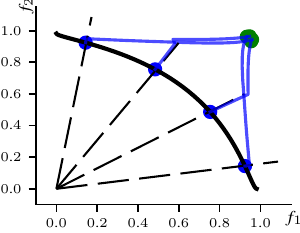}
  \caption{EPO}
  \label{sfig:EPO_toy}  
\end{subfigure}
\begin{subfigure}[b]{0.16\textwidth}
  \centering
  \includegraphics[width=.98\linewidth]{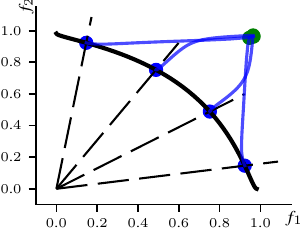}
  \caption{FERERO}
  \label{sfig:FERERO_toy}  
\end{subfigure}
\begin{subfigure}[b]{0.16\textwidth}
  \centering
  \includegraphics[width=.98\linewidth]{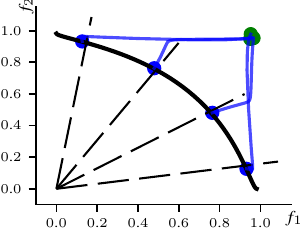}
  \caption{FOOPS}
  \label{sfig:FOOPS_toy}  
\end{subfigure}
\caption{Converging solutions (blue dots) and optimization trajectories (blue lines) on the objective space of different methods on synthetic objectives given in~\eqref{eq:obj_synthetic1}.
Dashed black arrows represent pre-specified preference vectors. The green dots represent initial objective values.
}
\label{fig:synthetic_concave_exponential}  
\end{figure*}

\begin{figure*}[ht]
\centering
\begin{subfigure}[b]{0.23\textwidth}
\centering
\includegraphics[width=.98\linewidth]{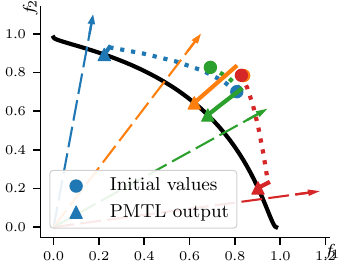}
\caption{PMTL}
\label{sfig:syn_init_close_pmtl2}  
\end{subfigure}
\begin{subfigure}[b]{0.23\textwidth}
\centering
\includegraphics[width=.98\linewidth]{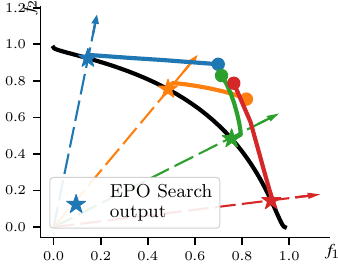}
\caption{EPO}
\label{sfig:syn_init_close_epo2}  
\end{subfigure}
\begin{subfigure}[b]{0.23\textwidth}
\centering
\includegraphics[width=.98\linewidth]{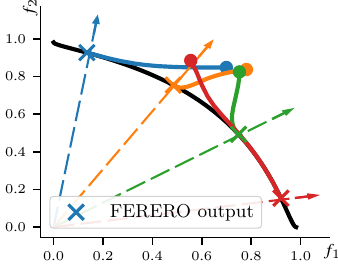}
\caption{FERERO ($A=I$)}
\label{sfig:syn_init_close_easy_ferero_app}  
\end{subfigure}
\begin{subfigure}[b]{0.23\textwidth}
\centering
\includegraphics[width=.98\linewidth]{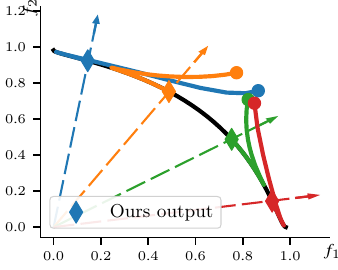}
\caption{FOOPS}
\label{sfig:syn_init_close_ours2}  
\end{subfigure}
\begin{subfigure}[b]{0.23\textwidth}
\centering
\includegraphics[width=.98\linewidth]{figures/synthetic/PMTL_hard_init.pdf}
\caption{PMTL}
\label{sfig:syn_init_close_pmtl_app}  
\end{subfigure}
\begin{subfigure}[b]{0.23\textwidth}
\centering
\includegraphics[width=.98\linewidth]{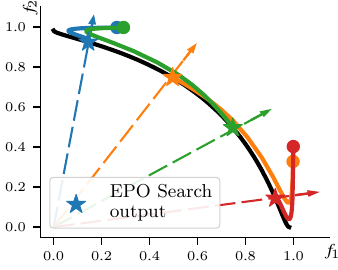}
\caption{EPO}
\label{sfig:syn_init_close_epo_app}  
\end{subfigure}
\begin{subfigure}[b]{0.23\textwidth}
\centering
\includegraphics[width=.98\linewidth]{figures/synthetic/FERERO_hard_init.pdf}
\caption{FERERO ($A=I$)}
\label{sfig:syn_init_close_ferero_app}  
\end{subfigure}
\begin{subfigure}[b]{0.23\textwidth}
\centering
\includegraphics[width=.98\linewidth]{figures/synthetic/FOOPS_hard_init.pdf}
\caption{FOOPS}
\label{sfig:syn_init_close_ours_app}  
\end{subfigure}
\caption{Extension of Figure~\ref{fig:syn_init_close}. Outputs (colored markers) and optimization trajectories (colored lines) of different methods when initial objectives are near the Pareto front. Different colors represent different preferences.
For FERERO, $A=I$ represents choosing the partial order cone $C_A = \R_+^M$ therein for vector optimization, which corresponds to Pareto dominance.}
\label{fig:syn_init_close_app}  
\end{figure*}

In Figure~\ref{fig:synthetic_concave_exponential},
we include additional results from LS and MGDA for comparison.
For all preferences and all methods, the initial model parameter $x_0$ is randomly generated from a Gaussian distribution $\mathcal{N}(0,1)$ for each dimension.
In Table~\ref{tab:hp-synthetic}, we provide a summary of the hyperparameters for the baselines and our methods for the experiments in Figure~\ref{fig:synthetic_concave_exponential}.

Figure~\ref{fig:syn_init_close_app} is an extension of Figure~\ref{fig:syn_init_close} under the same experiment objectives
but with additional results from EPO and by different initializations.
In Figures~\ref{sfig:syn_init_close_pmtl2}-\ref{sfig:syn_init_close_ours2}, the initial model parameters are randomly generated from a uniform distribution between $[-0.3, 0.3]$ for each dimension.
While in Figures~\ref{sfig:syn_init_close_pmtl}-\ref{sfig:syn_init_close_ours}
and Figures~\ref{sfig:syn_init_close_pmtl_app}-\ref{sfig:syn_init_close_ours_app}, the initial model parameters are randomly generated from a uniform distribution between $[-0.5, -0.15]$ or $[0.15, 0.5]$ for each dimension. 
Table~\ref{tab:hp-synthetic-init-close} summarizes the hyperparameters for the experiments in Figure~\ref{fig:syn_init_close}.
Compared to other baselines, our method is more robust to initializations
and requires the least number of iterations for the hard initialization
in Figures~\ref{sfig:syn_init_close_pmtl_app}-\ref{sfig:syn_init_close_ours_app}.
\begin{table}[ht]
\caption{Summary of hyper-parameters for the synthetic data experiments in Figure~\ref{fig:syn_init_close_app}.}
\label{tab:hp-synthetic-init-close}
\small
\centering
\begin{tabular}{c| cccc |cccc  }
\toprule 
Hyperparameters & \multicolumn{4}{c}{Figures~\ref{sfig:syn_init_close_pmtl2}-\ref{sfig:syn_init_close_ours2}}
& \multicolumn{4}{c}{Figures~\ref{sfig:syn_init_close_pmtl_app}-\ref{sfig:syn_init_close_ours_app}}\\
\cline{2-9}
 & PMTL & EPO & FERERO & FOOPS
 & PMTL & EPO & FERERO & FOOPS \\
\midrule
step size $\alpha_t$ & 0.25 & 0.10 & 0.60 & 0.20
& 0.50 & 0.20 & 0.60 & 0.20 \\ 
max iterations &100 & 60 & 10 & 100
&200 & 120 & 200 & 100 \\
\bottomrule
\end{tabular}
\end{table}

\paragraph{Multi-patch image classification.}
For a fair comparison, we follow the same data splitting and processing procedures as~\citep{mahapatra2020multi}.
In each of the three datasets, there are 120k samples for training and 20k samples for testing. There are two tasks on each dataset: 1) classifying the top-left image, and 2) classifying the bottom-right image.
For all methods, we use the SGD optimizer with batch size  256. 
The step sizes for updating the model parameters of all methods are $10^{-3}$. 
The number of epochs for all methods are $100$.
The parameters for other methods are chosen as default.
For the FOOPS method, we use hyperparameters 
$\theta=1, l = 0.6, \tau = 0.01$.
The penalty parameter $\gamma_t $
is set initially to $0.1$ and increased by $0.1$ after every 10 epochs
until it reaches $2$.
The inner-loop parameters are $K = 5, \beta = 10^{-3}$.

We use the Pymoo 0.6.1 library to compute the hypervolume.
The Nadir points for the hypervolume computation are given in Table~\ref{tab:nadir_points}.
For a fair comparison, the  Nadir points we use are the same with~\citep{momma2022multi,chen2024ferero}.

\begin{table}[ht]
\caption{Nadir points for the hypervolume computation }
\label{tab:nadir_points}
\small
\centering
\begin{tabular}{l| cccc   }
\toprule
Dataset and metrics & Nadir points, metrics on objective [$1,\ldots, M$] \\
\midrule
Multi-MNIST loss     & [0.500, 0.450] \\
Multi-Fashion loss   & [0.840, 0.800] \\
Multi-F+M loss       & [0.625, 0.575] \\
Multi-MNIST accuracy & [0.830, 0.848] \\
Multi-Fashion accuracy & [0.680, 0.710]\\
Multi-F+M accuracy    & [0.790, 0.785] \\
\bottomrule 
\end{tabular}
\end{table}

\paragraph{Multi-lingual speech recognition.}
We follow the same experiment settings in~\cite{chen2024ferero}.
We use two datasets, Librispeech and AISHELL v1.
Librispeech is an English speech dataset that consists of 960 hours of labeled audio data. For our experiments, we use the ``train-clean-100'' subset of the Librispeech dataset for supervised training, which contains 100 hours of clean training data. Additionally, we use the full 960 hours of data for self-supervised training. AISHELL v1 is a 178-hour  Mandarin speech corpus designed for various speech and speaker processing tasks. We use the full AISHELL v1 dataset for both self-supervised and supervised training. We combine these two datasets for our multi-lingual speech recognition experiments.

We use the conformer~\citep{gulati2020conformer} model with 8 conformer blocks as the encoder. Each block contains 512 hidden units and 8 attention heads. Each attention head has dimension 64. The convolutional kernel size is 31. Two classification heads are used. They contain two linear layers, one with 1000 output size for English, and another with 5000 output size for Chinese.
The total number of parameters is around 64.5M with 58.4M encoder layer parameters and the rest being the classification layer parameters.

The loss functions we use include the Contrastive Predictive Coding (CPC) loss, and the Connectionist Temporal Classification (CTC) loss. 
The \emph{CPC loss}~\citep{oord2018representation} is a self-supervised loss to learn robust representations from unlabeled speech data. 
The CPC loss is designed to maximize the probability of a future sample given a contextual representation generated from the current speech sequence. 
The \emph{CTC loss} is defined as the negative log-likelihood of the model parameter given the input sequence and the label sequence.

For all methods including the baselines, we use the step sizes $\alpha_{t,1}=5\times10^{-4}$ for training the backbone conformer parameters and $\alpha_{t,2} =5\times10^{-5}$ for training the classification head parameters.

\begin{table*}[ht]
\caption{Summary of average run time in seconds (s) or minutes (m) and number of iterations or epochs of different methods on different datasets.}
\label{tab:complexity}
\small
\fontsize{7}{8}\selectfont
\centering 
\begin{tabular}{c |c| ccccc }
\toprule 
  Datasets &Metrics &LS & PMTL & EPO & FERERO & FOOPS \\
\midrule
Synthetic, Figures 3(a-c)  &Iterations & 100 & 100 & 60 & 10 &100\\ 
&Per-iteration run time
&  3.50E-4s & 7.67E-4s &  4.93E-3s &  7.50E-4s & 7.61E-4s \\ 
&Total run time 
& 0.035s & 0.0767s & 0.296s & 0.0075s & 0.0761s\\ 
\hline
Synthetic, Figures 3(d-f)  &Iterations &100 & 200 & 80 & 200 &100 \\ 
&Per-iteration run time
& 3.10E-4s & 7.65E-4s &  4.93E-3s &  7.30E-4s 
& 7.43E-4s \\ 
&Total run time 
& 0.031s & 0.153s & 0.394s & 0.146s & 0.074s \\ 
\hline
Multi-MNIST/Fashion/F+M
&Epochs &100 & 100 & 100 & 100 & 100 \\   
&Per-epoch run time 
& 3.54s & 11.88s & 9.66s & 7.02s & 6.89s \\  
&Total run time 
& 5.9m & 19.8m & 16.1m & 11.7m &11.5m \\ 
\bottomrule
\end{tabular}
\end{table*}

\end{document}